\documentclass[12pt, reqno]{amsart}
\usepackage{amsmath}
\usepackage{amsthm}
\usepackage{amssymb}
\usepackage{amsrefs}
\usepackage{mathrsfs}
\usepackage[usenames]{color}
\usepackage[utf8]{inputenc}
\usepackage{latexsym}
\usepackage{yhmath}
\usepackage{graphicx}
\usepackage{ifthen}

\newtheorem{thm}{}[section]
\newtheorem{theorem}[thm]{Theorem}
\newtheorem{corollary}[thm]{Corollary}
\newtheorem{lemma}[thm]{Lemma}
\newtheorem{proposition}[thm]{Proposition}

\theoremstyle{remark}
\newtheorem{remark}[thm]{Remark}
\newtheorem{problem}[thm]{Problem}
\newtheorem{example}[thm]{Example}
\newtheorem{definition}[thm]{Definition}

\numberwithin{equation}{section}
\allowdisplaybreaks

\newcommand{\Env}[2][]{%
\ifthenelse{ \equal{#1}{} }  
{\ensuremath{#2_{\mathsf{c}}}}  
{\ensuremath{#2_{\mathsf{c},#1}}}
}
\newcommand{\env}[2]{\ensuremath{{\widehat{#1}}^{\,#2}}}
\newcommand{\Hb}{\ensuremath{\mathbb{H}}}
\newcommand{\ah}{\ensuremath{\mathbf{a}}}

\newcommand{\VV}{\ensuremath{\mathbb{V}}}

\newcommand{\yy}{\ensuremath{\mathbf{y}}}
\newcommand{\LL}{\ensuremath{\mathcal{L}}}
\newcommand{\hhh}{\ensuremath{\mathbf{h}}}

\newcommand{\uu}{\ensuremath{\mathbf{u}}}
\newcommand{\WW}{\ensuremath{\mathcal{W}}}

\newcommand{\tl}{\ensuremath{\mathbf{f}}}
\newcommand{\bv}{\ensuremath{\mathbf{b}}}
\newcommand{\KT}{\ensuremath{\mathrm{KT}}}
\newcommand{\ee}{\ensuremath{\mathbf{e}}}
\newcommand{\ttt}{\ensuremath{\mathbf{t}}}
\newcommand{\sss}{\ensuremath{\mathbf{s}}}
\newcommand{\vv}{\ensuremath{\mathbf{v}}}
\newcommand{\ww}{\ensuremath{\mathbf{w}}}
\newcommand{\Ind}{\ensuremath{\mathbf{1}}}
\newcommand{\Fou}{\ensuremath{\mathcal{F}}}
\newcommand{\II}{\ensuremath{\mathcal{I}}}
\newcommand{\OO}{\ensuremath{\mathcal{O}}}
\newcommand{\EE}{\ensuremath{\mathcal{E}}}
\newcommand{\AAA}{\ensuremath{\mathcal{A}}}
\newcommand{\BB}{\ensuremath{\mathcal{B}}}
\newcommand{\DDD}{\ensuremath{\mathcal{D}}}
\newcommand{\RRR}{\ensuremath{\mathcal{R}}}
\newcommand{\GG}{\ensuremath{\mathcal{G}}}
\newcommand{\HH}{\ensuremath{\mathcal{H}}}
\newcommand{\JJ}{\ensuremath{\mathcal{J}}}
\newcommand{\TT}{\ensuremath{\mathcal{T}}}
\newcommand{\UU}{\ensuremath{\mathcal{U}}}
\newcommand{\Id}{\ensuremath{\mathrm{Id}}}
\newcommand{\PP}{\ensuremath{\mathcal{P}}}
\newcommand{\FFF}{\ensuremath{\mathcal{F}}}
\newcommand{\FF}{\ensuremath{\mathbb{F}}}
\newcommand{\CC}{\ensuremath{\mathbb{C}}}
\newcommand{\RR}{\ensuremath{\mathbb{R}}}
\newcommand{\NN}{\ensuremath{\mathbb{N}}}
\newcommand{\Sym}{\ensuremath{\mathbb{S}}}
\newcommand{\ZZ}{\ensuremath{\mathbb{Z}}}
\newcommand{\YY}{\ensuremath{\mathbb{Y}}}
\newcommand{\XX}{\ensuremath{\mathbb{X}}}
\newcommand{\xx}{\ensuremath{\mathbf{x}}}

\DeclareMathOperator{\supp}{supp}
\DeclareMathOperator{\sgn}{sign}

\author[F. Albiac]{Fernando Albiac}
\address{Fernando Albiac\\
InaMat (Institute for Advanced Materials and Mathematics) and Department of Mathematics, Statistics, and Computer Sciences\\ 
Universidad P\'ublica de Navarra\\
Campus de Arrosad\'{i}a\\
Pamplona\\ 
31006 Spain\\
Tel.: +34-948-169553\\
Fax: +34-948-166057
}
\email{fernando.albiac@unavarra.es}

\author[J. L. Ansorena]{Jos\'e L. Ansorena}
\address{Jos\'e L. Ansorena\\
Department of Mathematics and Computer Sciences\\
Universidad de La Rioja\\ 
Logro\~no\\
26004 Spain\\
Tel.: +34-941-299464 \\
Fax: +34-941-299460
}
\email{joseluis.ansorena@unirioja.es}

\author[P.\ M. Bern\'a]{Pablo M. Bern\'a}
\address{Pablo M. Bern\'a\\
Mathematics Department, Universidad Aut\'onoma de Madrid\\ Madrid, 28040 Spain.\\
Tel.: +34-91-4973770 \\
Fax: +34-91-4974889}
\email{pablo.berna@uam.es}

\author[P. Wojtaszczyk]{Przemys\l{}aw Wojtaszczyk}
\address{Przemys\l{}aw Wojtaszczyk\\
Institute of Mathematics Polish Academy of Sciences\\
00-656 War\-sza\-wa\\
ul. \'Sniadeckich 8\\
Poland 
Tel.: +48 22 5228100\\}
\email{wojtaszczyk@impan.pl}



\begin{document}

\title[Greedy-like bases in quasi-Banach spaces]{Greedy approximation for biorthogonal systems in quasi-Banach spaces}

\begin{abstract} The general problem addressed in this work is the development of a systematic study of the thresholding greedy algorithm for general biorthogonal systems (also known as Markushevich bases) in quasi-Banach spaces from a functional-analytic point of view. We revisit the concepts of greedy, quasi-greedy, and almost greedy bases in this comprehensive framework and provide the (nontrivial) extensions of the corresponding characterizations of those types of bases. As a by-product of our work, new  properties arise, and the relations amongst them are carefully discussed. 
\end{abstract}

\thanks{F. Albiac acknowledges the support of the Spanish Ministry for Economy and Competitivity Grant  MTM2016-76808-P for \emph{Operators, lattices, and structure of Banach spaces}. J.L. Ansorena acknowledges the support of the Spanish Ministry for Economy and Competitivity Grant MTM2014-53009-P for \emph{An\'alisis Vectorial, Multilineal y Aplicaciones}. The research of P. M. Bern\'a was partially supported by the grants MTM-2016-76566-P (MINECO, Spain) and 20906/PI/18 from Fundaci\'on S\'eneca (Regi\'on de Murcia, Spain). P. M. Bern\'a was also supported by a Ph.\ D.\ Fellowship from the program \emph{Ayudas para contratos predoctorales para la formaci\'on de doctores 2017} (MINECO, Spain). P. Wojtaszczyk was partially supported by National Science Centre, Poland grant UMO-2016/21/B/ST1/00241. The first and third named authors would like to thank the Isaac Newton Institute for Mathematical Sciences, Cambridge, for support and hospitality during the program \emph{Approximation, Sampling and Compression in Data Science}, where some work on this paper was undertaken. This work was supportedby EPSRC grant no EP/K032208/1.}

\maketitle

\section*{Introduction}
\noindent The subject of finding estimates for the rate of approximation of a function by means of essentially nonlinear algorithms with respect to biorthogonal systems and, in particular, the greedy approximation algorithm using bases, has attracted much attention for the last twenty years, on the one hand from researchers interested in the applied nature of non-linear approximation and, on the other hand from researchers with a more classical Banach space theory background. Although the basic idea behind the concept of a greedy basis had been around for some time, the formal development of a theory of greedy bases was initiated in 1999 by Konyagin and Temlyakov. In the seminal paper \cite{KoTe1999} they introduced greedy and quasi-greedy bases and characterized greedy bases in terms of the unconditionality and the democracy of the bases. The theoretical simplicity of the greedy algorithm became a model for a procedure widely used in numerical applications and the subject was developed quite rapidly from the point of view of approximation theory.

The convergence of the greedy algorithm also raises many interesting questions in functional analysis. The idea of studying greedy bases and related greedy algorithms from a more abstract point of view seems to have originated with the work of Dilworth, Kutzarova, and Temlyakov \cite{DKT2002}, and the work of Wojtaszczyk \cite{Wo2000}, who characterized quasi-greedy bases for biorthogonal systems in the setting of quasi-Banach spaces. From there, the theory of greedy bases and its derivates evolved very fast as many fundamental results were discovered and new ramifications branched out; but these advances were achieved solely for Schauder bases in Banach spaces. 

The neglect in the study of greedy-like bases in the setting of non-locally convex spaces is easily understood. Even when they are complete and metrizable, working with quasi-Banach spaces requires doing without two of the most powerful tools in functional analysis: the Hanh-Banach theorem (and the duality techniques that rely on it), and Bochner integration (see the drawbacks of developing a satisfactory integration theory for quasi-Banach spaces in \cite{AA2013}). This difficulty in even making the simplest initial steps has led some to regard quasi-Banach spaces as too challenging and consequently they have been assigned a secondary role in the theory.   

However, greedy-like bases arise naturally in quasi-Banach spaces, for instance in the study of wavelets 
(see, e.g. \cites{Triebel2008, IzukiSawano2009, GSU2019}), and the unavailability of tools on the subject leaves the authors with two alternatives. The easiest one is to assume the validity of the results they need without bothering to do the proofs. For instance, in \cite{IzukiSawano2009}*{Proposition 8} it is stated that greedy bases in quasi-Banach spaces are characterized by the properties of being unconditional and democratic, but no further comment or proof is provided. This result is used in the same paper in \cite{IzukiSawano2009}*{Theorem 16} in a crucial way, thus the reader is left to believe that the proof is trivial, without being made aware of the idiosyncrasies involved. On the opposite end, in \cite{Triebel2008}*{Theorem 6.51}, Triebel, in  the same situation and in need of the very same result, uses caution before relying on something that is not substantiated by a reference or a proof. Then, he settles for a weaker theorem by adding, in hindsight, unnecessary hypotheses for the result to hold. 

Taking into account that more and more analysts find that quasi-Banach spaces have uses in their research, the task to learn about greedy bases and its derivates in this framework seems to be urgent and important. Our goal in this paper is to fill this gap in the theory and encourage further research in this direction. Needless to say, since Banach spaces are a special type of quasi-Banach spaces, proving new results in $p$-Banach spaces for $0<p<1$ often provides an alternative proof even for the limit case $p=1$. Hence, quasi-Banach spaces help us appreciate better and also shed new light on regular Banach spaces. 

The step from Schauder bases to Markushevich bases also deserves a word. Most systems arising in mathematical analysis are first and foremost biorthogonal systems, independently of whether one is able to show that they are Schauder bases or not. Take for example the trigonometric system in $L_{1}$, which is a biorthogonal system and yet it is not a Schauder basis. As far as the greedy algorithm is concerned, there are authors who work in the frame of biorthogonal systems whereas other authors prefer to remain within the safer conceptual framework of Schauder bases. In fact, in the existing literature one finds fundamental results in the theory, such as the characterization of almost greedy bases \cite{DKKT2003}, which were obtained for Schauder bases, whereas others, such as the aforementioned characterization of quasi-greedy bases, which were obtained for biorthogonal systems. This lack of consensus might (and in fact, does) create some confusion about the validity of the theorems that were obtained for Schauder bases, in the more general setting of biorthogonal systems. By writing all our results under the unifying approach of biorthogonal systems in quasi-Banach spaces, it is our hope in this article to contribute to clarify this incertitude.   

Let us next outline the contents of this paper. Section~\ref{PrelimSec} has a preliminary character. More specifically, \S\ref{Sec2} may work as a ``first-aid kit'' to get around when trying to prove theorems in a space that lacks local convexity. Indeed, local convexity plays, more or less implicitly, a crucial role in the main results that have been obtained on the subject so far. Here we will gather together some instrumental lemmas that will help us proceed when local convexity is dropped.  \S\ref{Sec3}, \S\ref{LinopSec}, and \S\ref{DualSec} contain the groundwork on bases in the context of quasi-Banach spaces. One of our  focuses of attention goes to making sure that the two facets of unconditionality, suppression unconditionality and lattice unconditionality, are equivalent notions in quasi-Banach spaces too.  

A property that has proved relevant in the study of greedy-like bases is the unconditionality for constant coefficients. Section~\ref{Sec4} is devoted to its study, emphasizing the peculiarities that arise in the lack of local convexity. Like for Banach spaces, we will see that working with the greedy algorithm naturally leads to consider variations of the property of unconditionality for constant coefficients, namely the lower unconditionality for constant coefficients (LUCC for short) and the suppression unconditonality for constant coefficients (SUCC for short).   

Section~\ref{Sec5} is dedicated to  quasi-greedy bases. Although this type of bases was already considered for quasi-Banach spaces in \cite{Wo2000}, here we make headway in the theory and show that a quasi-greedy basis is strong Markushevich, a fact that seems to be unknown to the specialists on the subject.  It also contains a key lemma stating that quasi-greedy bases are LUCC. The proof of the corresponding result for Banach spaces relies heavily on the local convexity of the space, hence the labor to make it work for quasi-Banach spaces entailed cooking up an entirely different proof. In the last part of the section we study a non-linear operator called the truncation operator, and analyse the relation between its boundedness and the convergence of the greedy algorithm.  

Section~\ref{Sec6} deals with different variations on the property of democracy of a basis. We determine the relations between democracy, superdemocracy, bidemocracy, and symmetry for largest coefficients (also known as Property (A)) taking into account the specific features of the underlying quasi-Banach space.

In Section~\ref{Sec7} we take advantage of the results obtained in Section~\ref{Sec5} to provide several (nontrivial) characterizations of almost greedy bases in quasi-Banach spaces, including the extension to quasi-Banach spaces of the well-known characterization of almost greedy bases by Dilworth et al.\ \cite{DKKT2003} as those bases that are simultaneously quasi-greedy and democratic.

In Section~\ref{Sec8} we substantiate the celebrated (and by now classical) characterization by Konyagin and Temlyakov \cite{KoTe1999} of greedy bases for general biorthogonal systems in quasi-Banach spaces in terms of unconditionality and democracy. With the equivalence between suppression unconditionality and lattice unconditionality at hand, it must be conceded that the proof of this fundamental fact is essentially the same as for Banach spaces. Here, the novelty has to be judged in the achievement of optimal constants by means of a combination of classical methods with other specific techniques from quasi-Banach spaces that apply to Banach spaces and permits to improve some of the estimates known to date.  

In Section~\ref{Sec9} we transfer to general biorthogonal systems in quasi-Banach spaces the work initiated by  Dilworth et al.\ in \cite{DKK2003} on the comparison between the ``best greedy error'' of the greedy algorithm with its ``best almost-greedy error.'' In order to carry out this study we previously analyse the different democracy functions associated to the basis.

In Section~\ref{Sec10} we strengthen the techniques from \cite{AA2015} (see also \cites{GHdN2012, Wo2000}) on linear embeddings of spaces in connection with the greedy algorithm. We show that, even in the absence of local convexity, quasi-Banach spaces with an almost greedy Markushevich basis can be sandwiched between two suitable sequence Lorentz spaces, which permits to shed iformation onto the nature of the basis of the space. Going further we characterize the type of bases for which such embedding is possible. 

 When dealing with a quasi-Banach space $\XX$ it is often convenient to know which is the ``smallest''  Banach space containing $\XX$ or, more generally, given $0<q\le 1$, the smallest $q$-Banach space containing $\XX$, known as the $q$-Banach envelope of $\XX$. In Section~\ref{Sec13} we discuss how certain properties of bases related to the greedy algorithm  transfer to  envelopes.

In Section~\ref{Sec11} we contextualize and illustrate with a selection of nontrivial examples all the properties that we discuss in the article. To mention a few, we include examples of bases that are SUCC and LUCC but not quasi-greedy. We also construct new examples (valid both for locally and non-locally convex spaces) of conditional almost greedy bases, and examples of superdemocratic bases that fulfil neither the LUCC condition nor Property (A). The greedy algorithm is thoroughly examined for Besov spaces and Triebel-Lizorkin spaces. 

In Section~\ref{sec:renorming} we investigate a fundamental aspect of the theory that has a long trajectory. The problem of renorming a Banach space in order to improve the greedy-like properties of bases was initiated in \cite{AW2006} and continued years later in \cites{DOSZ2011, AAW2018}. Here we see that for quasi-Banach spaces the situation is entirely different. In fact, we can always define a renorming (we should say ``re-quasi-norming'' to be precise) of the space so that a greedy basis (respectively, quasi-greedy or almost greedy) is $1$-greedy (respectively, $1$-quasi-greedy or $1$-almost greedy). 

Section~\ref{Sec:Problems}, which concludes this article, contains a relatively long list of open problems to encourage further research on the subject. Some of this problems are of interest also outside the framework of the present study. 

\section{Preliminaries on bases and quasi-Banach spaces}\label{PrelimSec}

\subsection{General Notation}
 Through this paper we use standard facts and notation from Banach spaces and approximation theory (see \cite{AlbiacKalton2016}). For the necessary background in the general theory of quasi-Banach spaces we refer the reader to \cite{KPR1984}. Next we record the notation that is  most heavily used. 

We write $\FF$ for the real or complex scalar field. As is customary, we put $\delta_{k,n}=1$ if $k=n$ and $\delta_{k,n}=0$ otherwise. The unit vector system of $\FF^\NN$ will be denoted by $\BB_e=(\ee_n)_{n=1}^\infty$, where $\ee_n=(\delta_{k,n})_{k=1}^\infty$  Also, $\langle x_j \colon j\in J\rangle$ stands for the linear span of a family $(x_j)_{j\in J}$ in a vector space, $[ x_j \colon j\in J]$ denotes the closed linear span of a family $(x_j)_{j\in J}$ in a quasi-Banach space, and $[ x_j \colon j\in J]_{w^*}$ the $w^*$-closed linear span of $(x_j)_{j\in J}$ in a dual Banach space. We set  $c_{00}=\langle \ee_n \colon n\in\NN\rangle$.

A sign will be a scalar of modulus one, and $\sgn(\cdot)$ will denote the sign function, i.e., $\sgn(0)=1$ and $\sgn(a)=a/|a|$ if $a\in\FF\setminus\{0\}$.

The symbol $\alpha_j\lesssim \beta_j$ for $j\in J$ means that there is a positive constant $C<\infty$ such that the families of non-negative real numbers $(\alpha_j)_{j\in J}$ and $(\beta_j)_{j\in J}$ are related by the inequality $\alpha_j\le C\beta_j$ for all $j\in J$. If $\alpha_j\lesssim \beta_j$ and $\beta_j\lesssim \alpha_j$ for $j\in J$ we say $(\alpha_j)_{j\in J}$ are $(\beta_j)_{j\in J}$ are equivalent, and we write $\alpha_j\approx \beta_j$ for $j\in J$. 

We write $\XX\oplus\YY$ for the Cartesian product of the quasi-Banach spaces $\XX$ and $\YY$ endowed with the quasi-norm 
\[
\Vert (x,y)\Vert=\max\{ \Vert x\Vert, \Vert y\Vert\}, \quad x\in\XX,\ y\in\YY.  
\]
Given a sequence $(\XX_n)_{n=1}^\infty$ of quasi-Banach spaces and $0\le p\le\infty$, $(\oplus_{k=1}^\infty \XX_k)_p$ denotes the quasi-Banach space consisting of all sequences $(f_k)_{k=1}^\infty\in \Pi_{k=1}^\infty \XX_k$ such that $(\Vert f_k\Vert)_{k=1}^\infty\in\ell_p$ ($c_0$ is $p=0$). If $\XX_k=\XX$ for all $k\in\NN$ we will put $\ell_p(\XX):=(\oplus_{k=1}^\infty \XX_k)_p$ ($c_0(\XX)$ if $p=0$).

We shall denote by $\LL(\XX,\YY)$ the quasi-Banach space of all bounded linear maps from a quasi-Banach space $\XX$ into a quasi-Banach space $\YY$. 
We write $\XX\simeq \YY$ if the quasi-Banach spaces $\XX$ and $\YY$ are isomorphic. 

$B_{\XX}$ will  denote the closed unit ball of a quasi-Banach space $\XX$. We say that a linear map $T$ from a dense subspace $V$ of $\XX$ into $\XX$ is \emph{well-defined} on $\XX$ if it is continuous with respect to the topology on $\XX$. Of course, $T$ is well-defined on $\XX$ if and only if $T$ has a (unique) continuous extension to a map, also denoted by $T$, from $\XX$ to $\XX$. 

An \emph{operator} on $\XX$ will be a (possibly non-linear nor bounded) map $T\colon \XX\to \XX$ such that $T(tf)=tT(f)$ for all $t\in\FF$ and $f\in\XX$.

Other more specific notation will be introduced in context when needed.

\subsection{Convexity-related properties of quasi-Banach spaces}\label{Sec2}
\noindent
A \emph{quasi-norm} on a vector space $\XX$ is a map $\Vert \cdot\Vert\colon \XX\to [0, \infty)$ satisfying
\begin{itemize}
\item[(q1)] $\Vert f\Vert >0$ for all $x\not=0$,

\item[(q2)] $\Vert t f\Vert=|t| \Vert f\Vert$ for all $t\in\RR$ and all $x\in \XX$; and

\item[(q3)] there is a constant $\kappa\ge 1$ so that for all $f$ and $g\in \XX$ we have 
\begin{equation*}\Vert f +g\Vert \le \kappa(\Vert f\Vert +\Vert g\Vert).\end{equation*}
\end{itemize}
The collection of sets of the form
\[
\left\{x\in \XX\colon \Vert x\Vert<\frac1n\right\},  \quad n\in \mathbb N,
\] 
are a base  of neighbourhoods of zero, so that the quasi-norm $\Vert\cdot\Vert$  induces a  metrizable linear topology on  $\XX$. If $\XX$ is complete for this topology we say that $(\XX,\Vert \cdot\Vert)$ is a \emph{quasi-Banach space}. 

Another quasi-norm $ \Vert\cdot\Vert_0$ on the same vector space $\XX$ is said to be a \emph{renorming} of $\Vert \cdot\Vert$ if $\Vert \cdot\Vert_0$ and $\Vert \cdot\Vert$ induce the same topology. It is well known that $\Vert \cdot\Vert_0$ is a renorming of $\Vert \cdot\Vert$ if and only if the functions $\Vert \cdot\Vert_0$ and $\Vert \cdot\Vert$ are equivalent.

 Given $0<p\le 1$, a \emph{$p$-norm} is a map $\Vert \cdot\Vert\colon \XX\to [0, \infty)$ satisfying (q1), (q2) and
\begin{itemize}
\item[(q4)] $\Vert f+g\Vert^{p}\le \Vert f\Vert^{p} +\Vert g\Vert^p$
for all $f,g\in \XX$.
\end{itemize}
Of course, (q4) implies (q3) with $\kappa=2^{1/p-1}$. A quasi-Banach space whose associated quasi-norm is a $p$-norm will be called a \emph{$p$-Banach space}. The Aoki-Rolewicz's Theorem \cites{Aoki,Rolewicz} states that any quasi-Banach space $\XX$ is \emph{$p$-convex} for some $0<p\le 1$, i.e., there is a constant $C$  such that
\[
\left\Vert \sum_{j=1}^n f_j\right\Vert \le C \left( \sum_{j=1}^n \Vert f_j\Vert^p\right)^{1/p}, \quad n\in\NN, \, f_j\in\XX.
\]
This way, $\XX$ becomes $p$-Banach under a suitable renorming. The term \emph{locally convex} designates a $1$-convex quasi-Banach space.

Let us recall the following easy and well-known result. 
\begin{proposition}\label{prop:lpgalbed} 
A quasi-Banach space $(\XX,\Vert \cdot\Vert)$ is $p$-Banach for some $0<p\le 1$ if and only if 
\begin{equation*}
\left\Vert \sum_{j\in J} a_j f_j \right\Vert \le \left(\sum_{j\in J} |a_j|^p\right)^{1/p} \sup_{j\in J} \Vert f_j\Vert, \quad J \text{ finite, }a_j\in\FF, \, f_j\in\XX.
\end{equation*}
\end{proposition}
Despite the fact that a quasi-norm is not necessarily a continuous map, the reverse triangle law 
\[
\left| \Vert f\Vert^p -\Vert g\Vert^p\right|\le \Vert f-g\Vert^p, \quad f,g\in\XX
\]
yields that a $p$-norm is a continuous map from $\XX$ onto $[0,\infty)$. Hence, the Aoki-Rolewicz Theorem yields, in particular, that any quasi-Banach space can be equipped with an equivalent continuous quasi-norm. 

Theorem~\ref{thm:convexity} will play the role of a substitute of the Bochner integral in the lack of local convexity. It will be heavily used throughout and it will allow us to extend to quasi-Banach spaces  \emph{certain} results that are valid in the Banach space setting. However, we would like to stress that this result is not a fix for all the obstructions we meet when transferring the greedy algorithm theory into quasi-Banach spaces.   

Let us introduce two geometrical constants that are closely related to the convexity properties of $p$-Banach spaces. Given $0<p\le 1$ we put
\begin{equation}\label{eq:fieldconstant1}
A_p=\frac{1}{(2^p-1)^{1/p}}
\end{equation}
and
\begin{equation}\label{eq:fieldconstant2}
B_p=\begin{cases} 
2^{1/p} A_p &\text{ if }\FF=\RR,\\
4^{1/p} A_p &\text{ if } \FF=\CC.
\end{cases}
\end{equation} 

\begin{theorem}\label{thm:convexity} Suppose $\XX$ is a $p$-Banach space for some $0<p\le 1$. Given any two families of functions $(g_j)_{j\in J}$,  $(h_j)_{j\in J}\in\XX^J$ with $J$ finite, we have
\[
\left\Vert \sum_{j\in J} (1-b_j)g_j + b_j h_j \right\Vert \le A_p \sup_{A\subseteq J} \left\Vert \sum_{j\in J\setminus A} g_j+ \sum_{j\in A} h_j\right\Vert,
\]
for all $(b_j)_{j=J}\in[0,1]^J$.
\end{theorem}

\begin{proof}
By approximation we can assume that $0\le b_j<1$ for all $j\in J$. Let
\[
\AAA=\{ (\beta_k)_{k=1}^\infty \in\{0,1\}^\NN \colon |\{k\in\NN \colon \beta_k=0\}|=\infty \}.
\]
Note that for each $0\le b<1$ there is a unique $(\beta_k)_{k=1}^\infty\in\AAA$ such that 
$b=\sum_{k=1}^\infty \beta_{k}\, 2^{-k}$. Let us write
\[
b_j=\sum_{k=1}^\infty \beta_{j,k} 2^{-k}, \quad (\beta_{j,k})_{k=1}^\infty \in \AAA,
\]
and put
\[
A_k=\{j\in J \colon \beta_{j,k}=1\}.
\]
By Proposition~\ref{prop:lpgalbed},
\begin{align*}
\left\Vert \sum_{j\in J} (1-b_j)g_j + b_j h_j \right\Vert
&= \left\Vert \sum_{j\in J}\left(\sum_{k=1}^\infty (1-\beta_{j,k}) 2^{-k}g_j + \sum_{k=1}^\infty \beta_{j,k} 2^{-k} h_j\right) \right\Vert \\
&= \left\Vert \sum_{k=1}^\infty 2^{-k}\left( \sum_{j\in J} (1-\beta_{j,k}) g_j + \beta_{j,k} h_j\right) \right\Vert\\
&= \left\Vert \sum_{k=1}^\infty 2^{-k}\left( \sum_{j\in J\setminus A_{k}} g_j + \sum_{j\in A_{k}} h_j\right) \right\Vert\\
&\le \left(\sum_{k=1}^\infty 2^{-kp}\right)^{1/p} \sup_{k\in\NN}\left\Vert\sum_{j\in J\setminus A_{k}} g_j + \sum_{j\in A_{k}} h_j \right\Vert.
\end{align*}
Since $\sum_{k=1}^\infty 2^{-kp}=A_p^p$, we are done.
\end{proof}

\begin{corollary}\label{cor:convexity} Let $\XX$ be a $p$-Banach space for some $0<p\le 1$. Let
$(f_j)_{j\in J}$ be any collection of functions in $\XX$ with $J$ finite, and $g\in\XX$. Then:
\begin{itemize}
\item[(i)] For any scalars $(a_{j})_{j\in J}$ with $0\le a_j \le 1$ we have
\[
\left\Vert g+ \sum_{j\in J} a_j f_j \right\Vert \le A_p \sup \left\{ \left\Vert g+ \sum_{j\in A} f_j\right\Vert \colon A\subseteq J\right\}.
\]

\item[(ii)] For any scalars $(a_{j})_{j\in J}$ with
$|a_j|\le 1$ we have
\[
\left\Vert g+ \sum_{j\in J} a_j f_j \right\Vert \le A_p \sup\left\{ \left\Vert g+ \sum_{j\in J} \varepsilon_j f_j\right\Vert \colon 
|\varepsilon_j|=1\right\}.
\] 
\end{itemize}
\end{corollary}

\begin{proof}  Pick a point  $0\notin J$ and set  $J'=\{0\}\cup J$. In order to prove both (i) and (ii) we will  apply Theorem~\ref{thm:convexity} with suitable families  $(g_j)_{j\in J'}$,  $(h_j)_{j\in J'}$ and  $(b_j)_{j\in J'}$.
In both cases we choose $g_0=h_0=g$ and an arbitrary scalar $b_0\in[0,1]$. 

If we choose $b_j=a_j$, $g_j=0$, and $h_j=f_j$ for $j\in J$ we obtain (i).

To see (ii), for each $j\in J$ we pick $b_j\in[0,1]$  and a sign $\delta_j$ such that
\[
a_j=(1-b_j)\delta_j-b_j \delta_j.
\]
Then, we choose  $g_j=\delta_j f_j$ and  $h_j=-\delta_j f_j$ for $j\in J$.
\end{proof}

\begin{corollary}\label{cor:convexity2}Suppose $\XX$ is a $p$-Banach space for some $0<p\le 1$. 
Then for any family $(f_j)_{j\in J}$ in $\XX$ with $J$ finite we have
\[
\left\Vert \sum_{j\in J} a_j f_j \right\Vert \le B_p\sup_{A\subseteq J} \left\Vert \sum_{j\in A} f_j\right\Vert,
\]
whenever $(a_j)_{j\in J}$ are scalars with $|a_j|\le 1$ for all $j\in J$.
\end{corollary}

\begin{proof}If $\FF=\RR$, we have $a_j=a_j^+-a_j^-$, while if $\FF=\CC$ we have $a_j=\Re(a_j)^+-\Re(a_j)^-+i\Im(a_j)^+-i\Im(a_j)^-$. Applying Corollary~\ref{cor:convexity}~(i) (with $g=0$) and using the $p$-subadditivity of the quasi-norm gives the desired result.
\end{proof}

\subsection{Bases in quasi-Banach spaces}\label{Sec3}
Throughout this paper a \emph{basis} in a quasi-Banach space $\XX$ will be a \emph{complete Markushevich basis}, i.e., a sequence $\BB=(\xx_n)_{n=1}^\infty$ such that
\begin{itemize}
\item[(i)] $[\xx_n \colon n\in\NN]=\XX$, and
\item[(ii)] there is a (unique) sequence $(\xx_n^*)_{n=1}^\infty$ of functionals, called coordinate functionals (also biorthogonal functionals), such that $\xx_n^*(\xx_k)=\delta_{k,n}$ for all $k,n\in\NN$.
\end{itemize}

The \emph{support} of a vector $f\in\XX$ with respect to a basis $\BB$ is the set 
\[ 
\supp(f)=\{ n\in\NN \colon \xx_n^*(f)\not=0\}. 
\]

A \emph{basic sequence} will be a basis of its closed linear span. 

A basis $\BB$ is said to be $M$-\emph{bounded} if 
\[
\sup_n \Vert \xx_n\Vert \, \Vert \xx_n^*\Vert <\infty,
\]
and is said to be \emph{semi-normalized} if 
\[
0<\inf_n \Vert \xx_n \Vert \le \sup_n \Vert \xx_n\Vert <\infty.
\]
\begin{lemma}\label{lem:bases1}A basis $(\xx_n)_{n=1}^\infty$ of a quasi-Banach space $\XX$ is semi-normalized and $M$-bounded if and only if 
$
\sup_n \max\lbrace \Vert \xx_n \Vert, \Vert \xx_n^*\Vert \rbrace<\infty.
$
\end{lemma}
\begin{proof}Assume that $C_1^{-1}\Vert \xx_n \Vert \le \sup_n \Vert \xx_n\Vert \le C_2$ and $\Vert \xx_n\Vert \Vert \xx_n^*\Vert \le C_3$ for all $n\in\NN$. Then $\Vert \xx_n^*\Vert \le C_1 C_3$ for all $n\in\NN$. Conversely, assume that $\Vert \xx_n\Vert \le C_4$ and $\Vert \xx_n^*\Vert \le C_5$ for all $n\in\NN$. Then, $\Vert \xx_n\Vert \Vert \xx_n^*\Vert \le C_4 C_5$. Since $1=\xx_n^*(\xx_n )\le \Vert \xx_n^*\Vert \Vert \xx_n\Vert$, we have $\Vert \xx_n \Vert \ge C_5^{-1}$. 
\end{proof}

Note that an \emph{Auerbach basis}, i.e., a basis $(\xx_n)_{n=1}^\infty$ with $\Vert \xx_n\Vert = \Vert \xx_n^*\Vert=1$ for all $n\in\NN$, is a particular case of semi-normalized $M$-bounded basis. 

A very natural condition to expect from a basis $\BB=(\xx_n)_{n=1}^\infty$ is that every $f\in \XX$ be univocally determined by its coefficient sequence, i.e., the linear operator
\begin{equation}\label{eq:Fourier}
\FFF\colon \XX \to \FF^\NN , \quad f\mapsto (\xx_n^*(f))_{n=1}^\infty,
\end{equation} 
called the \emph{coefficient transform}, is one-to-one. However, for reasons that will be understood in hindsight, here we will not impose \emph{a priori} this requirement to the basis. Note that the injectivity of $\Fou$ is equivalent to demanding that
\begin{equation}\label{eq:total}
[ \xx^*_{n} \colon n\in \NN]_{w^*}=\XX^*.
\end{equation}

Using the standard terminology in biorthogonal systems,  we say that a subset $V\subseteq\XX^*$ is \emph{total} if $[V]_{w^*}=\XX^*$ and that a basis $\BB$ of $\XX$ is total if its biorthogonal sequence $\BB^*=(\xx_n^*)_{n=1}^\infty$ is total, i.e.,  \eqref{eq:total} holds. If the basis $\BB$ is total we are allowed to use the ``formal'' expression $f=\sum_{n=1}^\infty a_n\, \xx_n\in\XX$ to mean that $f$ is a vector in $\XX$ with $\xx_n^*(f)=a_n$ for every $n\in\NN$.
In any case, if we consider the space
\[
\YY=\YY[\BB,\XX]=\left\{ (a_n)_{n=1}^\infty \colon \sum_{n=1}^\infty a_n \xx_n\ \text{ converges in } \XX \right\},
\]
and the linear map
\begin{equation}\label{eq:InvFourier}
\II\colon \YY \subseteq \FF^\NN \to \XX, \quad (a_n)_{n=1}^\infty\mapsto \sum_{n=1}^\infty a_n \xx_n.
\end{equation}
we have $\FFF\circ \II=\Id_\YY$.

Two sequences $\BB=(\xx_n)_{n=1}^\infty$ and $\BB'=(\xx_n')_{n=1}^\infty$ in respective quasi-Banach spaces $\XX$ and $\XX'$ are said to be \emph{equivalent} if there is a linear isomorphism $T\colon[\BB]\to[\BB']$ such that $T(\xx_n)=\xx_n'$ for $n\in\NN$, i.e., there is a constant $C$ so that 
$$
C^{-1}\left\Vert \sum_{n=1}^{N} a_{i}\xx_{n}\right\Vert\le \left\Vert \sum_{n=1}^{N}
a_{n}\xx_{n}^{\prime}\right\Vert\le C\left\Vert
\sum_{n=1}^{N} a_{n}\xx_{n}\right\Vert,$$
 for any choice of scalars $(a_{n})_{n=1}^{N}$ and every $N\in\NN$. Of course, if $\BB$ is a basic sequence and $\BB'$ is equivalent to $\BB$, then $\BB'$ also is a basic sequence.

\subsection{Linear operators associated to bases}\label{LinopSec}
Suppose $\BB=(\xx_n)_{n=1}^\infty$ is basis for a quasi-Banach space $\XX$. For a fixed sequence $\gamma=(c_n)_{n=1}^\infty\in\FF^\NN$, let us consider the map
\[
M_\gamma=M_\gamma[\BB,\XX]\colon \langle \xx_n \colon n\in\NN\rangle \to \XX,
\quad \sum_{n=1}^\infty a_n\, \xx_n \mapsto \sum_{n=1}^\infty c_n\, a_n \, \xx_n.
\]
The basis $\BB$ is said to be \emph{lattice unconditional} if $M_\gamma$ is well-defined on $\XX$ for every $\gamma\in\ell_\infty$ and
\begin{equation}\label{eq:lu}
K_{u}=K_{u}[\BB,\XX]:=\sup_{\Vert \gamma\Vert_\infty\le 1} \Vert M_\gamma\Vert <\infty.
\end{equation}
If $\BB$ is lattice unconditional, $K_{u}$ is called the lattice unconditional constant of the basis. Notice that $K_u=1$ with respect to the renorming
\[
\Vert f\Vert_0:=\sup_{\Vert \gamma\Vert_\infty\le 1}\Vert M_\gamma(f)\Vert, \quad f\in\XX.
\]

Now, given $A\subseteq \NN$, we define the \emph{coordinate projection} onto $A$ as
\[ 
S_A:=M_{\gamma_A}, 
\] 
where $\gamma_A=(c_n)_{n=1}^\infty$ is the sequence given by $c_n=1$ for $n\in A$ and $c_n=0$ otherwise. The basis $\BB$ is said to be \emph{suppression unconditional} if $S_A$ is well-defined on $\XX$ for every $A\subseteq\NN$ and
\begin{equation}\label{eq:su}
K_{su}=K_{su}[\BB,\XX]=\sup_{A\subseteq\NN} \Vert S_A\Vert <\infty.
\end{equation}
If $\BB$ is suppression unconditional, $K_{su}$ is called the suppression unconditional constant of the basis.

\begin{proposition}\label{prop:EstSupUnc} Suppose $\BB$  is a suppression unconditional basis in a $p$-Banach space. Then for every $\gamma\in[0,1]^\NN$,
\[\Vert M_\gamma\Vert \le A_p K_{su}.\]
\end{proposition}
\begin{proof} By Corollary~\ref{cor:convexity}~(i) we have $\Vert M_\gamma(f)\Vert \le A_p K_{su}\Vert f \Vert$
for all  $f\in\langle \xx_n \colon n \in \NN\rangle$. Hence, 
 $M_\gamma$ is well-defined and $\Vert M_\gamma\Vert \le A_p K_{su}$.
\end{proof}

\begin{proposition}\label{prop:SupvsLat} A basis $\BB$ for a quasi-Banach space $\XX$
 is suppression unconditional if and only if $\BB$ is lattice unconditional. Moreover $ 
K_{su}\le K_u$.
If $\XX$ is $p$-Banach then $K_u \le B_p K_{su}$.
\end{proposition}
\begin{proof}Use Corollary~\ref{cor:convexity2} and proceed as in the proof of Proposition~\ref{prop:EstSupUnc}.
\end{proof}

Note that for $n\in\NN$ and $f\in\XX$,
\[
S_{\{n\}}(f)=\xx_n^*(f) \, \xx_n,\]  and so 
\[\Vert S_{\{n\}}\Vert =\Vert \xx_n\Vert \Vert \xx_n^*\Vert.\] 
Therefore, $\BB$ is $M$-bounded if and only if $\sup_n \Vert S_{\{n\}}\Vert<\infty$. 

Let us write down some easy consequences of $M$-boundedness.

\begin{lemma}\label{lem:bases5}Let $\BB=(\xx_n)_{n=1}^\infty$ be a $M$-bounded basis of a quasi-Banach space $\XX$. Then:
\begin{itemize}
\item[(i)] For every $m\in\NN$, $\sup_{|A|\le m} \Vert S_A\Vert<\infty$. 
\item[(ii)] For every $f\in\XX$, every $m\in\NN$, and every $\epsilon>0$ there is $A\subseteq \NN$ finite such that $\Vert S_B(f) \Vert \le\epsilon$ whenever $A\cap B=\emptyset$ and $|B|\le m$.
\item[(iii)]If $\BB$ is semi-normalized, the linear map $\FFF$ defined in \eqref{eq:Fourier} is bounded from $\XX$ into $c_0$.
\end{itemize}
\end{lemma}

\begin{proof}Without loss of generality we assume that $\XX$ is $p$-convex for some $0<p\le 1$. 
To show (i), put $c:=\sup_n \Vert \xx_n\Vert \, \Vert \xx_n^*\Vert$ and pick $A\subseteq \NN$ with $|A|\le m$. Then,
\[
\Vert S_A\Vert^p \le \sum_{n\in A} \Vert S_{\{n\}}\Vert^p\le \sum_{n\in A} c^p\le c^p m.
\]

For (ii), let $\PP_m(\NN)$ be the set of all non-empty subsets of $\NN$ of cardinality at most $m$. Consider the linear map
\[
T_m\colon \XX\to \XX^{\PP_m(\NN)},\quad f\mapsto (S_A(f))_{A\in \PP_m(\NN)}.
\]
By (i), the map $T_m$ is bounded from $\XX$ into $\YY:=\ell_\infty(\PP_m(\NN),\XX)$. Let
$\YY_0$ be the set of all $(x_A)_{A\in \PP_m(\NN)}$ in $\YY$ so that given $\epsilon>0$ there exists $B \subseteq\NN$ finite such that $\Vert x_A\Vert \le \epsilon$ whenever $B\cap A=\emptyset$.
Since $\YY_0$ is a closed subspace of $\YY$, and $T(\xx_n)\in \YY_0$ for all $n\in\NN$, we are done.

(iii) If we identify $\NN$ with the set of singletons of $\NN$, the above argument in the case when $m=1$ yields that the operator $T\colon \XX \to \XX^\NN$ given by $T(f)=(\xx_n^*(f) \, \xx_n)_{n=1}^\infty$ is bounded from $\XX$ into $c_0(\XX)$. Consider the operator $S\colon \FF^\NN\to \XX^\NN$ given by $(a_n)_{n=1}^\infty\mapsto (a_n\, \xx_n)_{n=1}^\infty$. If $\BB$ is semi-normalized, $S$ restricts to an isomorphic embedding from $c_0$ into $c_0(\XX)$. Since $S\circ\FFF=T$ the proof is over. 
\end{proof}

A sequence $\BB=(\xx_n)_{n=1}^\infty$ in a quasi-Banach space $\XX$ is said to be a \emph{Schauder basis} if for every $f\in\XX$ there is a unique sequence of scalars $(a_n)_{n=1}^\infty$ such that the series $\sum_{n=1}^\infty a_n \, \xx_n$ converges to $f$. It is known that a Schauder basis is a basis (see e.g. \cite{AlbiacKalton2016}*{Theorem 1.1.3}). The partial-sum projections associated to a Schauder basis are the maps
\[
S_m:=S_{\{1,\dots,m\}}, \quad m\in\NN.
\] 
A basis $\BB$ is a Schauder basis if and only if 
\[
K=K[\BB,\XX]:=\sup_m \Vert S_m\Vert <\infty.
\] 
We will refer to $K[\BB,\XX]$ as the \emph{basis constant} of $\BB$ in $\XX$. Since, if $\XX$ is $p$-Banach, $\sup_m\Vert S_{\{m\}}\Vert \le 2^{1/p}K$, we see that any Schauder basis is a total $M$-bounded basis.

 A series $\sum_{n=1}^\infty f_n$ in a quasi-Banach space $\XX$ is said to be \emph{unconditionally convergent} if $\sum_{n=1}^\infty f_{\pi(n)}$ converges for any permutation $\pi$ of $\NN$. Like for Banach spaces, the concept of unconditionally convergent series can be restated in different ways which we summarize in Lemma~\ref{Lemma3.6}. The proof of this follows the same steps as in the locally convex case. 
\begin{lemma}[cf.\ \cite{AlbiacKalton2016}*{Lemma 2.4.2}]\label{Lemma3.6} Let $(f_n)_{n=1}^\infty$ be a sequence in a quasi-Banach space $\XX$. The following are equivalent:
\begin{itemize}
\item[(i)] $\sum_{n=1}^\infty f_n$ converges unconditionality.
\item[(ii)] $\sum_{n=1}^\infty f_{\pi(n)}$ converges for any increasing map $\pi\colon\NN\to\NN$.
\item[(iii)] $\sum_{n=1}^\infty a_n f_n$ converges whenever $|a_n|\le 1$.
\item[(iv)] For any $\epsilon>0$ there is $F\subseteq\NN$ finite such that whenever $G\subseteq\NN$ finite is such that
$F\cap G=\emptyset$ we have $\Vert \sum_{n\in G} f_n\Vert \le\epsilon$.
\end{itemize}
\end{lemma}

Hence, if $\sum_{n=1}^\infty f_n$ is unconditionally convergent,  there is $f\in \XX$ such that $\sum_{n=1}^\infty f_{\pi(n)}=f$ for every permutation $\pi$ of $\NN$. In this case we say that $f=\sum_{n=1}^\infty f_n$ unconditionally. 

A sequence $(\xx_n)_{n=1}^\infty$ in a quasi-Banach space $\XX$ is said to be an \emph{unconditional basis} if for every $f\in\XX$ there is a unique sequence of scalars $(a_n)_{n=1}^\infty$ such that $f=\sum_{n=1}^\infty a_n \, \xx_n$ unconditionally. Since being unconditional is formally  stronger  than being a Schauder basis, every unconditional basis is a total $M$-bounded basis. With the help of Proposition~\ref{prop:SupvsLat}, we readily realize that unconditional bases can be characterized as in the locally convex case. That is:

\begin{theorem}[cf.\ \cite{AlbiacKalton2016}*{Section 3.1}]\label{cu}Let $\BB$ be a basis for a quasi-Banach space $\XX$. The following are equivalent.
\begin{itemize}
\item[(i)] $\BB$ is unconditional.
\item[(ii)] $\BB$ is suppression unconditional.
\item[(iii)] $\BB$ is lattice unconditional.
\item[(iv)] $S_A$ is well-defined for every $A\subseteq\NN$.
\item[(v)] $M_\gamma$ is well-defined por every $\gamma\in\ell_\infty$.
\end{itemize}
\end{theorem}

\subsection{Bases and duality}\label{DualSec}
The sequence of coordinate functionals associated to a basis $\BB=(\xx_n)_{n=1}^\infty$ of a quasi-Banach space $\XX$ is a basic sequence of $\XX^*$. Indeed, if $h_\XX\colon \XX\mapsto \XX^{**}$ denotes the bidual map and $q_\BB\colon \XX^{**} \to [\BB^*]^*$ is the natural quotient map given by $f^{**}\mapsto f^{**}|_{[\BB^*]}$ then the composition map
$h_{\BB,\XX}=q_\BB\circ h_\XX\colon \XX \to [\BB^*]^*$ is given by
\begin{equation}\label{eq:bibiorthogonal}
h_{\BB,\XX}(f)(f^*)=f^*(f), \quad f\in\XX, \, f^*\in[\BB^*].
\end{equation}
Whence the sequence $(h_{\BB,\XX}(\xx_n))_{n=1}^\infty$ is biorthogonal to $\BB$, i.e.,
\begin{equation}\label{eq:bidualbasis}
\BB^{**}=(\xx_n^{**})_{n=1}^\infty=(h_{\BB,\XX}(\xx_n))_{n=1}^\infty.
\end{equation}
The coordinate operator  associated the the basic sequence $\BB^*$ is given by
\begin{equation}\label{eq:dualcoordinateoperator}
\Fou^*(f^*)=(f^*(\xx_n))_{n=1}^\infty, \quad f^*\in[\BB^*],
\end{equation}
and the map $\Fou^*$ can be defined on the whole of $\XX^*$. The support of $f^*\in\XX^*$ with respect to $\BB$ is the set
\[
\supp(f^*)=\{n\in\NN \colon f^*(\xx_n)\not=0\}.
\]

The mapping $h_{\BB,\XX}$ is linear and satisfies $\Vert h_{\BB,\XX}\Vert \le 1$.  By \eqref{eq:bibiorthogonal}, $h_{\BB,\XX}$ is an isomorphic embedding if and only if $[\BB^*]$ is a norming set for $\XX$. However, even in the case when $\XX$ is locally convex, so that $h_\XX$ is an isomorphic embedding,  $h_{\BB,\XX}$ is not necessarily an isomorphic embedding.  Since we could not find any examples or explicit references  about the existence of such kind of bases, we next show that they exist. We emphasize that if $h_{\BB,\XX}$ is an isomorphic embedding then the bidual sequence $\BB^{**}$ of $\BB$ is equivalent to $\BB$.

\begin{proposition}\label{prop:NonNorming}There is an $M$-bounded total basis $\BB$ of a Banach space $\XX$ such that $h_{\BB,\XX}$ is not an isomorphic embedding.
\end{proposition}

\begin{proof}Pick a Banach space $\XX$ which has infinite codimension in $\XX^{**}$. From \cite{DL1972} we know that  there exists a separable total closed subspace $V\subseteq\XX^*$ which is not norming. From Theorem 8.1 of \cite{Singer3}*{$\S$III} we infer that $\XX$  has a  basis $\BB_0$ such that $[\BB_0^*]=V$.
Using Proposition 1 from \cite{OP1975} we obtain that the space $\XX\oplus \ell_2$ has an $M$-bounded basis  $\BB$ such that, if we  identify  $ (\XX\oplus\ell_2)^*$ with $\XX^*\oplus\ell_2$, then $[\BB^*]=V\oplus\ell_2$. We deduce that $\BB$ is a non-norming total basis.
\end{proof}

Let $\gamma\in\FF^\NN$ be so that $M_\gamma=M_\gamma[\BB,\XX]$ is well-defined. We put 
\[
M_\gamma^*=M_{\gamma}^*[\BB,\XX]:=(M_{\gamma}[\BB,\XX])^*.\] From \eqref{eq:dualcoordinateoperator} we deduce that whenever $M_\gamma$ is well-defined on $\XX$,
\begin{equation}\label{eq:dualoperatorrelations}
M_\gamma^*[\BB,\XX]|_{[\BB^*]}=M_\gamma[\BB^*,\XX^*].
\end{equation}

 The next proposition gathers some familiar properties of dual bases that we will need. 
\begin{proposition}\label{prop:basesdual}
Let $\BB=(\xx_n)_{n=1}^\infty$ be a basis of a quasi-Banach space $\XX$ 
with coordinate functionals $\BB^*=(\xx_n^*)_{n=1}^\infty$. Then:
\begin{itemize}

\item[(i)] $\BB^*$ is a total basic sequence.

\item[(ii)] If $\BB$ is $M$-bounded so is $\BB^*$.

\item[(iii)] If $\BB$ is an Schauder basis so is $\BB^*$.

\item[(iv)] If $\BB$ is an unconditional basis so is $\BB^*$.

\item[(v)] If $\BB$ is $M$-bounded and semi-normalized so is $\BB^*$.

\end{itemize}
\end{proposition}

\begin{proof}(i) it is immediate from \eqref{eq:dualcoordinateoperator}. We infer from \eqref{eq:dualoperatorrelations} that
\[
\Vert S_A[\BB^*,\XX^*]\Vert\le \Vert S_A[\BB,\XX]\Vert, \quad A\subseteq \NN,\, |A|<\infty.
\]
Consequently, (ii), (iii) and (iv) hold.  

Since $\Vert h_{\BB,\XX}\Vert \le 1$,
\[
\sup_n \{ \Vert \xx_n^{**} \Vert, \Vert \xx_n^*\Vert \} 
= \sup_n \{ \Vert h_{\BB,\XX}(\xx_n^*)\Vert, \Vert \xx_n^*\Vert \} 
\le \sup_n \{ \Vert \xx_n^* \Vert, \Vert \xx_n\Vert \}.
\]
By Lemma~\ref{lem:bases1}, (v) holds.  
\end{proof}

Next we see a result about the operator defined in \eqref{eq:bibiorthogonal}.

\begin{lemma}\label{lem:basesreflexivity}Let $\BB$ be a basis of a locally convex quasi-Banach space $\XX$. Suppose that there is a positive constant $C<\infty$ such that for every $A\subseteq\NN$ finite and every $f^*\in\XX^*$ there is a subset $B$ of $\NN$ containing $A$ satisfying $\Vert S_B^*(f^*)\Vert \le C \Vert f^*\Vert$. Then $h_{\BB,\XX}$ is an isomorphic embedding. Thus, the bases $\BB^{**}$ and $\BB$ are equivalent. 
\end{lemma}
\begin{proof}Assume that $\XX$ is a Banach space. Let $f\in\langle \xx_n \colon n\in\NN\rangle$ and $f^*\in B_{\XX^*}$. Pick $B\supseteq \supp(f)$ finite such that $ \Vert S_B^*(f^*)\Vert \le C \Vert f^*\Vert$. Since $S_B^*(f^*)\in[\BB^*]$ we have
\begin{align*}
| f^*(f)| 
&=|S_B^*(f^*)(f)|\\
&=| h_{\BB,\XX}(f)(S_B^*(f^*))|\\
&\le \Vert h_{\BB,\XX}(f) \Vert \, \Vert S_B^*(f^*)\Vert\\
&\le C \Vert h_{\BB,\XX}(f) \Vert \, \Vert f^*\Vert\\
&\le C \Vert h_{\BB,\XX}(f) \Vert.
\end{align*}
Taking the supremum on $f^*$ we obtain $ \Vert f \Vert \le C \Vert h_{\BB,\XX}(f) \Vert$ for all $f\in\langle \xx_n \colon n\in\NN\rangle$. This inequality  extends to all $f\in\XX$ by density.
\end{proof} 

We close this section by noticing that the following well-known result about  Schauder bases can be obtained as a consequence of Lemma~\ref{lem:basesreflexivity}.
\begin{theorem}[see \cite{AlbiacKalton2016}*{Proposition 3.2.3 and Corollary 3.2.4}]\label{thm:Schaudereflexivity}Let $\BB$ be a Schauder basis of a locally convex quasi-Banach space. Then $h_{\BB,\XX}$ is an isomorphic embedding. Thus, the bases $\BB^{**}$ and $\BB$ are equivalent. 
\end{theorem}

\section{Unconditionality for constant coefficients}\label{Sec4}
\noindent
Given  $J\subseteq \NN$, we shall denote the set $\{(\varepsilon_j)_{j\in J}\colon |\varepsilon_{j}|= 1\}$ by $\EE_J$.
Let $\BB=(\xx_n)_{n=1}^\infty$ be basis (or basic sequence) of a quasi-Banach space $\XX$. If $A\subseteq\NN$ finite and $\varepsilon
=(\varepsilon_n)_{n\in A}\in \EE_A$, we set 
\[
\Ind_{\varepsilon,A}=\Ind_{\varepsilon,A}[\BB,\XX]=\sum_{n\in A} \varepsilon_n \xx_n.
\]
If $\varepsilon_n=1$ for all $n\in A$ we put $\Ind_A=\Ind_{\varepsilon,A}$. In the case when the basis is the unit vector system $\BB_e$ of $\FF^\NN$ we use the notation $\Ind_{\varepsilon,A}[\BB_e]$ and $\Ind_{A}[\BB_e]$, respectively.

\begin{definition} A basis $\BB=(\xx_n)_{n=1}^\infty$ of a quasi-Banach space $\XX$ is said to be \emph{suppression unconditional for constant coefficients} (SUCC for short) if there is a constant $1\le C<\infty$ such that for all $B\subseteq A\subseteq\NN$ and $\varepsilon\in\EE_A$,
\begin{equation}\label{eq:succ}
\left\Vert\Ind_{\varepsilon,B} \right\Vert
\le C\left\Vert\Ind_{\varepsilon,A} \right\Vert.
\end{equation} 
The smallest constant $C$ in \eqref{eq:succ} will be denoted by $K_{sc}=K_{sc}[\BB,\XX]$.
\end{definition}

\begin{lemma}\label{lem:ucc2} Let $\BB=(\xx_n)_{n=1}^\infty$ be an $M$-bounded semi-normalized basis of a quasi-Banach space $\XX$. The following are equivalent:
\begin{itemize}
\item[(i)] $\BB$ is SUCC.
\item[(ii)] There is a constant $C$ such that for any $A\subseteq\NN$ and $\varepsilon\in\EE_A$ we have 
\begin{equation}\label{eq:ucc3}
\left\Vert \sum_{n\in A} a_n\, \xx_n \right\Vert
\le C\left\Vert\Ind_{\varepsilon,A} \right\Vert,
\end{equation}
whenever $|a_n|\le 1$.
\item[(iii)] There is a constant $C$ such that for any $A\subseteq\NN$ and any $\varepsilon,\delta\in\EE_A$,
\begin{equation}\label{eq:ucc4}
\Vert \Ind_{\varepsilon,A}\Vert
\le C \Vert \Ind_{\delta,A}\Vert.
\end{equation}
\end{itemize}
Moreover, if $C_2$ is the optimal constant $C$ in \eqref{eq:ucc3} and $C_3$ is the optimal constant $C$ in \eqref{eq:ucc4} we have $K_{sc} \le C_2 \le B_p K_{sc}$ and $C_3 \le C_2 \le A_p C_3$.
\end{lemma}
\begin{proof}(i) $\Rightarrow$ (ii) follows Corollary~\ref{cor:convexity2}, and (iii) $\Rightarrow$ (ii) follows from Corollary~\ref{cor:convexity}~(ii). (ii) $\Rightarrow$ (i) and (ii) $\Rightarrow$ (iii) are obvious.
\end{proof}

In some sense, \eqref{eq:ucc3} is an upper lattice-unconditionality condition for constant coefficients. Next, we introduce two additional conditions with a flavor of unconditionality. 

\begin{definition} Let $\BB=(\xx_n)_{n=1}^\infty$ be a basis for a quasi-Banach space $\XX$.

 (a) $\BB$ is said to be \emph{lower unconditional for constant coefficients} (for short LUCC) if there is a constant $1\le C$ such that for all $A\subseteq\NN$ and all $\varepsilon=(\varepsilon_n)_{n\in A}\in\EE_A$, 
\begin{equation}\label{eq:lucc}
\Vert \Ind_{\varepsilon,A}\Vert \le C \left\Vert \sum_{n\in A} \varepsilon_n\, a_n \,\xx_n\right\Vert,
\end{equation}
whenever $a_n\ge 1$. The smallest constant $C$ in \eqref{eq:lucc} will be denoted by $K_{lc}=K_{lc}[\BB,\XX]$. 

 (b) $\BB$ is said to be \emph{lattice partially unconditional} (LPU for short) if there is a constant $1\le C$ such that 
for all $A\subseteq\NN$,
\begin{equation}\label{eq:ucc7}
\left\Vert \sum_{n\in A} a_n \,\xx_n \right\Vert \le C \left\Vert \sum_{n\in A} b_n \,\xx_n\right\Vert,\end{equation}
whenever $\max_{n\in A} |a_n|\le \min_{n\in A} |b_n|$.
The optimal constant $C$ in \eqref{eq:ucc7} will be denoted by $K_{pu}=K_{pu}[\BB,\XX]$.

\end{definition}

\begin{proposition}\label{prop:ucc4}A basis is LPU if and only if it is simultaneously SUCC and LUCC. Moreover, if $\XX$ is $p$-Banach, we have $\max \lbrace K_{sc}, K_{lc}\rbrace\le K_{pu}$ and $K_{pu}\le A_p K_{sc} K_{lc}$. 
\end{proposition}
\begin{proof}Assume that $\BB=(\xx_n)_{n=1}^\infty$ is SUCC and LUCC. Let $(a_n)_{n\in A}$ and $(b_n)_{n\in A}$ be such that $|a_n|\le t:=\min_{k\in A} |b_k|$ for all $n\in A$. Put $\varepsilon_n=\sgn(b_n)$. Lemma~\ref{lem:ucc2} yields
\begin{align*}
\left\Vert \sum_{n\in A} a_n\, \xx_n\right\Vert
& \le t A_p K_{sc} \Vert \Ind_{\varepsilon,A}\Vert\\
&\le t A_p K_{sc} K_{lc} \left\Vert \sum_{n\in A} \varepsilon_n \, \frac{b_n}{t}\, \xx_n\right\Vert\\
&= A_p K_{sc} K_{lc} \left\Vert \sum_{n\in A} b_n\, \xx_n\right\Vert.
\end{align*}
The converse is obvious.
\end{proof}

\begin{proposition} A basis of a quasi-Banach space $\XX$ is LPU if and only if there is a constant $C$ such that
for any $A\subseteq\NN$ and any $\varepsilon\in\EE_A$,
\begin{equation}\label{eq:lpu}
\Vert \Ind_{\varepsilon,A}\Vert 
\le C \left\Vert \sum_{n\in A} b_n \,\xx_n\right\Vert,
\end{equation}
whenever $|b_n|\ge 1$. Moreover, if $\XX$ is $p$-Banach and $C_1$ is the optimal constant $C$ in \eqref{eq:lpu}, then $C_1\le K_{pu}\le A_p C_1$.
\end{proposition}

\begin{proof}Let $\BB$ be a basis satisfying \eqref{eq:lpu}. Let $(a_n)_{n\in A}$ and $(b_n)_{n\in A}$ be such that $|a_n|\le t \le |b_n|$ for all $n\in A$ and some $t\in(0,\infty)$. By dilation we can assume without loss of generality that $t=1$.  Corollary~\ref{cor:convexity}~(ii) yields
\[
\left\Vert \sum_{n\in A} a_n \,\xx_n\right\Vert\le A_p \sup_{\varepsilon\in\EE_A}\Vert \Ind_{\varepsilon,A}\Vert \le
A_p C \left\Vert \sum_{n\in A} b_n \,\xx_n\right\Vert.
\]
The converse is obvious.
\end{proof}

We close this section with a lemma which connects SUCC bases with LUCC bases. Note, however, that  the example from   \S\ref{SUCCnotLUCC}
 shows that SUCC does not implies LUCC!

\begin{lemma}\label{lem:ucc5}Let $\BB=(\xx_n)_{n=1}^\infty$ be a SUCC basis of a quasi-Banach space $\XX$. Then there are constants $1<s, C<\infty$ such that for any $A\subset \NN$ finite and any $\varepsilon=(\varepsilon_n)_{n\in A}\in\EE_A$,
\begin{equation}\label{eq:ucc2}
\Vert \Ind_{\varepsilon,A}\Vert
\le C \left\Vert \sum_{n\in A} b_n\, \varepsilon_n \,\xx_n\right\Vert,
\end{equation}
whenever $1 \le b_n \le s$.
If $\XX$ is $p$-Banach, \eqref{eq:ucc2} is satisfied for each
\[
1<s<1+A_p^{-1} K_{sc}^{-1},
\] 
with
\[
C= (1-A_p^p K_{sc}^p (s-1)^p)^{-1/p}.
\]
\end{lemma}

\begin{proof}Assume that $\XX$ is $p$-Banach and that $1\le b_n \le s$. By Corollary~\ref{cor:convexity}~(i) we have
\[
\left\Vert \sum_{n\in A} (b_n-1) \, \varepsilon_n \, \xx_n\right\Vert \le A_p K_{sc} (s-1) \left\Vert \sum_{n\in A} \varepsilon_n \, \xx_n\right\Vert.
\]
Using  the reverse  triangle law, 
\begin{align*}
\left\Vert \sum_{n\in A} b_n \, \varepsilon_n \, \xx_n\right\Vert^p
&\ge \left\Vert \sum_{n\in A} \varepsilon_n \, \xx_n\right\Vert^p -\left\Vert \sum_{n\in A} (b_n-1) \, \varepsilon_n \, \xx_n\right\Vert^p\\
&\ge \left\Vert \sum_{n\in A} \varepsilon_n \, \xx_n\right\Vert^p -A_p^p K_{sc}^p (s-1)^p \left\Vert  \sum_{n\in A} \varepsilon_n \xx_n\right\Vert^p\\
& = C(s) \left\Vert  \sum_{n\in A} \varepsilon_n \xx_n\right\Vert^p,
\end{align*}
where $C(s)=1- A_p^p K_{sc}^p (s-1)^p$. Since $C(s)>0$ if $1<s<1+A_p^{-1} K_{sc}^{-1}$ we are done.
\end{proof}

\section{Quasi-greedy bases}\label{Sec5}
\noindent  
Let us first recall the main ingredients of the section. A \emph{greedy set} of a vector $f\in\XX$ with respect to a basis $\BB=(\xx_n)_{n=1}^\infty$ of $\XX$ is a set $A\subseteq\NN$ such that $|\xx_k^*(f)|\le |\xx_n^*(f)|$ for all $k\in\NN\setminus A$ and $n\in A$. Notice that $\emptyset$ is always a greedy set. By Lemma~\ref{lem:bases5}~(iii), if $A$ is infinite, then $A$ is a greedy set of $f$ if and only if $\supp(f)\subseteq A$. A \emph{greedy sum} of $f$ is a coordinate projection $S_A(f)$ onto a greedy set $A$. 

Suppose that $\BB$ is semi-normalized and $M$-bounded.
If $f\in\XX$, by Lemma~\ref{lem:bases5}~(iii), its coefficient sequence $(a_n)_{n=1}^\infty:=\Fou(f)$ belongs to $c_0$ and so for every $m\in\NN\cup \{0\}$ there is a unique (possibly empty) set $A=A_m(f)\subseteq \NN\cup\{0\}$ of cardinality $|A|=m$ such that whenever $n\in A$ and $k \in \NN\setminus A$, either $|a_n|>|a_k|$ or $|a_n|=|a_k|$ and $n<k$.  Such greedy set $A$ will be called the \emph{$m$th greedy set} of $f$. 

A set $A$ is said to be \emph{strictly greedy} for $f$ if $|\xx_k^*(f)|<|\xx_n^*(f)|$ whenever $n\in A$ and $k\in\NN\setminus A$. Note that if $|A|=m<\infty$ and $A$ is a strictly greedy set then $A=A_m(f)$.

The coordinate projection $\GG_m(f):=S_{A_m(f)}(f)$ onto the $m$th greedy set will be called the   \emph{$m$th greedy approximation} to $f$ and the sequence $(\GG_m(f))_{m=1}^\infty$, the (thresholding) \emph{greedy algorithm} of $f$. Note that given an $M$-bounded semi-normalized basis $\BB$ in $\XX$ and $f\in\XX$ we have $A_m(f)\subseteq A_{m+1}(f)$ for every $m\in\NN$ and so there is unique (one-to-one) map $\phi_f\colon\NN\to\NN$ such that $A_m(f)=\{ \phi_f(j) \colon 1\le j\le m\}$ for every $m\in\NN$. We will we refer to the map $\phi_f$ as the \emph{greedy ordering} of $f\in\XX$ with respect to $\BB$.  The formal series 
\begin{equation}\label{PWgreedyseries}
\sum_{j=1}^\infty \xx^*_{\phi_f(j)} (f)\, \xx_{\phi_f(j)}
\end{equation}
will be called the \emph{greedy series} of $f$. Note that the $m$th partial sum of the greedy series of $f$ is $\GG_m(f)$.

If $A$ is a greedy set of $f$ and $|A|=m$, then $A$ is the $m$th greedy set of the vector
\[f_{\delta,A}:=f+\delta \sum_{n\in A} \sgn(\xx_n^*(f))\,\xx_n,\] and $\lim_{\delta\to0^+} f_{\delta,A}=f$. Hence, in many situations, provided that the quasi-norm on $\XX$ is continuous, if a statement holds for the $m$th greedy set of every $f\in\XX$, by a perturbation technique the very same statement also holds for every greedy set of cardinality $m$ of every $f\in\XX$.

An $M$-bounded semi-normalized basis $\BB$ for a quasi-Banach space $\XX$ is said to be \emph{quasi-greedy} if there is a constant $C<\infty$ such that for all $f\in \XX$,
\[
\Vert S_A(f)\Vert \le C \Vert f \Vert
\]
whenever $A$ is a finite greedy set of $f$. If a basis is quasi-greedy there is a (possibly larger) constant $C$ such that for all $f\in \XX$,
\begin{equation}\label{eq:qg}
\Vert S_{A\setminus B}(f)\Vert\le C \Vert f \Vert
\end{equation}
whenever $A$ and $B$ are finite greedy sets of $f$ with $B\subseteq A$. The smallest constant $C$ in \eqref{eq:qg} will be called the quasi-greedy constant of the basis, and will be denoted by $C_{qg}=C_{qg}[\BB,\XX]$. 

Of course, any semi-normalized unconditional basis is quasi-greedy, and 
\[
C_{qg}\le K_{su}.
\]

The convergence of greedy series is closely related to quasi-greedy bases. In this direction,
Theorem~\ref{PW12.thm1} rounds off Theorem 1 from \cite{Wo2000} and completely settles the question of characterizing quasi-greedy bases in terms of the convergence of greedy series.
\begin{theorem}\label{PW12.thm1}
Let $\BB=(\xx_n)_{n=1}^\infty$ be an $M$-bounded semi-normalized basis in a quasi-Banach space $\XX$. The following conditions are equivalent.
\begin{itemize}
\item[(i)] $\BB$ is quasi-greedy.
\item[(ii)] For every $f\in \XX$ the greedy series of $f$ converges.
\item[(iii)] For every $f\in \XX$ the greedy algorithm $(\GG_m(f))_{m=1}^\infty$ is bounded.
\end{itemize}
\end{theorem}

Implicit in the proof of \cite{Wo2000}*{Theorem 1} is the concept of strong Markushevich basis. Following \cite{Ruckle}, we say that a basis is \emph{strong} if for every (infinite) subset $A$ of $\NN$,
\begin{equation}\label{strongbasis}
[\xx_n \colon n\in A]=\{ f\in \XX \colon \xx_n^*(f)=0 \text{ for all }n\in\NN\setminus A\}.
\end{equation}
If the greedy series of $f$ converges (to $f$) for every $f\in\XX$ then the basis must be strong. Therefore, Theorem~\ref{PW12.thm1} yields the following consequence. In fact, the advance of Theorem~\ref{PW12.thm1} with respect to \cite{Wo2000}*{Theorem 1} must be understood in light of Corollary~\ref{StrongAnso}.

\begin{corollary}\label{StrongAnso} If $\BB$ is a quasi-greedy basis of a quasi-Banach space $\XX$ then it is strong. 
\end{corollary} 
We split the proof of Theorem~\ref{PW12.thm1} in several preparatory lemmas which rely on original ideas from \cite{Wo2000}.

\begin{lemma}\label{lem:12a} Let $\BB=(\xx_n)_{n=1}^\infty$ be a quasi-greedy basis of a quasi-Banach space $\XX$. Suppose that $f$ and $z$ are vectors in $\XX$ and that $D$ is a greedy set of $f-z$ such that $\supp(z)\subseteq D$. Then 
$
\Vert f- S_D(f)\Vert \le C_{qg} \Vert f -z\Vert.
$
\end{lemma}
\begin{proof}Since $z=S_D(z)$ we have
\[ 
\|f- S_D(f) \| 
= \|f-z - S_D(f-z)\| 
\leq C_{qg} \Vert f-z\Vert.\qedhere
\]
\end{proof}

\begin{lemma}\label{lem:12c} Let $\BB$ be a basis of a quasi-Banach space $\XX$. If $\BB$ is not quasi-greedy then for every $R>0$ and every finite set $A\subseteq \NN$ there exists $f\in\XX$ of finite support disjoint with $A$ and a strictly greedy set $B$ of $f$ such that $\| S_B (f)\| > R \| f \|$.	
\end{lemma}

\begin{proof} Without loss of generality we assume that $\| \cdot \|$ is a $p$-norm for some $p\in(0,1]$. Pick $R_0=(M^p+R^p+M^pR^p)^{1/p}$, where $M=\max_{D\subseteq A} \|S_D\|$. Let us fix a vector $g_0\in\XX$ and a greedy set $B_0$ of $g_0$ such that $\| S_{B_0}(g_0) \| > R_0\Vert g_0\Vert$. If we put $g_1=g_0-S_A(g_0)$ then $\|g_1\|^p\leq(1+ M^p)\Vert g_0\Vert^p $. Let $B=B_0\setminus A$ and $D= B_0 \cap A$, so that $B$ is a greedy set of $g_1$. We have
\begin{align*}
\Vert S_{B}(g_1) \Vert^p 
&=\Vert S_B(g_0)\Vert^p\\
&=\Vert S_{B_0}(g_0)-S_D(g_0)\Vert^p\\
&=\Vert S_{B_0}(g_0)\Vert^p -\Vert S_D(g_0)\Vert^p\\
&> (R_0^p-M^p) \Vert g_0\Vert^p\\
&\ge \frac{R_0^p-M^p}{1+M^p} \Vert g_1 \Vert^p\\
&=R^p \Vert g_1 \Vert^p.
\end{align*}
For each $t>0$ we pick $f_t\in\XX$ with finite support such that 
\[
\left\Vert f_t-g_1- t \sum_{n\in B} \sgn(\xx_n^*(g_1)) \, \xx_n\right\Vert<\frac{t}{2\beta}.
\]
It follows that $B$ is a strictly greedy set of $f_t$ and $\lim_{t\to 0^+} f_t=g_1$. Then $\lim_{t\to 0^+} S_B(f_t)=S_B(g_1)$. Hence, for $t$ small enough, $\Vert S_{B}(f_t) \Vert>R \Vert f_t \Vert$. 
\end{proof}

\begin{proof}[Proof of Theorem~\ref{PW12.thm1}] We start by proving that (i) implies (ii). 
Let $f\in\XX$ and $\epsilon>0$. We pick $z=\sum_{n\in B}a_n \, \xx_n$ with $B$ finite  such that
$
\|f-z\|\leq \epsilon/C^2_{qg}.
$
Perturbing the vector $z$ if necessary we can assume that $B$ is  nonempty and that $a_n\neq \xx_n^*(f)$, i.e., $\xx_n^*(f-z)\neq 0$, for every $n\in B$. Set $\nu= \min_{n\in B} |\xx_n^*(f-z)| >0$. We have
\[
\supp(z)\subseteq B\subseteq D:= \{n\in \NN \colon |\xx_n^*(f-z)|\ge \nu \}.
\]
Since $D$ is a strictly greedy set of $f-z$, applying Lemma~\ref{lem:12a} we obtain,
\[
\Vert f -S_D(f)\Vert \le C_{qg} \Vert f-z\Vert \le \frac{\epsilon}{C_{qg}}.
\]
Set $\mu=\min_{n\in D\cap\supp(f)} |\xx_n^*(f)|$ (with the convention that $\mu=\infty$ if $D\cap\supp(f)=\emptyset$) and let $m\ge |\{n\in \NN \colon |\xx_n^*(f)|\ge \mu \}|$. Since $D\cap\supp(f)\subseteq A_m(f)$, the set $G:=A_m(f)\setminus (D\cap\supp(f))$ is greedy for $g:=f-S_D(f)$, so that
\[
\left \|f-\sum_{n\in A_m(f)} \xx_n^*(f) \, \xx_n\right\|=\Vert g -S_G(g)\Vert\leq C_{qg} \Vert g\Vert\leq \epsilon.
\]	

The implication (ii) $\Rightarrow$ (iii) is trivial, so we complete the proof by showing that if (i) does not hold then (iii) does not hold either. Under the assumption that $\BB$ is not quasi-greedy we recursively construct a sequence $(f_{k})_{k=1}^{\infty}$ in $\XX$ and a sequence $(B_k)_{k=1}^{\infty}$ of finite subsets of $\NN$ such that, if we set $\mu_1=\infty$, and for $k\ge 2$ define
\[
\mu_k=\min\{ |\xx_n^*(f_{k-1})| \colon n\in\supp(f_{k-1}) \},
\]
then:
\begin{itemize}
\item[(a)] $A_k:=\supp(f_k)$ is finite and disjoint with $\cup_{i=1}^{k-1} \supp(f_i)$,
\item[(b)] $B_k$ is a strictly greedy set of $f_{k}$,
\item[(c)] $\Vert f_{k}\Vert\le 2^{-k}$,
\item[(d)] $\Vert S_{B_k}(f_{k})\Vert >2^k$, and
\item[(e)] $\max\{ |\xx_n^*(f_k)|: n\in\NN \} < \mu_k$
\end{itemize}
for every $k\in\NN$. Suppose we have manufactured $f_{i}$ and $B_{i}$ for $i<k$. Put $\beta=\max_{n\in\NN} \Vert \xx_n^*\Vert <\infty$ and
\[ 
\gamma_{k}=\min\{2^{-2k},(2 \beta)^{-1} 2^{-k} \mu_k \}.
\] 
By Lemma~\ref{lem:12c}, there exists $f_{k}\in\XX$ whose support is finite and disjoint with $\cup_{i=1}^{k-1} \supp(f_i)$ and a strictly greedy set $B_k$ of $f_{k}$ such that $\Vert S_{B_k}(f_k) \Vert> \gamma_{k}^{-1}\Vert f_{k}\Vert$. By homogeneity we can choose $f_{k}$ satisfying $\Vert f_{k}\Vert= 2^k\gamma_{k}$, so that (d) holds. Since $\gamma_k\le 2^{-2k}$ (c) also holds. For any $n\in \NN$ we have
\[
|\xx_{n}^{*}(f_{k})| \le \Vert \xx_{n}^{\ast}\Vert\, \Vert f_{k}\Vert\le \Vert \xx_{n}^{\ast}\Vert \frac{\mu_k}{2\beta}
\le
\frac{\mu_k}{2},
\]
and so (e) also holds.

Since $\sum_{k=1}^\infty \Vert f_k\Vert^p<\infty$ for every $p\in(0,1]$, the series $\sum_{k=1}^{\infty}f_{k}$ converges to some $f\in\XX$ such that
\[
\xx_n^*(f)=\begin{cases} \xx_n^*(f_k) &\text{ if }n\in A_k, \\ 0 & \text{ if }n\notin\cup_{k=1}^\infty A_k.\end{cases}
\]
By (b) and (e), both $D_k=\cup_{i=1}^{k-1} A_i$ and $ F_k=D_k\cup B_k$ are strictly greedy sets of $f$. Thus, if $m_k=|D_k|$ and $q_k=|F_k|$, 
\[
\Vert \GG_{m_k+q_k}(f) - \GG_{m_k}(f) \Vert =\Vert S_{F_k}(f) - S_{D_k}(f) \Vert=\Vert S_{B_k}(f_{k})\Vert > 2^k
\]
for every $k\in\NN$. We deduce that $\sup_m \Vert \GG_m(f)\Vert=\infty$.
\end{proof} 

Before moving on, let us write down an immediate consequence of Theorem~\ref{PW12.thm1}.
\begin{corollary}\label{cor:QGtotal}Suppose $\BB=(\xx_n)_{n=1}^\infty$ is a quasi-greedy basis of a quasi-Banach space $\XX$. Then $\BB$ is total.
\end{corollary}
\begin{proof}Let $f\in\XX$ be such that $\xx_n^*(f)=0$ for every $n\in\NN$. Then, $\GG_m(f)=0$ for every $m\in\NN$. By Theorem~\ref{PW12.thm1}, $f=\lim_m \GG_m(f)=0$.
\end{proof}

\begin{definition}\label{PW:def_qglc}
A basis $\BB$ of a quasi-Banach space $\XX$ is said to be \emph{quasi-greedy for largest coefficients} (QGLC for short) if there is constant $C$ such that
\begin{equation*}
\Vert\Ind_{\varepsilon, A}\Vert \le C \Vert f + \Ind_{\varepsilon, A}\Vert
\end{equation*}
for any $A\subseteq\NN$ finite, any $\varepsilon\in\EE_A$, and any $f\in\XX$ such that $\supp f\cap A=\emptyset$ and $|\xx_n^*(f)|\le 1$ for all $n\in\NN$. 

If $\BB$ is QGLC, then the smallest constant $C$ such that (under the same conditions as before)
\begin{equation}\label{eq:qglc}
\max\{\Vert f\Vert, \Vert\Ind_{\varepsilon, A}\Vert\} \le C \Vert f + \Ind_{\varepsilon, A}\Vert 
\end{equation}
will be called the QGLC constant of the basis, and will be denoted by $C_{ql}=C_{ql}[\BB,\XX]$. 
\end{definition}

Of course, any quasi-greedy basis is QGLC, and 
\[
C_{ql}\le C_{qg}.
\] 

The following simple result was stated for quasi-greedy bases and put to use in \cite{Wo2000}.
\begin{lemma}\label{lem:QGtoSUCC} Suppose $\BB$ is a QGLC basis for a quasi-Banach space $\XX$. Then $\BB$ is SUCC. Quantitatively,
$K_{sc}[\BB,\XX] \le C_{ql}[\BB,\XX]$.
\end{lemma} 

For Banach spaces it is known that quasi-greedy bases are LUCC, and attention must be drawn to the fact that all known proofs of this result (see \cite{DKKT2003}*{Lemma 2.2} and \cite{AA2017}*{Lemma 3.5}) depend heavily on the local convexity of the space. The LUCC property becomes a key tool, for instance, in the study of conditional quasi-greedy bases, which calls for the corresponding relation between quasi-greedy bases and LUCC in the nonlocally convex setting.  This is what the next result accomplishes with a radically different approach in the techniques.

\begin{theorem}\label{thm:QGtoLUCC} If $\BB$ is a quasi-greedy basis in a quasi-Banach space $\XX$ then $\BB$ is LUCC. Quantitavely, if $\XX$ is $p$-Banach then
\[
K_{lc}\le C_{qg} \eta_p(C_{qg}),
\]
where for $u> 0$,
\begin{equation}\label{eq:function}
\eta_p(u)=\min_{0<t<1} (1-t^p)^{-1/p} (1-(1+ A_p^{-1}u^{-1} t)^{-p})^{-1/p}.
\end{equation}
\end{theorem}

\begin{proof} Assume that $\XX$ is $p$-Banach and let $\BB=(\xx_n)_{n=1}^\infty$ be quasi-greedy. By Lemma~\ref{lem:QGtoSUCC}, $\BB$ is SUCC and so we can pick $1<s,C<\infty$ as in Lemma~\ref{lem:ucc5}. Let $A\subseteq\NN$ finite, $\varepsilon=(\varepsilon_n)_{n\in A}\in\EE_A$, and $(b_n)_{n\in A}\in[1,\infty)^A$. For $j\in\NN\cup\{0\}$ consider the sets
\[
B_j=\{n\in A \colon s^{j} \le b_n \}\]
and 
\[
A_j=B_{j} \setminus B_{j+1}=\{n\in A \colon s^{j} \le b_n <s^{j+1} \}.
\]
Notice that $(A_j)_{j=0}^\infty$ is a partition of $A$. Using Proposition~\ref{prop:lpgalbed} and taking into account that $B_j$ is a greedy set of $f=\sum_{n\in A} b_n \, \xx_n$, we obtain 
\begin{align*}
\left\Vert \sum_{n\in A} \varepsilon_n \xx_n\right\Vert 
&=\left\Vert \sum_{j=0}^\infty s^{-j} s^{j} \sum_{n\in A_j} \varepsilon_n \xx_n\right\Vert \\
&=\left( \sum_{j=0}^\infty s^{-jp} \right)^{1/p} \sup_{j\ge 0} s^j \left\Vert \sum_{n\in A_j} \varepsilon_n \xx_n\right\Vert \\
&\le C \left( \sum_{j=0}^\infty s^{-jp} \right)^{1/p}\sup_{j\ge 0} s^j \left\Vert \sum_{n\in A_j} \varepsilon_n s^{-j} b_n \xx_n\right\Vert\\
&= C (1-s^{-p})^{-1/p} \sup_{j\ge 0} \left\Vert S_{B_j \setminus B_{j+1}}(f) \right\Vert\\
&\le C_{qg} C(1-s^{-p})^{-1/p} \Vert f\Vert.
\end{align*}
Hence, for any $ 1<s<1+A_p^{-1} C_{qg}^{-1}$,
\[
K_{lc} \le C_{qg} C(1-s^{-p})^{-1/p}= C_{qg}(1-A_p^p K_{sc}^p (s-1)^p)^{-1/p} (1-s^{-p})^{-1/p}.
\]
Minimizing over $s$ puts an end to the proof.
\end{proof}

\begin{remark} For $p=1$ the best known estimate for $K_{lc}$ is 
\[K_{lc}\le C_{qg}\] (see \cite{DKKT2003}*{Proof of Lemma 2.2} and \cite{AA2017}*{Lemma 3.5}). Notice that the function $\eta_p$ defined in \eqref{eq:function} is increasing and that $\lim_{u \to 0^+} \eta_p(u)=1$. Thus the upper bound for $K_{lc}$ provided by Theorem~\ref{thm:QGtoLUCC} is larger than $C_{qg}$ even when $p=1$. Furthermore, since for a given $p\in(0,1]$ we have the estimate
\[
1-(1+ A_p^{-1} x)^{-p}\approx x, \quad 0<x\le 1,
\]
we obtain $\eta_p(u)\approx u^{1/p}$ for $u\ge 1$, which by Theorem~\ref{thm:QGtoLUCC} yields 
\[
K_{lc}\lesssim C_{qg}^{1+1/p}.\] And, again, for $p=1$ this gives an asymptotic estimate coarser than the one already known. 
\end{remark}

\begin{theorem}\label{thm:QGtoLPU} If $\BB$ is a quasi-greedy basis in a quasi-Banach space $\XX$ then $\BB$ is LPU. Quantitavely, if $\XX$ is $p$-Banach then
\[
K_{pu}\le A_p C^2_{qg} \eta_p(C_{qg}).
\]
\end{theorem}

\begin{proof}Just combine Theorem~\ref{thm:QGtoLUCC}, Lemma~\ref{lem:QGtoSUCC}, and Proposition~\ref{prop:ucc4}.
\end{proof} 

Theorem~\ref{thm:QGtoLUCC} is the tool we use to fix the following stability property of quasi-greedy bases, whose original proof seemed to be garbled. 

\begin{theorem}[see \cite{Wo2000}*{Proposition 3}] Let $\BB=(\xx_n)_{n=1}^\infty$ be a quasi-greedy basis of a quasi-Banach space $\XX$ and let $(\lambda_n)_{n=1}^\infty$ be a sequence of scalars such that $\inf_n |\lambda_n|>0$ and $\sup_n |\lambda_n| <\infty$. Then the perturbed basis $\BB'=(\lambda_n \, \xx_n)_{n=1}^\infty$ is quasi-greedy.
\end{theorem}
\begin{proof} By hypothesis, 
\[
E:=\sup_{k,n}\frac{|\lambda_n|}{|\lambda_k|}<\infty.
\]
Let $A$ be a greedy set of $f=\sum_{n=1}^\infty a_n \, \xx_n\in\XX$ with respect to $\BB'$. Put $t= \min_{n\in A} |a_n|$ and choose 
\begin{equation}\label{eq:renorming9}
A_1=\{j \in \NN \colon |a_j| > E t \} \text{ and }
A_2=\{j \in \NN \colon |a_j|\ge t \}.
\end{equation}
Note that $A_1$ and $A_2$ are greedy sets of $f$ with respect to $\BB$, and that $A_1\subseteq A\subseteq A_2$. If $j\in \NN\setminus A_1$ and $k\in A_2$ we have $|a_j| \le E t \le E |a_k|$. By Theorem~\ref{thm:QGtoLPU},
\[
\Vert S_{A\setminus A_1}(f)\Vert \le E K_{pu} \Vert S_{A_2\setminus A_1}(f)\Vert.
\]
In this estimate, if $\XX$ is $p$-Banach and we denote the quasi-greedy constant $C_{qg}[\BB,\XX]$ of the basis simply by $C$, we have $K_{pu}\le A_p C^2 \eta_p(C)$. 

Let $B\subseteq A$   be another greedy set of $f$ with respect to $\BB'$ and define $B_1$ and $B_2$ as in \eqref{eq:renorming9} by  replacing $A$ with $B$. We have $B_1\subseteq A_1$. Then,
\begin{align*}
\Vert S_{A\setminus B}(f)\Vert^p 
&\le \Vert S_{A_1\setminus B_1}(f)\Vert^p+\Vert S_{A\setminus A_1}(f)\Vert^p+\Vert S_{B\setminus B_1}(f)\Vert^p\\
&\le \Vert S_{A_1\setminus B_1}(f)\Vert^p+ E^p K_{pu}^p \left(\Vert S_{A_2\setminus A_1}(f)\Vert^p+ \Vert S_{B_2\setminus B_1}(f)\Vert^p\right).
\end{align*}
Hence, $\BB'$ is quasi-greedy and
\[
C_{qg}[\BB',\XX] \le C(1+2  A_p^p E^p C^{2p} \eta_p^p(C))^{1/p}.\qedhere
\]
\end{proof}

\subsection{Nonlinear operators related to the greedy algorithm} 
Let $\BB=(\xx_n)_{n=1}^\infty$ be an $M$-bounded semi-normalized basis of a quasi-Banach space $\XX$. With the convention that $\GG_\infty(f)=f$, let us put
\[
\GG_{r,m}(f)=\GG_m(f)-\GG_r(f),\quad
0\le r\le m\le\infty
\] 
and 
\[
\HH_m=\GG_{m,\infty}=\Id_\XX-\GG_m.
\] 
If the quasi-norm is continuous, a standard perturbation technique yields 
\[
\sup_{0\le r \le m\le\infty} \Vert \GG_{r,m}\Vert=\sup\Big\{ \Vert S_{A\setminus B} (f) \Vert \colon \Vert f\Vert\le 1,\, B\subseteq A \text{ greedy sets} \Big\}, 
\]
\[
\sup_m \Vert \GG_m\Vert=\sup\Big\{ \Vert S_A (f) \Vert \colon \Vert f\Vert\le 1,\, A \text{ greedy set} \Big\}
\]
and
\[\sup_m \Vert \HH_m\Vert=\sup\Big\{ \Vert f-S_A (f) \Vert \colon \Vert f\Vert\le 1,\, A \text{ greedy set} \Big\}.
\]
Thus,  a semi-normalized $M$-bounded basis is quasi-greedy if and only if $(\GG_{r,m})_{ r\le m}$ (or $(\GG_m)_{m=1}^\infty$, or $(\HH_m)_{m=1}^\infty$) is a uniformly bounded family of (non-linear) operators and, if the quasi-norm is continuous, $C_{qg}=\sup_{ r \le m} \Vert \GG_{r,m}\Vert$.

For each $f\in \XX$ and each $A\subseteq\NN$ finite, put
\begin{align*}
\UU(f,A) &= \min_{n\in A} |\xx_n^*(f)| \sum_{n\in A} \sgn (\xx_n^*(f)) \, \xx_n,\\
\TT(f,A)&=\UU(f,A)+S_{A^c}(f).
\end{align*}
Note that this definition makes sense even if $A=\emptyset$, in which case $\UU(f,A)=0$ and $\TT(f,A)=f$. If $A$ is an infinite set with $\supp(f)\subseteq A$ we will use the convention $\UU(f,A)=\TT(f,A)=0$.

Given $m\in\NN\cup\{0\}$, the $m$\emph{th-restricted truncation operator} $\UU_m\colon \XX \to \XX$  is defined as 
\[
\UU_m(f)=\UU(f,A_m(f)), \quad  f\in\XX.
\]
In turn, the  $m$\emph{th-truncation operator} $\TT_m\colon \XX\to \XX$ is defined as
\[
\TT_m(f)=\TT(f,A_m(f)),\quad f\in \XX.
\]
Notice that both $\UU_m$ and $\TT_m$ are non-linear and that $\TT_m=\UU_m+\HH_{m}$.

In the context of greedy-like bases the truncation operator was introduced in \cite{DKK2003} and studied in depth in \cite{BBG2017}. It is also implicit within the characterization of $1$-almost greedy bases from \cite{AA2017}. 

We put
\begin{equation}\label{eq:unnamed}
\Lambda_u=\Lambda_u[\BB,\XX]=\sup\{ \Vert \UU(f,A)\Vert \colon A \text{ greedy set of } f, \, \Vert f\Vert \le 1\},
\end{equation}
and
\begin{equation}\label{eq:truncation}
\Lambda_t=\Lambda_t[\BB,\XX]=\sup\{ \Vert \TT(f,A)\Vert \colon A \text{ greedy set of } f, \, \Vert f\Vert \le 1\}.
\end{equation}
If the quasi-norm is continuous, applying a perturbation technique yields 
\[
\Lambda_u=\sup_m \Vert \UU_m\Vert\] 
and  
\[
\Lambda_t=\sup_m \Vert \TT_m\Vert.
\]
Thus, $(\UU_m)_{m=0}^\infty$ (resp.\ $(\TT_m)_{m=0}^\infty$) is a uniformly bounded family of operators if and only if $\Lambda_u<\infty$ (resp.\ $\Lambda_t<\infty$).

\begin{lemma}\label{lem:qg9} Suppose $\BB$ is a quasi-greedy basis for a quasi-Banach space $\XX$. Then:
\begin{enumerate} 
\item[(i)] $\Lambda_u \le C_{qg} \Lambda_t$  
\item[(ii)] If $\XX$ is $p$-Banach,
\[\Lambda_t \le (C_{qg}^p+\Lambda_u^p)^{1/p}\quad \text{and}\quad \Lambda_u \le (C_{qg}^p+\Lambda_t^p)^{1/p}. 
\] 
\end{enumerate}
\end{lemma}

\begin{proof}
(i) Let $A$ be a greedy set of $f\in \XX$. Then $A$ also is a greedy set of $\TT(f,A)$, and $S_A(\TT(f,A))=\UU(f,A)$. Therefore
\[\Vert \UU(f,A)\Vert \le C_{qg} \Vert \TT(f,A) \Vert.\] 

\noindent (ii) If $\XX$ is $p$-Banach,
\[
\Vert \TT(f,A)\Vert
\le \left(\Vert \UU(f,A)\Vert^p +\Vert f -S_A(f)\Vert^p \right)^{1/p} 
\le ( \Lambda_u^p+C_{qg}^p)^{1/p} \Vert f\Vert.
\]
This inequality also holds switching the roles of $\TT$ and $\UU$, and this completes the proof.
\end{proof}

Given $\lambda=(\lambda_n)_{n=1}^\infty$ consider, when well-defined, the non-linear operator
\[
T_\lambda\colon\XX\to \XX, \quad f=\sum_{n=1}^\infty a_n \, \xx_n \mapsto \sum_{n=1}^\infty \lambda_n\, a_{\phi_f(n)}\, \xx_{\phi_f(n)}.
\]
Let $\JJ$ be the set of all non-decreasing sequences bounded below by $0$ and bounded above by $1$. Notice that the uniform boundedness of the family $(\TT_m)_{m=0}^\infty$ can be derived from the uniform boundedness of $(T_\lambda)_{\lambda\in \JJ}$. In the case when $\XX$ is a Banach space  and the basis $\BB$ is quasi-greedy, then $(T_\lambda)_{\lambda\in \JJ}$ is a uniformly bounded family of operators. The proof of this fact from \cite{AA2017} relies heavily on the convexity of the target space $\XX$, and it seems to be hopeless to try to generalize the arguments there to quasi-Banach spaces. In spite of that constraint we obtain a proof  of the uniform boundedness of $(\TT_m)_{m=0}^\infty$ based on different techniques.

\begin{theorem}\label{thm:qg5} Let $\BB$ be a quasi-greedy basis of a quasi-Banach space $\XX$. Then:
\begin{itemize}
\item[(i)] $(\UU_m)_{m=0}^\infty$ is a uniformly bounded family of operators. If $\XX$ is $p$-Banach, $\Lambda_u \le C_{qg}^2 \eta_p(C_{qg})$. 
\item[(ii)]For every $f\in \XX$ we have $\lim_{m\to \infty} \UU_m(f)=0$.
\end{itemize}
\end{theorem}

\begin{proof} By Theorem~\ref{thm:QGtoLUCC}, $\BB$ is LUCC. Then, if $A$ is a finite greedy set of $f\in\XX$,
\[
\Vert \UU(f,A)\Vert =K_{lc} \Vert S_A(f)\Vert \le C_{qg} K_{lc}\Vert f \Vert.
\]

To show the convergence, note that for fixed $k$ and $f$, since $\Fou(f)\in c_0$ we have
$
\|\lim_{m\to \infty}\GG_k(\UU_m(f))\|=0.
$
Let us assume that there is $f\in \XX$ for which $\UU_m(f)$ does not converge to zero, and pick $0<\delta<\limsup_m \Vert \UU_m(f)\Vert$. We can  recursively construct an increasing sequence  of integers $(m_j)_{j=0}^\infty$ with $m_0=0$ and such that $\Vert g_j\Vert \ge \delta$ for all $j\in\NN$, where, if we denote $A_{m_j}(f)$ by $B_j$,
\[
g_j=\UU_{m_j}(f)- \GG_{m_{j-1}}(\UU_{m_j}(f))= \min_{n\in B_j} |\xx_n^*(f)| \sum_{n\in B_j \setminus B_{j-1}} \sgn(\xx_n^*(f)) \, \xx_{n}.
\]
Using again that $\BB$ is LUCC, we obtain
\[
\|\GG_{m_{j}}(f)-\GG_{m_{j-1}}(f)\|
\ge\frac{1}{K_{lc}} \Vert g_j \Vert
\ge \frac{\delta}{K_{lc}}
\]
for every $j\in\NN$. This implies that $(\GG_m(f))_{m=1}^\infty$ does not converge, which  contradicts Theorem \ref{PW12.thm1}.
\end{proof}

\begin{proposition}\label{prop:qg6}
Let $\BB$ be a semi-normalized $M$-bounded basis of a quasi-Banach space $\XX$. Then $\BB$ is quasi-greedy if and only if $\BB$ is QGLC and the 
 truncation operators $(\TT_m)_{m=0}^\infty$ are uniformly bounded.
Moreover, if $\XX$ is $p$-Banach, then $\Lambda_t\le C_{qg} (1+ C_{qg}^{p} \eta_p^p(C_{qg}))^{1/p}$ and $C_{qg}\le 2^{1/p} C_{ql} \Lambda_t$.
\end{proposition}

\begin{proof}The ``only if part'' follows from Lemma~\ref{lem:qg9}, the identity  $\TT_m=\UU_m+\HH_{m}$, and Theorem~\ref{thm:qg5}.
Assume that the truncation operator is uniformly bounded and that $\BB$ is QGLC. Let $f\in\XX$ and let $A$ be a greedy set of $f$. Put  $t=\min_{n\in A} |\xx_n^*(f)|$ and $\varepsilon=(\sgn(\xx_n^*(f))_{n\in A}$. Then 
\[
\Vert f-S_A(f)\Vert \le C_{ql} \Vert f-S_A(f)+ t \Ind_{\varepsilon,A}\Vert
=C_{ql}\Vert \TT(f,A)\Vert
\le C_{ql} \Lambda_t \Vert f\Vert.
\]
Hence, if $\XX$ is $p$-Banach, $C_{qg}\le 2^{1/p} C_{ql} \Lambda_t$.
\end{proof}

From Theorems~\ref{PW12.thm1} and \ref{thm:qg5} we infier the following convergence properties.
\begin{corollary}
Suppose $\BB$ is a quasi-greedy basis in a quasi-Banach space $\XX$. For $f\in \XX$,
\[
\lim_{m\to \infty} \GG_m(f)=f\]
 and 
 \[\lim_{m\to \infty} \UU_m(f)=\lim_{m\to \infty} \TT_m(f)=\lim_{m\to \infty} \HH_m(f)=0.\]
\end{corollary}
\begin{proof}It is straightforward from the relations $\HH_m=\Id_\XX-\GG_m$ and $\TT_m=\UU_m+\HH_{m}$.
\end{proof}

\begin{proposition}\label{prop:qg7}
Let $\BB$ be an  $M$-bounded semi-normalized basis of a quasi-Banach space $\XX$. Suppose that $(\UU_m)_{m=0}^\infty$ are uniformly bounded. Then $\BB$ is QGLC and LPU. In  case that $\XX$ is $p$-Banach, \[
C_{ql}\le (1+\Lambda_u^p)^{1/p} \quad \text{and}\quad K_{pu}\le A_p \Lambda_u (1+\Lambda_u^p)^{1/p}.\]
\end{proposition}

\begin{proof} Let $f\in\XX$ and $A\subseteq\NN$ finite be such that $\supp f\cap A=\emptyset$ and $|\xx_n^*(f)|\le 1$ for all $n\in \NN$, and  let $\varepsilon\in\EE_A$.
Since $A$ is a greedy set of $f+\Ind_{\varepsilon,A}$,
\[
\Vert \Ind_{\varepsilon,A}\Vert = \Vert \UU(f+\Ind_{\varepsilon,A},A)\Vert\le \Lambda_u \Vert f+\Ind_{\varepsilon,A}\Vert.
\]
Hence, $\BB$ is QGLC. 

Let $(b_n)_{n\in A}\in[1,\infty)^A$. It is clear that $A$ is a greedy set of $f=\sum_{n\in A} b_n \, \varepsilon_n \, \xx_n$, therefore 
\[
\Vert \Ind_{\varepsilon,A}\Vert \le \min_{n\in A} b_n \Vert \Ind_{\varepsilon,A}\Vert =\Vert \UU(f,A)\Vert \le \Lambda_u \Vert f \Vert,
\]
and $\BB$ is LUCC. 

The estimates for the constants $C_{ql}$ and $K_{pu}$ follow from Proposition~\ref{prop:ucc4} and Lemma~\ref{lem:QGtoSUCC}.
\end{proof}

We close this section exhibiting a new unconditionality-type property enjoyed by quasi-greedy bases.

\begin{theorem}\label{thm:qgunc}Let $\BB=(\xx_n)_{n=1}^\infty$ be a quasi-greedy basis of a quasi-Banach space $\XX$. Then there is a positive constant $C<\infty$ such that 
\begin{equation}\label{eq:qgunc} 
\Vert \Ind_{\varepsilon,A}[\BB,\XX]\Vert \le C \Vert \Ind_{\varepsilon,A}[\BB,\XX] + f\Vert 
\end{equation}
for every finite subset $A$ of $\NN$, every $\varepsilon\in\EE_A$ and every $f\in\XX$ with $\supp(f)\cap A=\emptyset$.
\end{theorem}

\begin{proof}Assume that $\XX$ is $p$-Banach for some $0<p\le 1$. Set $B=\{ n\in\NN \colon |\xx_n^*(f)|>1\}$. Since both $B$ and $A\cup B$ are greedy sets of
$g=\Ind_{\varepsilon,A}+f$, 
\begin{align*}
\Vert \Ind_{\varepsilon,A} \Vert^p 
&\le \Vert \Ind_{\varepsilon,A}+S_B(f) \Vert^p +\Vert S_B(f) \Vert^p\\
&=\Vert S_{A\cup B} (g) \Vert^p +\Vert S_B(g) \Vert^p\\
&\le 2 C_{qg}^p \Vert g \Vert^p.
\end{align*}
That is, \eqref{eq:qgunc} holds with $C=2^{1/p} C_{qg}$.
\end{proof}

\begin{remark}\label{rmk:qguncIf} Suppose $\BB=(\xx_n)_{n=1}^\infty$ is a basis of a quasi-Banach space $\XX$. Given  $A\subseteq\NN$ finite, the quotient map 
\[
P_A\colon [\xx_n \colon n\in A] \to \XX/[\xx_n \colon n\notin A], \quad f\mapsto f+[\xx_n \colon n\notin A]
\]
is an isomorphism. Note that the basis $\BB$ is unconditional if and only if  $\sup \{ \Vert P_A^{-1}\Vert \colon A\text{ finite}\}<\infty$, that is
\[
\Vert f \Vert \approx \Vert P_A(f)\Vert, \quad f\in[\xx_n \colon n\in A].
\]
Theorem~\ref{thm:qgunc} yields that if $\BB$ is quasi-greedy (even in the case that is conditional),
\[
\Vert \Ind_{\varepsilon, A} \Vert \approx \Vert P_A( \Ind_{\varepsilon, A}) \Vert
\]
for $A\subseteq \NN$ finite and $\varepsilon\in\EE_A$.
\end{remark}

\begin{corollary}
Let $\BB=(\xx_n)_{n=1}^\infty$ be a quasi-greedy basis of a Banach space $\XX$. Then, for $A\subseteq \NN$ finite and $\varepsilon\in\EE_A$,
\[
\Vert \Ind_{\varepsilon,A}[\BB,\XX]\Vert \approx \{ \max\{ f^*( \Ind_{\varepsilon,A}[\BB,\XX]) \colon f^*\in B_{\XX^*}, \, \supp(f^*)\subseteq A\}.
\]
\end{corollary}

\begin{proof} Let $\YY=[\xx_n \colon n\notin A]$ and $\VV=\{ f^*\in \XX^* \colon \supp(f^*)\subseteq A\}$. Since the dual space of ${ \XX}/\YY$ is naturally isometric to the set
\[
\{ f^*\in \XX^* \colon f^*(\xx_n)=0 \text{ for all }n\in\NN\setminus A\}=\VV,
\]
we have 
\[
\Vert \Ind_{\varepsilon,A} + \YY \Vert=\sup\{ |f^*(\Ind_{\varepsilon,A})| \colon f^*\in \VV, \Vert f^*\Vert \le 1\},
\] 
and   Remark~\ref{rmk:qguncIf} finishes the proof. 
\end{proof}

\section{Democratic properties of bases}\label{Sec6}
\noindent
Let us consider the following condition 
\begin{equation}\label{Equation:Democracy}
\left\Vert f +\Ind_{\varepsilon,A}
\right\Vert \le C \left\Vert f +\Ind_{\delta,B}
\right\Vert,
\end{equation}
which involves a basis $\BB$ of a quasi-Banach space $\XX$, a constant $C$, two finite subsets $A$, $B$ of $\NN$, two collection of signs $\varepsilon\in \EE_A$ and $\delta\in\EE_B$, and a vector $f\in \XX$.

A basis $\BB$ of $\XX$ is said to be \emph{democratic} if there is $1\le C<\infty$ such that \eqref{Equation:Democracy} holds with $f=0$ and $\varepsilon=\delta=1$ whenever $|A|\le |B|$. The smallest constant $C$ in \eqref{Equation:Democracy} will be denoted by $\Delta=\Delta[\BB,\XX]$ . By imposing the additional assumption $A\cap B=\emptyset$ we obtain an equivalent definition of democracy, and $\Delta_d=\Delta_{d}[\BB,\XX]$ will denote the optimal constant under the extra assumption on disjointness of sets.

In turn, a basis $\BB$ is said to be \emph{super-democratic} if there is $1\le C<\infty$ such that \eqref{Equation:Democracy} holds with $f=0$ for every $A$ and $B$ with $|A|\le |B|$ and every choice of signs $\varepsilon\in \EE_A$ and $\delta\in\EE_B$. Again, by imposing the extra assumption $A\cap B=\emptyset$ we obtain an equivalent definition of super-democracy, and $\Delta_{sd}=\Delta_{sd}[\BB,\XX]$ will denote the optimal constant under this extra assumption.

Finally, a basis $\BB$ is said to be \emph{symmetric for largest coefficients} (SLC for short) if there is constant $1\le C<\infty$ such that \eqref{Equation:Democracy} holds for all $A$ and $B$ with $|A|\le|B|$ and $A\cap B=\emptyset$, all choices of signs $\varepsilon\in \EE_A$ and $\delta\in\EE_B$, and all $f\in\XX$ such that $\supp(f)\cap(A\cup B)=\emptyset$ and $|\xx_n^*(f)|\le 1$ for  $n\in\NN$. We will denote by $\Gamma=\Gamma[\BB,\XX]$ the optimal constant $C$.

The following result is well-known in the Banach space setting (see \cite{AA2017}*{Remark 2.6}).
\begin{proposition}\label{prop:ucc3} A basis is super-democratic if and only if it is democratic and suppression unconditional for constant coefficients. Moreover, $\max\{ K_{sc},\Delta\} \le \Delta_{s}$ and, if $\XX$ is $p$-Banach, $\Delta_s \le B_p^2\Delta K_{sc}$.
\end{proposition}

\begin{proof}If a basis $\BB$ is super-democratic, then by definition it is SUCC and democratic. On the other hand if $\BB$ is SUCC and democratic, then for $A$ and $B$ subsets of $\NN$ with $|A|\le |B|<\infty$, Corollary~\ref{cor:convexity2} yields
\[
\left\Vert \sum_{n\in A}a_n\, \xx_n\right\Vert
\le B_p \Delta \Vert \Ind_B\Vert,
\]
for any scalars $(a_{n})_{n\in A}$ with $ |a_n|\le 1$ for all $n$. Moreover, by Lemma~\ref{lem:ucc2}, for every $\varepsilon\in \EE_B$  we have
\[
\Vert \Ind_B\Vert\le B_p K_{sc} \Vert \Ind_{\varepsilon,B}\Vert.
\]
Hence $\BB$ is super-democratic with $\Delta_s \le B_p^2 \Delta K_{sc}$.
\end{proof}

Next we shall pay close attention to the property of symmetry for largest coefficients.
\begin{lemma}[cf. \cite{AA2017}*{Proposition 3.7}]\label{lem:AA}
For a basis $\BB$ of a quasi-Banach space $\XX$ the following are equivalent: 
\begin{itemize}
\item[(i)] $\BB$ is symmetric for largest coefficients.
\item[(ii)] There is a constant $1\le C<\infty$ such that 
\begin{equation}\label{eq:paa}
\left\Vert f \right\Vert
\le C\left\Vert f-S_A(f) + t \Ind_{\varepsilon,B} \right\Vert
\end{equation}
for all sets $A$, $B$  with  $0 \le |A| \le |B| < \infty$ and $ \supp(f)\cap B = \emptyset$, all sings $\varepsilon\in\EE_B$, and all $t$ such that $|\xx_n^*(f)|\le t$ for every $n\in\NN$.

\item[(iii)] There is a constant $1\le C<\infty$ such that \eqref{eq:paa} holds for all sets $A$, $B$ with $0 \le |A| \le |B| < \infty$ and $( \supp(f)\setminus A)\cap B = \emptyset$, all signs $\varepsilon\in\EE_B$, and all $t$ such that $|\xx_n^*(f)|\le t$ for every $n\in\NN$.
\end{itemize}
Moreover, if $C_2$ is the optimal constant in (ii) and $C_3$ is the optimal constant in (iii), we have $\Gamma\le C_2 \le C_3$ and, in the case when $\XX$ is $p$-Banach, $C_2 \le A_p \Gamma$ and $C_3\le C_2\Gamma$.
\end{lemma}

\begin{proof}
In order to prove (ii) $\Rightarrow$ (i), pick sets $A,B \subset\NN$ with $|A|\le |B|<\infty$, signs  $\varepsilon\in \EE_A$, $\delta\in\EE_B$, and a vector $f\in\XX$ such that $A\cap B=\supp(f)\cap(A\cup B)=\emptyset$ and $|\xx_n^*(f)|\le 1$ for every $n\in\NN$. If $g=f+\Ind_{\varepsilon,A}$ we have
\[
\Vert g \Vert\le C\Vert g-S_A(g)+\Ind_{\delta,B}\Vert =C \Vert f+\Ind_{\delta,B}\Vert.
\]

To show the implication (i) $\Rightarrow$ (ii), assume that $\XX$ is $p$-Banach and let $f$, $t$, $A$, $B$ and $\varepsilon$ be as in (ii). If  $D\subseteq A$ and $\delta\in\EE_D$ we have
\[
\Vert f-S_A(f)+t \Ind_{\delta,D}\Vert
\le \Gamma \Vert f-S_A(f)+t \Ind_{\varepsilon,B} \Vert.
\]
By Corollary~\ref{cor:convexity}~(ii),
\[
\left\Vert f-S_A(f)+\sum_{n\in A} a_n\, \xx_n\right\Vert
\le A_p \Gamma \Vert f-S_A(f)+t \Ind_{\varepsilon,B} \Vert
\]
whenever $|a_n|\le t$ for every $n\in A$. Choosing $a_n=\xx_n^*(f)$ for $n\in A$ we are done.

(iii) $\Rightarrow$ (ii) is trivial so let us prove (ii) $\Rightarrow$ (iii). Assume that $\XX$ is $p$-Banach, and so the quasi-norm is continuous. Let $f$, $t$, $A$, $B$ and $\varepsilon$ be as in (iii). Pick $\epsilon>0$. By Lemma~\ref{lem:bases5}~(ii), there is $D\subseteq\NN$ such that $|D|=|B|$, and $\Vert S_E(f)\Vert \le \epsilon$ whenever $E\subseteq D$.
Taking into account the equivalence between (i) and (ii) we obtain
\begin{align*}
\Vert f -S_D(f)
&\Vert\le C\Vert f-S_D(f) -S_A(f-S_D(f))+ t \Ind_{D}\Vert\\
&= C\Vert f-S_{A\cup D} (f) + t \Ind_{D}\Vert\\
&\le C \Gamma \Vert f-S_{A\cup D} (f)+ t \Ind_{\varepsilon,B}\Vert\\
&= C \Gamma \Vert f-S_{A} (f)+t \Ind_{\varepsilon,B} - S_{D\setminus A}(f)\Vert.
\end{align*}
Since $\Vert S_D(f)\Vert \le\epsilon$ and $\Vert S_{D\setminus A}(f)\Vert \le\epsilon$, letting $\epsilon$ tend to $0$ yields 
\[ \Vert f \Vert\le C \Gamma \Vert f-S_{A} (f)+ t\Ind_{\varepsilon,B}\Vert.\qedhere\]
\end{proof}

Our next result conveys the idea  that symmetry for largest coefficients evinces an appreciable flavour of quasi-greediness,  hence of unconditionality.
\begin{proposition}\label{prop:qg1} A basis $\BB$ of a quasi-Banach space $\XX$ is SLC if and only if it is 
democratic and QGLC. Moreover, if $\XX$ is $p$-Banach, we have $C_{ql} \le A_p\Gamma$ and $\Gamma\le C_{ql} ( 1+A_p^p B_p^p C_{ql}^p \Delta^p)^{1/p}$.
\end{proposition}

\begin{proof}Assume that $\XX$ is a $p$-Banach space and that $\BB$ is a democratic QGLC basis. By Lemma~\ref{lem:QGtoSUCC} and Proposition~\ref{prop:ucc3}, $\BB$ is super-democratic with $\Delta_s \le A_p B_p C_{ql} \Delta$. Let $A$ and $B$ be subsets of $\NN$ with $|A|\le |B|<\infty$, let $\varepsilon\in \EE_A$ and $\delta\in\EE_B$, and let $f\in\XX$ be such that $\supp(f)\cap(A\cup B)=\emptyset$ and $|\xx_n^*(f)|\le 1$ for all $n\in\NN$. Then
\begin{align*}
\Vert f+\Ind_{\varepsilon,A} \Vert^p
&\le \Vert f\Vert^p + \Vert \Ind_{\varepsilon,A}\Vert^p\\
&\le \Vert f\Vert^p + \Delta_{sd}^p \Vert \Ind_{\delta,B} \Vert^p\\
&\le (1+\Delta_{sd}^p) C_{ql}^p\Vert f + \Ind_{\delta,B}\Vert^p.
\end{align*}
Conversely, assume that $\BB$ is SLC. By Lemma~\ref{lem:AA}, we can use \eqref{eq:paa} with $A=\emptyset$. This gives that $\BB$ is QGLC.
\end{proof}

 Every democratic basis is semi-normalized and
the democratic constants of a basis in a quasi-Banach space are related by the following inequalities:
\begin{align*}
\max\lbrace\Delta_{sd},\Delta\rbrace & \le \Gamma, \\
\Delta_{d}&\le \Delta\le \Delta_d^2, \\
\max\lbrace \Delta_{sd},\Delta\rbrace &\le \Delta_{s}\le \Delta_{sd}^2.
\end{align*}

 This section ends with a new addition to these estimates.
\begin{lemma}\label{lem:PabloEstimate}Let $\BB$ be a SLC basis of a $p$-Banach space $\XX$. Then $\Delta_s[\BB,\XX]\le B_p \, \Gamma[\BB,\XX]$.
\end{lemma}
\begin{proof}
Let $B$ be a finite subset of $\NN$ and let $\varepsilon\in\EE_B$. We choose an arbitrary sequence of signs indexed by $\NN$ containing $\varepsilon$ which, for simplicity, keep denoting by $\varepsilon$. If we apply the definition of SLC to $f= \Ind_{\varepsilon,D\cap B}$ we obtain
\[
\Vert \Ind_{\varepsilon,D} \Vert \le \Gamma \Vert \Ind_{\varepsilon,B} \Vert, \]
for any subset $D$ of $\NN$ or cardinality $|D|\le |B|$. Hence, given $A\subseteq\NN$ with $|A|\le |B|<\infty$,  by Corollary~\ref{cor:convexity2}, 
\[
\left\Vert \sum_{n\in A} a_n \, \varepsilon_n \, \xx_n\right\Vert \le B_p \, \Gamma \Vert \Ind_{\varepsilon,B} \Vert,
\] 
for all scalars $(a_{n})_{n\in A}$ with  $|a_n|\le 1$ for $n\in A$. In particular, we have $\Vert \Ind_{\delta,A} \Vert \le B_p \, \Gamma \Vert \Ind_{\varepsilon,B} \Vert$ for every $\delta\in\EE_A$.
\end{proof}

\subsection{Bidemocracy}
Let us introduce now another property of bases related to  democracy. Following \cite{DKKT2003} we say that a basis $\BB$ of a quasi-Banach space $\XX$ is \emph{bidemocratic} if there is a positive constant $C$ such that 
\begin{equation}\label{eq:bidem}
\Vert \Ind_A[\BB,\XX] \Vert \, \Vert \Ind_B[\BB^*,\XX^*] \Vert \le C\, m,
\end{equation}
for all $m\in\NN$ and all $A$, $B\subseteq\NN$ with $\max\{ |A|,|B|\}\le m$.
\begin{lemma}\label{lem:BDtoBSD}A basis $\BB$ of a quasi-Banach space $\XX$ is bidemocratic if and only if there is a positive constant $C$ such that 
\begin{equation}\label{eq:bisuperdem}
\Vert \Ind_{\varepsilon,A}[\BB,\XX] \Vert \, \Vert \Ind_{\delta,B}[\BB^*,\XX^*] \Vert \le C m
\end{equation}
for all $m\in\NN$, all $A$, $B\subseteq\NN$ with $\max\{ |A|,|B|\}\le m$, and all $\varepsilon\in\EE_A$ and $\delta\in\EE_B$. Moreover if $C_1$ is the optimal constant in \eqref{eq:bidem} and $C_2$ is the optimal constant in \eqref{eq:bisuperdem} we have $C_1\le C_2$ and, if $\XX$ is $p$-Banach, $C_2\le  B_1 B_p C_1$.
\end{lemma}
\begin{proof}Just apply Corollary~\ref{cor:convexity2}.
\end{proof}
We will denote by $\Delta_b[\BB,\XX]$ the optimal constant in \eqref{eq:bidem} and by $\Delta_{sb}[\BB,\XX]$  the optimal constant in \eqref{eq:bisuperdem}. Using this notation,
Lemma~\ref{lem:BDtoBSD} for $p$-Banach spaces reads as
\[ 
\Delta_b[\BB,\XX]\le \Delta_{sb}[\BB,\XX] \le  B_1 B_p \, \Delta_b[\BB,\XX].
\]

\begin{lemma}\label{lem:dualbidem} Suppose $\BB$ is a bidemocratic basis of a quasi-Banach space $\XX$. Then the dual basis $\BB^*$ is bidemocratic, and $\Delta_{sb}[\BB,\XX]\le \Delta_{sb}[\BB^*,\XX^*]$.
\end{lemma}

\begin{proof}It is a ready consequence of the inequality 
\[
\Vert \Ind_{\varepsilon,A}[\BB^{**},[\BB^*]^*] \Vert \le \Vert \Ind_{\varepsilon,A}[\BB,\XX] \Vert,
\]
which follows from \eqref{eq:bidualbasis}.
\end{proof}
The elementary identity
\[
\Ind_{\varepsilon,A}[\BB^*,\XX^*] \left( \frac{1}{m}\Ind_{\overline \varepsilon,A}[\BB,\XX] \right)=1, \quad A\subseteq\NN, \, |A|=m, \, \varepsilon\in\EE_A
\]
tells us that $1\le\Delta_b$ and that, roughly speaking, bidemocracy is the property that ensures that, up to a constant, the supremum defining  $\Vert \Ind_{\varepsilon,A}[\BB^*,\XX^*]\Vert$ is essentially attained on the vector $\Ind_{\overline \varepsilon,A}[\BB,\XX]$.

Bidemocratic bases were introduced in \cite{DKKT2003} with the purpose to investigate duality properties of the greedy algorithm. Recently, the authors of \cite{AAB2020} took advantage of the bidemocracy to achieve new estimates for the greedy constant of (greedy) bases. Bidemocracy will also play a significant role here in this paper. For the time being, we clarify the relation between bidemocracy and other democracy-like properties.

\begin{proposition}\label{prop:BDtoD+RTO}Let $\BB=(\xx_n)_{n=1}^\infty$ be a bidemocratic basis of a quasi-Banach space $\XX$. Then $\BB$ and $\BB^*$ are $M$-bounded  democratic  bases for which the restricted truncation operator is uniformly bounded. Quantitatively we have
\begin{equation*}
\Vert \Ind_{\delta,B}[\BB,\XX]\Vert \le \Delta_{sb} \Vert f\Vert\text{ and }\Vert \Ind_{\delta,B}[\BB^*,\XX^*]\Vert \le \Delta_{sb} \Vert f^*\Vert
\end{equation*}
for all $B\subseteq\NN$ finite, all $\delta\in\EE_B$, all $f\in\XX$ with 
$|\{n \in \NN \colon |\xx_n^*(f)|\ge 1\}|\ge |B|$, and all $f^*\in\XX^*$ with $|\{n \in \NN \colon |f^*(\xx_n)|\ge 1\}|\ge |B|$.
In particular, 
\[
\max\{ \Delta_s[\BB,\XX], \Delta_s[\BB^*,\XX^*], \Lambda_u[\BB,\XX], \Lambda_u[\BB^*,\XX^*] \} \le \Delta_{sb}.
\] 
\end{proposition}
\begin{proof}Let $A$ and $B$  be  finite subsets of $\NN$  with $|B|\le|A|$,  let  $\delta\in\EE_B$, and let  $f\in\XX$ be such that $|\xx_n^*(f)|\ge 1$ for all $n\in A$. If $\varepsilon=(\overline{\sgn(\xx_n^*(f))})_{n\in A}$ we have 
\begin{equation*}
\Vert \Ind_{\delta,B}[\BB,\XX]\Vert \le \Delta_b \frac{|A|}{\Vert \Ind_{\varepsilon,A}[\BB^*,\XX^*] \Vert}
\le \Delta_b \frac{ \Ind_{\varepsilon,A}[\BB^*,\XX^*](f) }{ \Vert \Ind_{\varepsilon,A}[\BB^*,\XX^*] \Vert}
\le \Delta_b \Vert f\Vert.
\end{equation*}
 Choosing $f=\Ind_{\varepsilon,A}$ we obtain that $\Delta_{s}\le \Delta_{sb}$.  Switching the roles of $\BB$ and $\BB^*$ we obtain that $\Delta_s[\BB^*,\XX^*]\le  \Delta_{sb}$. In particular, both $\BB$ and $\BB^*$ are semi-normalized. Thus, $\BB$ is $M$-bounded.

 Now, if we choose $B=A$ and $\delta=\overline\varepsilon$, we obtain that the restricted truncation operator is bounded by $ \Delta_b$. Switching again the roles of $\BB$ and $\BB^*$ we obtain the corresponding result for $\BB^*$.
\end{proof}

\begin{corollary}If $\BB$ is a bidemocratic basis then $\BB$ is LPU and SLC.
\end{corollary}

\begin{proof}Combine Proposition~\ref{prop:BDtoD+RTO} with Proposition~\ref{prop:qg7} and Proposition~\ref{prop:qg1}.
\end{proof}

\begin{corollary}\label{cor:dicotomy}If $\BB=(\xx_n)_{n=1}^\infty$ is a bidemocratic basis then either $\BB^*$ is equivalent to the unit vector system of $c_0$ or $\Fou^*(f^*)\in c_0$ for  all  $f^*\in\XX^*$.
\end{corollary}

\begin{proof}Suppose there is $g^*\in \XX^*$ such that $\Fou^*(g^*)\notin c_0$. Then the set $\{ n \in \NN \colon |g^*(\xx_n)|\ge c\}$ is infinite for some $c>0$. Therefore, by Proposition~\ref{prop:BDtoD+RTO}, 
\[
\Vert \Ind_{\delta,B}[\BB^*,\XX^*]\Vert \le \Delta_{sb} \frac{\Vert g^*\Vert}{c}
\]
for all $B\subseteq \NN$ finite and all $\delta\in\EE_B$. Hence, by Corollary~\ref{cor:convexity}~(ii), 
\[
\left\Vert \sum_{n\in B} a_n \, \xx_n^*\right\Vert \le A_p \Delta_{sb} \frac{\Vert g^*\Vert}{c},
\] 
for all $B\subseteq \NN$ finite and all $(a_{n})_{n\in B}$ with $|a_n|\le 1$ for  $n\in B$. Thus, the map $\II^*$ defined   as in \eqref{eq:InvFourier} with respect to the basic sequence $\BB^*$, restricts to a bounded linear map from $c_0$ into $[\BB^*]$. We deduce that  $\II^*$ is an isomorphism from $c_0$ onto $[\BB^*]$ whose inverse is $\Fou^*$.
\end{proof}

\section{Almost greedy bases}\label{Sec7}
\noindent
An $M$-bounded, semi-normalized basis $\BB$ for a quasi-Banach space $\XX$ is said to be \emph{almost greedy} if there is a constant $1\le C$ such that
\begin{equation}\label{eq:ag}
\Vert f- S_A(f)\Vert \le C \Vert f -S_B(f)\Vert, 
\end{equation}
whenever $A$ is a  finite greedy set of $f\in\XX$ and $B$ is another subset of $\NN$ with $|B|\le |A|$. The smallest constant $C$  in \eqref{eq:ag} will be denoted by $C_{ag}=C_{ag}[\BB,\XX]$. We draw attention to the fact our  definition is slightly different from the original definition of almost greedy basis given by Dilworth et al.\  in \cite{DKKT2003}, however it is equivalent.

\begin{lemma}[cf.\ \cite{AA2017}*{Theorem 3.3}]\label{lem:ag1} An $M$-bounded semi-normalized basis $\BB$ of a quasi-Banach space $\XX$ is almost greedy if only if there is a constant $C<\infty$ such that
\begin{equation}\label{eq:ag3}
\Vert f- \GG_m(f)\Vert \le C \Vert f -S_B(f)\Vert, \quad |B|=m, \, f\in\XX.
\end{equation}
Moreover, if the quasi-norm is continuous, the optimal constant in \eqref{eq:ag3} is  $C_{ag}$.
\end{lemma}
\begin{proof}
The proof   for the locally convex case from \cite{AA2017} works also in this case, and so we leave it as an exercise for the reader.
\end{proof}

\begin{lemma}\label{lem:ag2}An $M$-bounded semi-normalized basis $\BB=(\xx_n)_{n=1}^\infty$ of a quasi-Banach space $\XX$ is almost greedy if and only if there is a constant $C$ such that
\begin{equation}\label{eq:ag2}
\Vert f\Vert \le C\Vert f-S_B(f)+z\Vert
\end{equation}
for all vectors $f$, $z\in\XX$ and all subsets $B\subseteq\NN$ finite  such that $\supp(f)\cap \supp (z)=\emptyset$, $|B|\le|\supp(z)|$, and $\max_{n\in\NN} |\xx_n^*(f)|\le \min_{n\in\supp (z)} |\xx_n^*(z)|$. Moreover, the smallest constant $C$ in \eqref{eq:ag2} is the almost greedy constant of the basis $C_{ag}$.
\end{lemma}

\begin{proof}Let us first show the ``only if" part.  Suppose $\BB$ is almost greedy. Let $f$, $B$ and $z$ be as above. Without loss of generality we assume that $B\subseteq\supp(f)$. Since $z$ is a greedy sum of $f+z$,
\[
\Vert f\Vert =\Vert f+z-z\Vert\le C_{ag} \Vert f+z-S_B(f+z)\Vert =C_{ag}\Vert f-S_B(f)+z\Vert.
\]

To show the ``if" part, we pick $g\in\XX$,  a greedy set  $A$  of $g$, and   a subset  $B$ of $\NN$ with $|B|\le |A|$. Assume, without loss of generality, that $A\cap B=\emptyset$. If we apply our hypothesis to $f=g-S_A(g)$ and $z=S_A(g)$ we obtain
\[
\Vert g-S_A(g)\Vert =\Vert f\Vert 
\le C\Vert f-S_B(f)+z\Vert 
=C\Vert g-S_B(g)\Vert,
\]
so that $C_{ag}\le C$.
\end{proof}

The following theorem relates almost greedy bases to quasi-greedy bases and democratic bases. We would like to point out that, although the statement of Theorem~\ref{thm:ag3} is similar to the corresponding result for Schauder bases in Banach spaces by Dilworth et al.\ from \cite{DKKT2003}, our approach is substantially different. Indeed, the alternative route we follow in the proof has to overcome, on the one hand, the obstructions resulting from the non-uniform boundedness of the partial sum operators associated to the basis and the absence of local convexity of the underlying space, on the other.
\begin{theorem}\label{thm:ag3}Let $\BB$ be an $M$-bounded semi-normalized basis $\BB$ for a quasi-Banach space $\XX$. The following are equivalent.
\begin{itemize}
\item[(i)] $\BB$ is almost greedy.
\item[(ii)] $\BB$ is SLC and quasi-greedy.
\item[(iii)] $\BB$ is super-democratic and quasi-greedy.
\item[(iv)] $\BB$ is democratic and quasi-greedy.
\item[(v)] $\BB$ is SLC and the truncation operator is uniformly bounded.
\end{itemize}
Moreover $\Gamma \le C_{ag}$ and, given $ f\in \XX$ and $A$ a greedy set of $f$,
\begin{equation}\label{eq:ag8}
\Vert f -S_A(f)\Vert \le C_{ag}\Vert f \Vert.
\end{equation}
 In the particular case that $\XX$ is $p$-Banach  we also have
\[
C_{ag}\le A_p\Gamma \Lambda_t\quad \text{and}\quad C_{qg} \leq 2^{1/p} C_{ag}.
\]
\end{theorem}
\begin{proof}
(i) $\Rightarrow$ (ii) Combining Lemma~\ref{lem:AA} and Lemma~\ref{lem:ag2} gives $\Gamma\leq C_{ag}$. Choosing $B=\emptyset$ in the definition of almost greedy basis we get \eqref{eq:ag8}. Consequently, if $\XX$ is $p$-Banach, $C_{qg}\le 2^{1/p} C_{ag}$. 

(ii) $\Rightarrow$ (iii) $\Rightarrow$ (iv) are obvious, and (iv) $\Rightarrow$ (v) follows from Propositions~\ref{prop:qg1} and \ref{prop:qg6}. 

 (v) $\Rightarrow$ (i) Let us pick a finite greedy set $A\subseteq \NN$ of $f\in\XX$, and $B\subseteq\NN$ with $|B|\le |A|$. Without loss of generality we assume that $A\setminus B\not=\emptyset$. Let $t=\min\{ |\xx_n^*(f)| \colon n\in A\setminus B\}$ and $\varepsilon=(\sgn(\xx_n^*(f)))_{n\in A\setminus B}$. Since $|\xx_n^*(f)|\le t $ for all $n\in\NN\setminus A$, and $A\setminus B$ is a greedy set of $f-S_B(f)$, Lemma~\ref{lem:AA} gives
\begin{align*}
\Vert f-S_A(f)\Vert 
& \le A_p\Gamma \left\Vert f-S_A(f)-S_{B\setminus A}(f) +t \Ind_{\varepsilon,A\setminus B} \right\Vert \\
& = A_p\Gamma \left\Vert f-S_B(f)- S_{A\setminus B}(f) +t \Ind_{\varepsilon,A\setminus B} \right\Vert \\
&=A_p\Gamma\Vert \TT(f-S_B(f), A\setminus B)\Vert\\
&\le A_p\Gamma \Lambda_t \left\Vert f-S_B(f) \right\Vert,
\end{align*}
as desired. 
\end{proof}

Notice that, in light of Lemma~\ref{lem:ag2}, condition~(ii) in Lemma~\ref{lem:AA} is sort of an ``almost greediness for largest coefficients'' condition. Our next Corollary provides an strengthening of the characterization of almost greedy bases given in Lemma~\ref{lem:ag2}, and thus manifests that there is a characterization of almost greedy bases similar to the characterization of symmetry for largest coefficients stated in Lemma~\ref{lem:AA} (iii).
\begin{corollary}\label{cor:ag4}Let $\BB=(\xx_n)_{n=1}^\infty$ be an  $M$-bounded semi-normalized basis of a quasi-Banach space $\XX$. If $\BB$ is almost greedy then there is a constant $C$ such that
\eqref{eq:ag2} holds
whenever $f$, $z\in\XX$ and $B\subseteq\NN$ finite are such that 
\[
(\supp(f)\setminus B)\cap \supp (z)=\emptyset,
\]
$|B|\le|\supp(z)|$, and $\max_{n\in\NN} |\xx_n^*(f)|\le \min_{n\in\supp (z)} |\xx_n^*(z)|$. In case that $\XX$ is $p$-Banach, we can choose $C=A_p\Gamma^2 \Lambda_t$.
\end{corollary}

\begin{proof}Assume that $\XX$ is $p$-Banach. By hypothesis, $D:=\supp(z)$ is a greedy set of $g:=f-S_B(f)+z$.
Let $\varepsilon=(\sgn(\xx_n^*(z))_{n\in D}$ and $t= \min_{n\in D} |\xx_n^*(z)|$. By Lemma~\ref{lem:AA} and Theorem~\ref{thm:ag3},
\begin{equation*}
\Vert f \Vert \le A_p\Gamma^2 \Vert f -S_B(f) + t \Ind_{\varepsilon,D}\Vert= A_p\Gamma^2 \Vert \TT(g,D)\Vert
\le A_p\Gamma^2 \Lambda_t \Vert g\Vert,
\end{equation*}
where both $\Gamma$ and $\Lambda_t$ are finite.
\end{proof}

Next we introduce a lemma that is of interest for the purposes of Section~\ref{sec:renorming}. Note that its hyphotesis is precisely the condition obtained in Corollary~\ref{cor:ag4}.
\begin{lemma}\label{lem:ag5} Let $\BB$ be an  $M$-bounded semi-normalized basis of a quasi-Banach space $\XX$. Suppose there is a constant $C$ such that \eqref{lem:ag2} holds
whenever $f$, $z\in\XX$, and  $A\subseteq\NN$  is a finite set such that $
(\supp(f)\setminus A)\cap \supp (z)=\emptyset,$
$|A|\le|\supp(z)|$, and $\max_{n\in\NN} |\xx_n^*(f)|\le \min_{n\in\supp (z)} |\xx_n^*(z)|$.
Then $\Lambda_t\le C$.
\end{lemma}

\begin{proof}Let $f\in\XX$ and let $B$ be a greedy set of $f$.  Using the hypothesis with $f'=\TT(f,B)$, $A'=B$ and $z'= S_B(f)$ gives
\[
\Vert \TT(f,B)\Vert \le C \Vert S_{B^c}(f) + S_B(f)\Vert = C \Vert f\Vert.\qedhere
\]
\end{proof}

Theorem~\ref{thm:ag3} provides tight estimates for the almost greedy constant $C_{ag}$ from estimates for the SLC constant and the truncation operator constant. We complement this result with the following Proposition, which provides estimates for $C_{ag}$ in terms of the quasi-greedy constant of the basis and some constants related to its democracy.

\begin{proposition}\label{prop:AGEstimates}Let $\BB$ be an $M$-bounded semi-normalized basis for a $p$-Banach space $\XX$. 
Suppose that $\BB$ is quasi-greedy. 
\begin{itemize}
\item[(i)] If $\BB$ is democratic, then $\BB$ is almost greedy with
\[
C_{ag}\leq C_{qg} (1 + (A_p B_p \Delta_{d} C^2_{qg} \eta_p(C_{qg}) )^p)^{1/p}.
\]
\item[(ii)] If $\BB$ is super-democratic, then $\BB$ is almost greedy with
\[
C_{ag}\leq C_{qg} (1 + (A_p\Delta_{sd} \eta_p(C_{qg}))^p)^{1/p}.
\]
\item[(iii)] If $\BB$ is SLC, then $\BB$ is almost greedy with 
\[
C_{ag}\leq A_p\Gamma C_{qg} (1+ C_{qg}^{p} \eta_p^p(C_{qg}))^{1/p}.
\]

\item[(iv)] If $\BB$ is bidemocratic, then $\BB$ is almost greedy with 
\[
C_{ag}\leq \left( C_{qg}^p+\Delta_{sb}^p\right)^{1/p}.
\]
\end{itemize}
\end{proposition}

\begin{proof}Pick $f\in\XX$, a greedy set $A$ for $f$, and $B\subseteq\NN$ with $\vert B\vert=m$. Note that $A\setminus B$ is a greedy set of
$S_{B^c}(f)=f-S_B(f)$. Therefore
\begin{equation}\label{eq:c1-al2}
\max\{ \Vert S_{(A\cup B)^c}(f)\Vert, \Vert S_{A\setminus B} (f)\Vert \}\leq C_{qg}\Vert f-S_B(f)\Vert.
\end{equation}

Since $f-S_A(f)= S_{(A\cup B)^c}(f)+ S_{B\setminus A}(f)$ we have
\begin{equation}\label{eq:c1-al1}
\Vert f-S_A(f)\Vert^p \leq \Vert S_{(A\cup B)^c}(f)\Vert^p + \Vert S_{B\setminus A}(f)\Vert^p.
\end{equation}
Using the democracy, Corollary~\ref{cor:convexity2}, and Theorem~\ref{thm:QGtoLPU} we obtain
\begin{align}\label{eq:c1-al3}
\nonumber
\Vert S_{B\setminus A}(f)\Vert&\leq B_p\Delta_{d}\max_{n\in B\setminus A}\vert\xx_n^*(f)\vert\,\Vert\Ind_{A\setminus B}\Vert\\\nonumber
&\leq B_p\Delta_{d}\min_{n\in A\setminus B}\vert\xx_n^*(f)\vert\Vert\Ind_{A\setminus B} \Vert\\
&\leq B_p \Delta_{d} K_{pu} \Vert S_{A\setminus B} (f) \Vert.
\end{align}
Combining \eqref{eq:c1-al2}, \eqref{eq:c1-al1} and \eqref{eq:c1-al3} yields (i).

Let $\varepsilon =(\sgn(\xx_n^*(f)))_{n\in A\setminus B}$. Using the super-democracy of the basis, Corollary~\ref{cor:convexity}~(ii) and Theorem~\ref{thm:QGtoLUCC} we obtain
\begin{align}\label{c2-al1}
\nonumber
\Vert S_{B\setminus A}(f)\Vert &\leq A_p\Delta_{sd}\max_{n\in B\setminus A}\vert\xx_n^*(f)\vert\Vert\Ind_{\varepsilon,A\setminus B}\Vert\\\nonumber
&\leq A_p\Delta_{sd}\min_{n\in A\setminus B}\vert\xx_n^*(f)\vert\Vert\Ind_{\varepsilon,A\setminus B}\Vert\\
&\leq A_p\Delta_{sd} K_{lc} \Vert S_{A\setminus B}(f)\Vert.
\end{align}
Combining \eqref{eq:c1-al2}, \eqref{eq:c1-al1} and \eqref{c2-al1} we obtain (ii).

(iii) follows from Theorem~\ref{thm:ag3} and Proposition~\ref{prop:qg6}.

If $\BB$ is bidemocratic, by Proposition~\ref{prop:BDtoD+RTO}, 
\begin{equation}\label{eq:c4-al1}
\Vert S_{B\setminus A}(f)\Vert \le \Delta_{sb} \Vert f\Vert.
\end{equation}
Putting together \eqref{eq:c1-al1} with \eqref{eq:c1-al2} and \eqref{eq:c4-al1} yields (iv).
\end{proof}

\subsection{Almost greediness and bidemocracy} 
Proposition~\ref{prop:AGEstimates} hints at the fact that combining quasi-greediness with bidemocracy yields better asymptotic estimates (i.e., estimates for large values of the constants) than combining quasi-greediness with other democracy-related properties. To be precise, the estimate achieved in Proposition~\ref{prop:AGEstimates}~(iv) grows linearly with $C_{qg}$ and $\Delta_{sb}$. Next we show that bidemocracy plays also an important role when aiming at qualitative results.

Let $\BB=(\xx_n)_{n=1}^\infty$ be an $M$-bounded semi-normalized basis of a quasi-Banach space. Then, by Proposition~\ref{prop:basesdual}~(v), the sequence $\BB^*$ is also  semi-normalized and $M$-bounded. Moreover, by \eqref{eq:dualcoordinateoperator}, a set $A\subseteq\NN$ is a greedy set of $f^*\in[\BB^*]$ if and only if 
\begin{equation}\label{eq:dualgreedyset}
|f^*(\xx_n)|\ge |f^*(\xx_k)|, \quad n\in A, \, k\in\NN\setminus A.
\end{equation}
Since condition \eqref{eq:dualgreedyset} makes sense for every functional $f^*$ we can safely extend the definition of greedy set and greedy projection to the whole space $\XX^*$. We point out that greedy projections with respect to non-complete biorthogonal systems have appeared in the literature before. Important examples are the wavelet bases in the (highly non-separable) space $\mathrm{BV}(\RR^d)$ for $d\ge 2$ (see \cites{CDPX1999,Wo2002}). In any case, we must be aware that since we cannot guarantee that $\Fou^*(f^*)\in c_0$,  the existence of finite greedy sets of $f^*$ is not ensured either. 

\begin{theorem}\label{thm:BDtoQG}Let $\BB=(\xx_n)_{n=1}^\infty$ be a bidemocratic quasi-greedy basis of a quasi-Banach space $\XX$. 
Then there is a positive constant $C$ such that 
\begin{equation}\label{eq:QGGeneralized}
\Vert S_A^*(f^*) \Vert \le C\Vert f^*\Vert 
\end{equation} 
for all $f^*\in\XX^*$ and all finite greedy sets $A$ of $f^*$. In fact,  \eqref{eq:QGGeneralized} holds with $C= 2\Delta_b[\BB,\XX]+C_{qg}[\BB,\XX]$. 
\end{theorem}

\begin{proof}The proof we present here is inspired by that of \cite{DKKT2003}*{Theorem 5.4}. Let $f^*\in\XX^*$ and let $A$ be a greedy set of $f^*$. For  $f\in\XX$, pick a greedy set $B$ of $f$ with $|A|=|B|=m$. If $\varepsilon=(\sgn(\xx_n^*(f)))_{n\in B\setminus A}$, invoking Proposition~\ref{prop:BDtoD+RTO} we have 
\begin{align*}
|S_{B}^*(f^*)(S_{A^c}(f))|
&=\left|\sum_{n\in B\setminus A} f^*(\xx_n) \, \xx_n^*(f)\right|\\
&\le\min_{n\in A} | f^*(\xx_n)| \sum_{n\in B\setminus A} |\xx_n^*(f)|\\
&=\min_{n\in A} | f^*(\xx_n)| \,  \Vert\Ind_{\varepsilon,B\setminus A}[\BB^*,\XX^*](f) \Vert\\
&\le \min_{n\in A} | f^*(\xx_n)| \, \Vert \Ind_{\varepsilon,B\setminus A}[\BB^*,\XX^*]\Vert \, \Vert f\Vert \\
&\le \Delta_b \Vert f^* \Vert \, \Vert f\Vert.  
\end{align*} 
Switching the roles of $\BB$ and $\BB^*$ we obtain 
\[ 
|S_{A}^*(f^*)(S_{B^c}(f))|\le \Delta_b \Vert f^* \Vert \, \Vert f\Vert. 
\] 
A straightforward computation yields 
\begin{align*}\label{eq:DKKTIdea} 
S_{A^c}^*(f^*)(S_B(f))- S_{A}^*(f^*)(S_{B^c}(f)) 
&=S_{A^c}^*(f^*)(f) -f^*(S_{B^c}(f))\\ 
&= f^*(S_{B}(f))- S_{A}^*(f^*)(f).\\
\end{align*} 
Summing up we deduce that 
\[
\max\{ |S_{A^c}^*(f^*)(f)|, | S_{A}^*(f^*)(f)|\}
\le (2\Delta_b+C_{qg}) \Vert f^*\Vert \, \Vert f \Vert,
\]
Taking the supremum on $f\in B_\XX$ we obtain the desired inequality.
\end{proof}

\begin{corollary}\label{cor:BDtoQG}Let $\BB$ be a bidemocratic quasi-greedy basis of a quasi-Banach space. Then $\BB^*$ is an almost greedy basis.
\end{corollary}

\begin{proof}Restricting the inequality \eqref{eq:QGGeneralized} provided by Theorem~\ref{thm:BDtoQG} to $[\BB^*]$ yields that $\BB^*$ is quasi-greedy. By Proposition~\ref{prop:BDtoD+RTO}, $\BB^*$ is  democratic as well. An appeal to Theorem~\ref{thm:ag3} puts an end to the proof.
\end{proof}

\section{Greedy bases}\label{Sec8}
\noindent
An  $M$-bounded semi-normalized basis $\BB=(\xx_n)_{n=1}^\infty$ for a quasi-Banach space $\XX$ is \emph{greedy} if there is a constant $1\le C$ such that for any $f\in \XX$,
\begin{equation}\label{eq:greedy}
\Vert f- S_A(f)\Vert
\le C \Vert f-z\Vert
\end{equation}
whenever $A$ is a  finite greedy set of $f\in\XX$, and $z\in\XX$ satisfies $|\supp(z)|\le |A|$. The smallest admissible constant $C$ in \eqref{eq:greedy} will be denoted by $C_g[\BB,\XX]=C_g$, and will be referred as the \emph{greedy constant} of the basis. 

As for quasi-greedy basis, a standard perturbation argument yields that a basis is greedy if and only if 
\begin{equation}\label{eq:greedy2}
\Vert f- \GG_m(f)\Vert 
\le C \Vert f- z \Vert,
\end{equation}
for all $f$ and $z$ in $\XX$ and all $m\in\NN$ such that $|\supp(z)|\le m$. If the quasi-norm in $\XX$ is continuous, the optimal constant in \eqref{eq:greedy2} coincides with $C_g$.

As we advertised in the Introduction, our goal in this section is to corroborate the characterization by Konyagin and Temlyakov of greedy bases in terms of the unconditionality and the democracy of the basis. The doubts in that regard that made Tribel write (see \cite{Triebel2008}*{Proof of Theorem 6.51}): 
\begin{quote} \emph{It is not immediately clear weather a greedy basis in a quasi-Banach space is also unconditional} 
\end{quote}
will dissipate now. With the aid of Theorem~\ref{cu}, the original proof from \cite{KoTe1999} permits to adapt to the non-locally convex setting the standard techniques used in \cite{KoTe1999} for Banach spaces (see \cite{Wo2000}*{Theorem 4}). For the sake of completeness, here we shall revisit this characterization, paying close attention to obtaining tight estimates for the constants involved. As a matter of fact, the techniques we develop provide estimates that extend the ones previously known for Banach spaces (see \cites{KoTe1999,DOSZ2011, DKOSZ2014, AA2017, BBG2017,AW2006, AAB2020}).

\begin{theorem}\label{thm:chg} Suppose $\BB$ is a semi-normalized $M$-bounded basis of a quasi-Ba\-nach space $\XX$. Then, $\BB$ is greedy if and only if
it is unconditional and almost greedy, with 
\[
C_{ag} \le C_g\le C_{ag} K_{su}.
\]
Moreover, if the quasi-norm is continuous, $K_{su}\le C_g$.
\end{theorem}

\begin{proof} It is obvious that $C_{ag}\le C_g$. To prove the right hand-side inequality, let $A$ be a  finite greedy set of $f\in \XX$ and fix $z\in\XX$ with $|\supp(z)|\le |A|$. Then, if the quasi-norm is continuous and we put
$B=\supp(z)$, we have
\[
\Vert f-S_A(f)\Vert \le C_{ag} \Vert f-S_B(f)\Vert = C_{ag} \Vert S_{B^c}(f-z)\Vert \le C_{ag} K_{su} \Vert f -z \Vert.
\]

In order to prove the upper bound for $K_{su}$, we take $f\in \XX$ with $\supp(f)=B$ finite and a subset $A$ of $\NN$. 
To estimate $\Vert S_A(f)\Vert$ we assume without loss of generality that $A\subseteq B$.  Pick $t \ge\sup_{n\in A}\vert \xx_n^*(f)\vert+\sup_{n\in B\setminus A}\vert\xx_n^*(f)\vert$ and let $g=t \Ind_{B\setminus A}+f$. The set $B\setminus A$ is a greedy set of $g$ and so \begin{equation*} 
\Vert S_{A}(f)\Vert = \Vert g -S_{B\setminus A}(g) \Vert \leq C_g\left\Vert g-t \Ind_{B\setminus A}\right\Vert=C_g\Vert f\Vert.
\end{equation*}
If $\Vert \cdot\Vert$ is continuous, we extend by continuity this inequality to any $f\in\XX$.\qedhere
\end{proof}

\begin{theorem}\label{thm:chg2}Let $\BB$ be an unconditional semi-normalized basis for a $p$-Banach space $\XX$.
\begin{itemize}
\item[(i)] If $\BB$ is  democratic, then $\BB$ is greedy with
\[
C_g\leq (K_{su}^p + \Delta_d^pK_u^{p} \min\lbrace B_p^p,K_u^p\rbrace)^{1/p}.
\]
\item[(ii)] If $\BB$ is super-democratic, then $\BB$ is greedy with
\[
C_g\leq( K_{su}^p+\Delta_{sd}^p \min\lbrace A_p^p , K_u^p\rbrace\min\lbrace A_p^p K_{su}^p, K_u^p\rbrace)^{1/p}.
\]
\item[(iii)]If $\BB$ is symmetric for largest coefficients, then $\BB$ is greedy with
\[
C_g\leq \min\lbrace A_p^2\Gamma K_{su}, A_p \Gamma K_u\rbrace .
\]

\item[(iv)]If $\BB$ is bidemocratic, then $\BB$ is greedy with 
\[
C_{g}\leq \left( K_{su}^p+\Delta_{sb}^p\right)^{1/p}.
\]
\end{itemize}
\end{theorem}

\begin{proof}
Let $A$ be a  finite greedy set of $f\in\XX$ and let $z=\sum_{n\in B}a_n\, \xx_n$ with $\vert B\vert=|A|$. Notice that
\begin{equation}\label{c1-g}
\max_{n\in B\setminus A}\vert\xx_n^*(f)\vert\le\min_{n\in A\setminus B}\vert\xx_n^*(f)\vert
=\min_{n\in A\setminus B}\vert\xx_n^*(f-z)\vert.
\end{equation}
Since
\[
f- S_A(f) = S_{(A\cup B)^c}(f-z) + S_{B\setminus A}(f)
\] we have
\begin{align}\label{c2-g}
\nonumber
\Vert f- S_A(f)\Vert^p 
& \le\Vert S_{(A\cup B)^c}(f-z)\Vert^p + \Vert S_{B\setminus A}(f)\Vert^p\\
& \le K_{su}^p\Vert f-z\Vert^p+\Vert S_{B\setminus A}(f)\Vert^p.
\end{align}
Using the unconditionality and the democracy of the basis we obtain
\begin{align}\label{c3-g1}
\nonumber
\Vert S_{B\setminus A}(f)\Vert&\leq K_u\max_{n\in B\setminus A}\vert\xx_n^*(f)\vert \, \Vert \Ind_{B\setminus A}\Vert\\
&\leq K_u\Delta_d\max_{n\in A\setminus B}\vert\xx_n^*(f)\vert\Vert \, \Ind_{A\setminus B}\Vert,
\end{align}
while if we merely use the democracy of the basis, by Corollary~\ref{cor:convexity2} we obtain
\begin{equation}\label{c4-g1}
\Vert S_{B\setminus A}(f)\Vert\leq B_p \Delta_d\max_{n\in B\setminus A}\vert\xx_n^*(f)\vert\, \Vert \Ind_{A\setminus B}\Vert.
\end{equation}
Since
\[
\min_{n\in A\setminus B}\vert\xx_n^*(f-z)\vert\,\Vert \Ind_{A\setminus B}\Vert\le K_u\Vert f-z\Vert,
\]
combining \eqref{c1-g}, \eqref{c2-g}, \eqref{c3-g1} and \eqref{c4-g1} gives (i).

Let $\varepsilon = (\sgn(\xx_j^*(f-z)))_{n\in A\setminus B}$ and $\delta = (\sgn(\xx_j^*(f)))_{n\in B\setminus A}$. Using the unconditionality and the super-democracy of the basis we obtain
\begin{align}\label{c3-g2}
\nonumber
\Vert S_{B\setminus A}(f)\Vert&\leq K_u \max_{n\in B\setminus A}\vert\xx_n^*(f)\vert\, \Vert \Ind_{\delta,B\setminus A}\Vert\\
&\leq K_u \Delta_{sd} \max_{n\in B\setminus A}\vert\xx_n^*(f)\vert
\Vert \Ind_{\varepsilon,A\setminus B} \Vert,
\end{align}
while if we only appeal to the super-democracy of the basis, using Corollary~\ref{cor:convexity}~(i) we obtain
\begin{equation}\label{c4-g2}
\Vert S_{B\setminus A}(f)\Vert\leq A_p \Delta_{sd} \max_{n\in B\setminus A}\vert\xx_n^*(f)\vert \, \Vert \Ind_{\varepsilon,A\setminus B}\Vert.
\end{equation} 
 For $n\in A\setminus B$ let $\lambda_n\in[0,1]$ be such that  
$\min_{k\in A\setminus B}\vert\xx_k^*(f-z)\vert=\lambda_n  \vert\xx_n^*(f-z)\vert$.
 Proposition~\ref{prop:EstSupUnc} yields
\begin{equation}\label{c5-g2}
\min_{n\in A\setminus B}\vert\xx_n^*(f-z)\vert\,\Vert \Ind_{\varepsilon,A\setminus B} \Vert
= \left\Vert \sum_{n\in A\setminus B}\lambda_n \xx_n^*(f-z)\xx_n\right\Vert
\leq A_p K_{su}\Vert f-z\Vert.
\end{equation}
Combining  inequalities \eqref{c1-g}, \eqref{c2-g}, \eqref{c3-g2}, \eqref{c4-g2} and \eqref{c5-g2} gives (ii).

 For (iii), let $A$, $B$ and $z$ be as before, $\delta=(\sgn(\xx_n^*(f))_{n\in B}$ and $t=\min_{n\in B}\vert\xx_n^*(f)\vert$. By Lemma~\ref{lem:AA},
\begin{equation*}
\Vert f -S_A(f)\Vert
\leq A_p\Gamma \Vert f-S_A(f)-S_{B\setminus A}(f)+ t \Ind_{\delta,A\setminus B}\Vert
=A_p\Gamma\Vert g \Vert,
\end{equation*}
where $ g= f-S_{A\cup B}(f)+ t \Ind_{\delta,A\setminus B}$. Using  that the basis is lattice unconditional  we obtain
$
\Vert g \Vert\le K_u \Vert f-z\Vert,
$
 while using  that the basis is suppression unconditional  and Proposition~\ref{prop:EstSupUnc} we obtain
$\Vert g \Vert\le A_p K_{su} \Vert f-z\Vert$.

If $\BB$ is bidemocratic we obtain   inequality \eqref{eq:c4-al1} as in the proof of Proposition~\ref{prop:AGEstimates}.
Combining this estimate with \eqref{c2-g} gives (iv).
\end{proof}

  From Theorems \ref{thm:chg} and \ref{thm:chg2}, disregarding the constants, we get:
  \begin{corollary}\label{thm:chgequi}  Suppose $\BB$ is a semi-normalized $M$-bounded basis of a quasi-Ba\-nach space $\XX$. The following conditions are equivalent:
\begin{itemize}
\item[(i)] $\BB$ is greedy.
\item[(ii)]  $\BB$ is unconditional and  democratic.
\item[(iii)]  $\BB$ is unconditional and  super-democratic.
\item[(iv)] $\BB$ is unconditional and symmetric for largest coefficients.

\item[(v)] $\BB$ is unconditional and  bidemocratic.
\end{itemize}
\end{corollary}

\section{The best greedy error versus the best almost greedy error}\label{Sec9}
\noindent
Given a basis $\BB$ of a quasi-Banach space, a vector $f\in\XX$ and $m\in\NN$, we put
\[
\sigma_m(f)
=\inf\left\{\left\Vert f-\sum_{n\in B}b_n\,\xx_n\right\Vert \colon |B|=m,\,b_n\in\FF\right\}\]
and
\[
\tilde\sigma_m(f)=\inf\{\left\Vert f - S_B(f) \right\Vert \colon |B|\le m\}.
\]
 
 An $M$-bounded semi-normalized basis is greedy if and only if there is a constant $C$ such that for every $f\in\XX$, every $m\in\NN$, and every greedy set $A$ of $f$ of cardinality $m$, 
\[\Vert f-S_A(f)\Vert \le C \sigma_m(f).\] The optimal contant $C$ is $C_g[\BB,\XX]$. Similarly, a basis is almost greedy if and only if there is a constant $C$ such that, if $m$, $f$ and $A$ are as above, 
\[\Vert f-S_A(f)\Vert \le C \tilde\sigma_m(f).\] The optimal constant $C$ equals $C_{ag}[\BB,\XX]$. Our first result in this section quantifies the distance between these two approximation errors in terms of the democracy functions of the basis. 

The \emph{upper democracy function} $\varphi_u$ and the \emph{lower democracy function} $\varphi_l$ of a basis $\BB$ are typically used to quantify the lack of democracy of $\BB$. For $m\in\NN$, 
\begin{align*}
\varphi_u(m)&=\varphi_{u}[\BB,\XX](m)=\sup\left\lbrace \left\Vert \Ind_{A} \right\Vert \colon |A|\le m \right\rbrace,\\
\varphi_l(m)&=\varphi_{l}[\BB,\XX](m)=\inf\left\lbrace \left\Vert \Ind_{A} \right\Vert \colon |A|\ge m \right\rbrace.
\end{align*}
We will also consider the \emph{upper super-democracy function} (also known as the \emph{fundamental function}) and the \emph{lower super-democracy function} of $\BB$, respectively defined by
\begin{align*}
\varphi^\EE_u(m)&=\varphi^\EE_{u}[\BB,\XX](m)=\sup\left\lbrace \left\Vert \Ind_{\varepsilon,A} \right\Vert \colon |A|\le m,\, \varepsilon\in\EE_A \right\rbrace,\\
\varphi^\EE_l(m)&=\varphi^\EE_{l}[\BB,\XX](m)=\inf\left\lbrace \left\Vert \Ind_{\varepsilon,A} \right\Vert \colon |A|\ge m,\, \varepsilon\in\EE_A \right\rbrace.
\end{align*}

All these sequences $\varphi_l$, $\varphi_u$, $\varphi_l^\EE$ and $\varphi^\EE_u$ are non-decreasing and are related by the inequalities 
\[
\varphi_l^\EE\le \varphi_l\le\varphi_u\le\varphi^\EE_u.
\] 

If $\XX$ is a $p$-Banach, axiom (q4) in the definition of a $p$-norm yields
\begin{equation}\label{eq:doubling}
\varphi_u^\EE(k m)\le k^{1/p} \varphi_u^\EE(m), \quad k,m\in\NN.
\end{equation}
Moreover,  if we restrict ourselves to semi-normalized bases, we have
\[
\varphi^\EE_u(1)=\sup_n \Vert \xx_n\Vert<\infty.
\] 
We infer that $\varphi^\EE_u$ takes only finite values, that $\varphi_u^\EE( m)$ is doubling and that,
\[
\varphi^\EE_u(m)\lesssim m^{1/p}, \quad m\in\NN.
\]
Recall that a sequence $(s_m)_{m=1}^\infty$ of positive numbers is {\em doubling} if there exists a constant $C$ such that  $s_{2m}\leq C s_m$ for $m\in \NN$. This  is equivalent to $s_{\lfloor \lambda m\rfloor}\leq C_\lambda s_m$ for every $\lambda>1$.

In the reverse direction,  if $\BB$ is semi-normalized and $M$-bounded, by Lemma~\ref{lem:bases1},  $c:=\sup_n \Vert \xx_n^*\Vert <\infty$.Thus
\[
\varphi^\EE_l(1)
\ge \frac{1}{c} \inf_{\substack{A\subseteq\NN \\ \varepsilon\in\EE_A  }}\sup_{n\in\NN} |\xx_n^*( \Ind_{\varepsilon,A})|
=\frac{1}{c}>0.
\]
However, $\varphi_l^{\EE}(m)$ need not  be doubling (see \cite{Wo2014}).

By definition, $\BB$ is democratic (resp. super-democratic) if and only if $\varphi_u\lesssim\varphi_l$ (resp. $\varphi_u^\EE\lesssim\varphi_l^\EE$), and we have
\[
\Delta=\sup_m \frac{\varphi_u(m)}{\varphi_l(m)} \quad \text{(resp. }\Delta_s=\sup_m \frac{\varphi_u^\EE(m)}{\varphi^\EE_l(m)}\text{)}.
\] 
Also by definition, a basis is bidemocratic if and only if 
\begin{equation}\label{eq:BDFundamental}
\varphi_u[\BB,\XX](m) \varphi_u[\BB^*,\XX^*](m)\approx m, \quad m\in\NN,
\end{equation}
and we have
\begin{align*}
\Delta_b[\BB,\XX] & =\sup_m \frac{1}{m} \varphi_u[\BB,\XX](m) \varphi_u[\BB^*,\XX^*](m),\\
\Delta_{sb}[\BB,\XX] & =\sup_m \frac{1}{m} \varphi_u^\EE[\BB,\XX](m) \varphi_u^\EE[\BB^*,\XX^*](m).
\end{align*} 

If $\XX$ is $p$-Banach, by Corollary~\ref{cor:convexity2},
\begin{equation}\label{analog}
\varphi^\EE_u\le B_p\varphi_u,
\end{equation}
that is, the sequence $\varphi^\EE_u$ is always equivalent to $\varphi_u$. By Corollary~\ref{cor:convexity}~(i),
\[
\varphi^\EE_u(m)\le A_p \sup\left\lbrace \left\Vert \Ind_{\varepsilon,A} \right\Vert \colon |A|= m,\, \varepsilon\in\EE_A \right\rbrace,
\]
and so the supremum defining $\varphi^\EE_u$ is essentially attained on sets of maximum cardinality.  

If $\BB$ is SUCC, so that democracy and super-democracy coincide, the supremum defining the sequences $\varphi_u$, $\varphi_l$ and $\varphi^\EE_l$ is also essentially attained on sets of maximum cardinality. Indeed, if $\XX$ is $p$-Banach, Proposition~\ref{lem:ucc2} gives 
\begin{align}
\label{eq:fundamental1}
\varphi_u(m)&\le K_{sc}\sup\left\lbrace \left\Vert \Ind_{A} \right\Vert \colon |A|= m \right\rbrace,\\
\label{eq:fundamental2}
\inf\left\lbrace \left\Vert \Ind_{A} \right\Vert \colon |A|= m \right\rbrace &\le K_{sc}\varphi_l(m),\\
\label{eq:fundamental3}
\inf\left\lbrace \left\Vert \Ind_{\varepsilon,A} \right\Vert \colon |A|= m,\, \varepsilon\in\EE_A \right\rbrace&\le A_p K_{sc} \varphi_l^\EE(m).
\end{align}
Moreover, in correspondence to \eqref{analog} we have 
\[
\varphi_l \le B_p K_{sc} \varphi_l^\EE.
\] 

Note that \eqref{eq:fundamental1}, \eqref{eq:fundamental2} and \eqref{eq:fundamental3} trivially hold for Schauder bases using just the basis constant.
\begin{proposition}\label{prop:gvsag} Given a quasi-greedy basis $\BB$ of a quasi Banach space $\XX$, there is $0<C<\infty$ such that for all $f\in\XX$ and all $r>m$,
\begin{equation}\label{eq:GvsAG9}
\tilde\sigma_r(f) \le C \max\left\lbrace 1, \frac{\varphi^\EE_u(m)}{\varphi^\EE_l(r-m)}\right\rbrace\sigma_m(f). 
\end{equation}
If $\XX$ is $p$-Banach we can choose $C= 2^{1/p} A_p C_{qg} \eta_p(C_{qg})$ in \eqref{eq:GvsAG9}.
\end{proposition}

\begin{proof}
Assume that $\XX$ is $p$-Banach. Let $f\in\XX$ and $g=f-\sum_{n\in B}b_n\,\xx_n$ with $|B|=m$. Pick a greedy set $A$ of $g$ with $|A|=r-m$. Since $S_D(f)=S_D(g)$ whenever $D\cap B=\emptyset$, we have
\[
f-S_{A\cup B}(f)
=S_{(A\cup B)^c}(g)=S_{A^c}(g)-S_{B\setminus A}(g)
\]
and, moreover, $|A\cup B|\le r$. Therefore 
\[
\tilde\sigma_r(f)\le 2^{1/p-1} \left( \Vert S_{A^c}(g)\Vert + \Vert S_{B\setminus A}(g)\Vert\right).
\]
Since $|\xx_k^*(g)|\le t:=\min_{n\in A} |\xx_n^*(g)|$ for all $k\in B\setminus A$, by 
Corollary~\ref{cor:convexity}~(ii) and Theorem~\ref{thm:QGtoLUCC},
\[
\Vert S_{B\setminus A}(g)\Vert \le A_p t \varphi_u^\EE(m) \le A_p K_{lc} \frac{ \varphi_u^\EE(m)}{ \varphi_l^\EE(r-m)} \Vert S_A(g)\Vert,
\]
where $K_{lc}\le C_{qg}\eta_p(C_{qg})$. The proof is over by using that 
\[
\max\lbrace \Vert S_A(g)\Vert, \Vert S_{A^c}(g)\Vert \rbrace\le C_{qg} \Vert g\Vert=\left\Vert f-\sum_{n\in B}b_n\,\xx_n \right\Vert
\]
and minimizing over $B$ and $(b_n)_{n\in B}$.
\end{proof}

\begin{theorem}[cf. \cite{DKKT2003}*{Theorem 3.3}] Suppose $\BB$ is an $M$-bounded semi-normalized basis of a quasi-Banach space $\XX$. Then $\BB$ is almost-greedy if and only if for every (respectively, for some) $\lambda>1$ there is a constant $C>0$ such that for every $m\in\NN$ and every greedy set $A$ of $f\in\XX$ of cardinality $\lceil \lambda m\rceil$,
\begin{equation}\label{eq:gvsag1}
\Vert f- S_A(f) \Vert \leq C \sigma_m(f).
\end{equation}
Moreover, if $\XX$ is $p$-Banach and we denote the optimal constant in \eqref{eq:gvsag1} by $C_\lambda=C_\lambda[\BB,\XX]$, we have 
\[
C_\lambda\le 2^{1/p} A_p B_p C_{ag}^3 \eta_p(C_{ag}) \lceil (\lambda-1)^{-1}\rceil^{1/p}.
\]
\end{theorem}

\begin{proof}Without loss of generality we assume that $\XX$ is $p$-Banach for some $0<p\le 1$. Let us suppose that $\BB$ is almost greedy and fix $\lambda>1$. If $k_\lambda=\lceil (\lambda-1)^{-1}\rceil$ we have 
\[
\lceil \lambda m \rceil \le k_\lambda (\lceil \lambda m \rceil -m)\] for all $m\in\NN$. Therefore, if $A$ is a greedy set of cardinality $\lceil \lambda m\rceil$, by Proposition~\ref{prop:gvsag},  inequality \eqref{eq:doubling}, Lemma~\ref{lem:PabloEstimate} and Theorem~\ref{thm:ag3},
\[
\Vert f-S_A(f)\Vert 
\le C_{ag} C k_\lambda^{1/p} \frac{\varphi^\EE_u( \lceil \lambda m \rceil -m)}{\varphi^\EE_l( \lceil \lambda m \rceil -m)}\sigma_m(f) 
\le B_p C_{ag}^2 C k_\lambda^{1/p} \sigma_m(f),
\]
where $C\le 2^{1/p} A_p C_{ag} \eta_p(C_{ag})$.

Now suppose that \eqref{eq:gvsag1} holds for some $\lambda>1$. By Lemma~\ref{lem:bases5}~(i),
\[
L:=\sup \{ \Vert S_A\Vert \colon |A|< \lambda\}<\infty.
\]

Given $r\in\NN\cup\{0\}$, pick $m=m(r)\in\NN\cup\{0\}$ such that 
\begin{equation}\label{eq:rm}
\lceil \lambda m \rceil \le r < \lceil \lambda (m+1) \rceil.
\end{equation}
Let $A$ be a greedy set of cardinality $r$. Pick $B$ a greedy set of $f$ of cardinality $\lceil \lambda m \rceil$ such that $B\subseteq A$. Since 
\[|A\setminus B|=r-\lceil \lambda m \rceil
< \lceil \lambda (m+1) \rceil-\lceil\lambda m \rceil\le\lambda\] and 
\[f-S_A(f)= f- S_B(f)-S_{A\setminus B}(f),\] it follows that
\begin{align*}
\Vert f- S_A(f)\Vert^p
&\le C^p
\Vert f- S_B(f)\Vert^p+ \Vert S_{A\setminus B}(f)\Vert^p\\
&\le C^p \sigma_m^p(f)+L^p \Vert f\Vert^p\\
&\le (C^p + L^p)\Vert f\Vert^p.
\end{align*}
Hence $\BB$ is quasi-greedy. In particular, by Lemma~\ref{lem:QGtoSUCC}, $\BB$ it is SUCC. Let us prove that $\BB$ is democratic. Let $m\in\NN$ and $A$ and $B$ be subsets of $\NN$ with $|A|=\lceil \lambda m \rceil$ and $|B|\le m$. Pick $D\subseteq \NN$ with $(A\cup B)\cap D=\emptyset$ and $|A\cap B|=|D|$. Since $E:=(A\setminus B)\cup D$ is a greedy set of $f=\Ind_{A\cup B\cup D}$, $|E|=|A|=\lceil \lambda m \rceil$, and $|D\cup(B\setminus A)|=|B|\le m$, we have
\[
\Vert \Ind_B\Vert= \Vert f -S_E(f)\Vert \le C \Vert f - S_{D\cup(B\setminus A)}(f)\Vert =C \Vert \Ind_A\Vert.
\]
Maximizing over $B$, minimizing over $A$, and  using \eqref{eq:fundamental1} we obtain
\[
\varphi_u(m)\le C K_{sc} \varphi_l(\lceil \lambda m \rceil), \quad m\in \NN.
\]
Let $r\ge\lceil \lambda \rceil$ and pick $m$ as in \eqref{eq:rm}. Since $m\ge 1$ we guarantee that
\[\lceil \lambda (m+1) \rceil\le \lceil 2 \lambda \rceil m.\] Therefore, using again \eqref{eq:doubling},
\[
\varphi_u(r)\le \lceil 2 \lambda \rceil^{1/p} \varphi_u(m) 
\le C K_{sc} \lceil 2 \lambda \rceil^{1/p} \varphi_l(\lceil \lambda m \rceil)
\le C K_{sc} \lceil 2 \lambda \rceil^{1/p} \varphi_l(r).
\]
Since $0<\varphi_l(1)$ and $\varphi_u(r-1)<\infty$, 
\[
\sup_r \frac{\varphi_u(r)}{\varphi_l(r)}<\infty.
\]
An appeal to Theorem~\ref{thm:ag3} finishes the proof.
\end{proof}

\section{Linear embeddings related to the greedy algorithm}\label{Sec10}

\subsection{Symmetric spaces and embeddings} A \emph{gauge} on $\NN$ will be a map $\Vert \cdot\Vert_\Sym\colon\NN\to \NN$ verifying (q1) and (q2) in the definition of a quasi-norm, and also:
\begin{itemize}
\item[(q5)] $\Vert (b_n)_{n=1}^\infty \Vert_\Sym\le \Vert (a_n)_{n=1}^\infty\Vert_\Sym$ whenever $|b_n|\le |a_n|$ for every $n\in\NN$,
\item[(q6)] $\Vert \sum_{n\in A} \ee_n \Vert_\Sym<\infty$ for every $A\subseteq\NN$ finite, and
\item[(q7)] if $(a_{n,k})_{n,k\in\NN}$ in $[0,\infty)$ is non-decreasing in $k$, then
\[
\left\Vert \left(\lim_k a_{n,k}\right)_{n=1}^\infty\right\Vert_\Sym=\lim_k \Vert (a_{n,k})_{n=1}^\infty\Vert_\Sym.
\]
\end{itemize}
Associated to a gauge $\Vert \cdot\Vert_\Sym$ on $\NN$ we have the space
\[
\Sym=\{ f\in \FF^\NN \colon \Vert f \Vert_\Sym<\infty\}.
\]
If the following condition is fulfilled
\begin{itemize}
\item[(q8)] $\Vert (a_{\pi(n)})_{n=1}^\infty\Vert_\Sym=\Vert (a_n)_{n=1}^\infty\Vert_\Sym$ for every permutation $\pi$ of $\NN$, 
\end{itemize}
then the gauge and its associated space are said to be \emph{symmetric}. 

Given a symmetric gauge $\Vert \cdot\Vert_\Sym\colon\NN\to \NN$, we will refer to 
\[\varphi[\Sym]:=\left(\Big\Vert \sum_{n=1}^m \ee_n \Big\Vert_\Sym\right)_{m=1}^\infty\] as the \emph{fundamental function} of the gauge (and of the space). 

If $\Vert \cdot\Vert_\Sym$ verifies (q3) (respectively, (q4)) in the definition of a quasi-norm (resp., a $p$-norm) we say that $\Vert \cdot\Vert_\Sym$ is a \emph{function quasi-norm} (resp., a function $p$-norm) on $\NN$.

As for locally convex spaces, if $\Vert \cdot\Vert_\Sym$ is a quasi-norm then $(\Sym,\Vert \cdot\Vert_\Sym)$ is a quasi-Banach space (see \cite{BS1988}*{Theorem 1.7}). In this case we will say that $(\Sym,\Vert \cdot\Vert_\Sym)$ is a \emph{quasi-Banach function space} on $\NN$. Thus, if $\Vert \cdot\Vert_\Sym$ is a $p$-norm then $(\Sym,\Vert \cdot\Vert_\Sym)$ is a $p$-Banach space. Note that the unit vector system is an unconditional basic sequence of any quasi-Banach function space on $\NN$.

Given two gauges $\Vert \cdot\Vert_{\Sym_1}$ and $\Vert \cdot\Vert_{\Sym_2}$ on $\NN$ with associated spaces $\Sym_1$ and $\Sym_2$, respectively, we say that $\Sym_1$ is continuously contained  in $\Sym_2$ and write $\Sym_1\subseteq\Sym_2$ if there is a constant $C$ such that $\Vert f \Vert_{\Sym_2}\le C \Vert f \Vert_{\Sym_1}$ for  all $f\in\FF^\NN$. We write $\Sym_1=\Sym_2$ if $\Sym_1\subseteq\Sym_2$ and $\Sym_2\subseteq\Sym_1$. In the case when $\Vert f \Vert_{\Sym_2}= \Vert f \Vert_{\Sym_1}$ for every $f\in\FF^\NN$ we say that $\Sym_1=\Sym_2$ isometrically.

The conjugate gauge $\Vert \cdot \Vert_{\Sym'}$ of a gauge $\Vert \cdot \Vert_\Sym$ on $\NN$ is defined for $f= (b_n)_{n=1}^\infty\in\FF^\NN$ by
\[
\Vert f \Vert_{\Sym'} = \sup \left\{ \left|\sum_{n=1}^m a_n b_n \right| \colon \Vert (a_n)_{n=1}^\infty\Vert_\Sym\le 1, m\in \NN \right\}.
\]
It is straightforward to check that $\Vert \cdot \Vert_{\Sym'}$ is a function norm on $\NN$. The conjugate space $\Sym'$ of a function space $\Sym$ on $\NN$ is 
\[
\Sym'=\{f= (b_n)_{n=1}^\infty\in\FF^\NN\colon \Vert f \Vert_{\Sym'} <\infty\}.
\]

If $\Sym$ is a quasi-Banach function space, we denote by $\Sym_0$ the space generated by $c_{00}$ in $\Sym$.
\begin{lemma}\label{lem:conjugate}Let $\Vert \cdot \Vert_\Sym$ be a function quasi-norm on $\NN$. Then, under the natural dual mapping, $\Sym_0^*=\Sym'$ isometrically.
\end{lemma}
\begin{proof}
There is a natural bounded linear map $T\colon\Sym'\to\Sym_0^*$ 
of norm $\Vert T \Vert\le 1$
defined by 
\[T( (b_n)_{n=1}^\infty)((a_n)_{n=1}^\infty)=
\sum_{n=1}^\infty a_n b_n.\]
 Let $f^*\in \Sym_0^*$, and  set $b_n=f^*(\ee_n)$  for $n\in\NN$. If $\Vert (a_n)_{n=1}^\infty\Vert_\Sym\le 1$ and $m\in\NN$,
\[
\left|\sum_{n=1}^m a_n b_n \right|
=\left| f^*\left(\sum_{n=1}^m a_n \ee_n\right)\right|
\le \Vert f^*\Vert \left\Vert \sum_{n=1}^m a_n \ee_n\right\Vert 
\le \Vert f^*\Vert.
\]
Hence, $\Vert (b_n)_{n=1}^\infty\Vert_{\Sym'}\le \Vert f^*\Vert$. Since $T( (b_n)_{n=1}^\infty)(\ee_k)=f^*(\ee_k)$ for every $k\in\NN$, it follows that $T( (b_n)_{n=1}^\infty)=f^*$.
\end{proof}

Let us introduce the following properties involving the mappings $\Fou$ and $\II$ defined in \eqref{eq:Fourier} and \eqref{eq:InvFourier}, respectively.
\begin{definition} Suppose that $\XX$ is a quasi-Banach space and that $\BB$ is a basis for $\XX$.

\noindent (a) A function space $(\Sym, \Vert \cdot\Vert_\Sym)$ on $\NN$ is said to embed in $\XX$ via $\BB$, and we denoted it by putting 
\[\Sym\stackrel{\BB}\hookrightarrow \XX,
\] 
if $\Sym\subseteq\YY$ and there is a constant $C$ such that 
$
\Vert \II(f) \Vert \le C \Vert f \Vert_\Sym
$
for all $f\in \Sym$.

\noindent (b) In the reverse direction, we say that $\XX$ embeds in a function space $(\Sym, \Vert \cdot\Vert_\Sym)$ on $\NN$ via $\BB$, and put 
\[
\XX\stackrel{\BB}\hookrightarrow\Sym,
\] 
if there is a constant $C$ such that $\Vert \FFF(f) \Vert_\Sym\le C\Vert f\Vert$ for all $f\in\XX$.

\noindent (c)  We say that $\XX$ can be \emph{sandwiched between symmetric spaces via $\BB$} if there are symmetric function spaces $\Sym_1$ and $\Sym_2$ on $\NN$ with $\varphi[\Sym_1]\approx \varphi[\Sym_2]$ such that 
\[
\Sym_1\stackrel{\BB}\hookrightarrow \XX \stackrel{\BB}\hookrightarrow \Sym_2.
\]
\end{definition}

\begin{lemma}\label{lem:embedding3}Let $\XX$ be a quasi-Banach space with a semi-normalized $M$-bounded basis $\BB$. Suppose that $\Sym_1$ and $\Sym_2$ are symmetric function spaces on $\NN$ such that $\Sym_1\stackrel{\BB}\hookrightarrow \XX \stackrel{\BB}\hookrightarrow \Sym_2$.
Then, for $m\in\NN$,
\begin{itemize}
\item[(i)] $\varphi_u^\EE(m)\lesssim \varphi[\Sym_1](m)$.
\item[(ii)] $\varphi[\Sym_2](m)\lesssim \varphi_l^\EE(m)$.
\item[(iii)] $\displaystyle \Vert \UU_m\Vert \lesssim \frac{\varphi[\Sym_1](m)}{ \varphi[\Sym_2](m)}$.
\end{itemize}

\end{lemma}
\begin{proof}Assume that $\Vert \II(f)\Vert\le C_1\Vert f\Vert_{\Sym_1}$ for all $f\in\Sym_1$ and that $\Vert \FFF(f)\Vert_{\Sym_2}\le C_2 \Vert f\Vert$ for all $f\in\XX$. Let $m\in\NN$, $A\subset \NN$ with $|A|\le m$, and $\varepsilon\in\EE_A$. We have
\[
\Vert \Ind_{\varepsilon, A}[\BB,\XX]\Vert 
\le C_1\Vert \Ind_{\varepsilon, A}[\BB_e]\Vert_{\Sym_1} 
=C_1\varphi[\Sym_1](|A|)
\le C_1\varphi[\Sym_1](m).
\]
Taking the supremum on $A$ and $\varepsilon$ we get 
\[
\varphi_u^\EE(m)\le C_1 \varphi[\Sym_1](m).
\]
Next, pick $B\subseteq\NN$ with $|B|\ge m$, and $\delta\in\EE_B$. We have
\[
\varphi[\Sym_2](m)\le \varphi[\Sym_2](|B|)
= \Vert \Ind_{\delta, B}[\BB_e]\Vert_{\Sym_2}
\le C_2 \Vert \Ind_{\delta , B}[\BB,\XX]\Vert.
\]
Taking the infimum over $B$ and $\delta$ we get 
\[
\varphi[\Sym_2](m) \le C_2 \varphi_l^\EE(m).
\] 

Finally, let $f \in\XX$. For $m\in\NN$, let $A$ be the $m$th greedy set of $f$, $t=\min_{n\in A} |\xx^*_n(f)|$, and $\varepsilon=(\sgn(\xx^*_n(f)))_{n\in A}$. We have
\begin{align*}
\Vert \UU_m(f)\Vert 
&=t \Vert \Ind_{\varepsilon, A}[\BB,\XX]\Vert\\
&\le C_1 t \Vert \Ind_{\varepsilon, A}[\BB_e] \Vert_{\Sym_1}\\
&=C_1 \frac{ \varphi[\Sym_1](m)}{ \varphi[\Sym_2](m) } t \Vert \Ind_{\varepsilon, A}[\BB_e] \Vert_{\Sym_2}\\
&\le C_1 \frac{ \varphi[\Sym_1](m)}{ \varphi[\Sym_2](m) } \Vert \Fou(f) \Vert_{\Sym_2}\\
&\le C_1 C_2 \frac{ \varphi[\Sym_1](m)}{ \varphi[\Sym_2](m) } \Vert f \Vert,
\end{align*}
as desired.
\end{proof}

\begin{proposition}\label{prop:embedding97}Let $\BB$ be an $M$-bounded semi-normalized basis of a quasi-Banach space $\XX$. Supposse that $\XX$ can be sandwiched between symmetric spaces via $\BB$. Then $\BB$ is super-democratic and the restricted truncation operator is uniformly bounded.
\end{proposition}

\begin{proof}It is immediate from Lemma~\ref{lem:embedding3}.
\end{proof}

\begin{corollary} Let $\XX$ be a quasi-Banach space that can be sandwiched between symmetric spaces via a basis $\BB$. Then $\BB$ is lattice partially unconditional and symmetric for largest coefficients.
\end{corollary}

\begin{proof}It is straightforward by combining Proposition~\ref{prop:qg7}, Proposition~\ref{prop:qg1}, and Proposition~\ref{prop:embedding97}.
\end{proof}

The converse of Proposition~\ref{prop:embedding97} also holds, as we shall see below. The proof of this requires  the introduction of new techniques that have Lorentz spaces as the main ingredient.

\subsection{Embeddings via Lorentz  spaces}

 Let $\ww=(w_n)_{n=1}^\infty$ be a \emph{weight}, i.e., a sequence of positive numbers. By the \emph{primitive weight} of $\ww$ we mean the weight $\sss=(s_n)_{n=1}^\infty$    defined by $s_n=\sum_{j=1}^n w_j$.  

Given $(a_n)_{n=1}^\infty\in c_0$, for $0<p<\infty$ and $0<q<\infty$ we put
\[
\left\Vert (a_n)_{n=1}^\infty\right\Vert_{p,q,\ww}= \left(\sum_{n=1}^\infty a_{n}^{\ast q} s_n^{q/p-1} w_n\right)^{1/q},
\]
 and for for $0<p<\infty$ and $q=\infty$ we put
\[
\left\Vert (a_n)_{n=1}^\infty\right\Vert_{p,\infty,\ww}= \sup_{n\in\NN}a_{n}^{\ast}s_n^{1/p},
\]
where $(a_n^*)_{n=1}^\infty$ is the non-increasing rearrangement of $(|a_n|)_{n=1}^\infty$.  

Let $S_m\colon\FF^\NN\to\FF^\NN$ be the $m$th partial-sum projection associated to the unit vector system $\BB_e$. We extend $\Vert \cdot \Vert_{p,q,\ww}$ to $\FF^\NN$ via the (consistent) equation
\[
\Vert f\Vert_{p,q,\ww}=\sup_m \Vert S_m (f)\Vert_{p,q,\ww}.
\] 
It is routine to check that, if $p$, $q$ and $\ww$ are as above, then $\Vert \cdot\Vert_{p,q,\ww}$ is a symmetric gauge on $\NN$. 

The \emph{weighted Lorentz sequence space} $d_{p,q}(\ww)$ is defined as the space associated to the symmetric gauge $\Vert \cdot \Vert_{p,q,\ww}$. We denote  its fundamental function by $\varphi_{p,q,\ww}$, i.e.,
\[
\varphi_{p,q,\ww}(m)=\left(\sum_{n=1}^m s_n^{q/p-1} w_n\right)^{1/q}, \quad m\in\NN.
\]
Note that $d_{p,q}(\ww)\subseteq \ell_\infty$ continuously, and that $d_{p,q}(\ww)\subseteq c_0$ unless   

\begin{equation}\label{eq:trivialLorentz} 
\left(\sum_{n=1}^\infty s_n^{q/p-1} w_n\right)^{1/q}<\infty, 
\end{equation} 
in which case  
$d_{p,q}(\ww)=\ell_\infty$. 

For potential weights, this general definition recovers the \emph{classical sequence Lorentz} spaces. Explicitly, if for $0<\alpha\le 1$ we consider the sequence
\begin{equation}\label{eq:potential}
\uu_\alpha=(n^{\alpha-1})_{m=1}^\infty
\end{equation}
then we have
\[d_{p,q}(\uu_\alpha)=\ell_{p/\alpha,q}, \]
 for $0<p<\infty$ and $0<q\le\infty$.

It is known (see e.g. \cite{AA2015}*{Lemma 2.11~(a)}) that $d_{p,q}(\ww_1)\subseteq d_{p,q}(\ww_2)$ continuously if and only $\sss_2\lesssim \sss_1$, where $\sss_1$ and $\sss_2$ are the primitive weights of $\ww_1$ and $\ww_2$,  respectively.

Given a positive  increasing  sequence $\ttt=(t_n)_{n=1}^\infty$, its \emph{discrete derivative} $\Delta(\ttt)$ will be the weight whose primitive weight is $\ttt$, i.e., with the convention that $t_0=0$,
\[\Delta(\ttt)=(t_n-t_{n-1})_{n=1}^\infty.\] Observe that the very definition of the sequence Lorentz spaces yields $d_{p,\infty}(\ww)=d_{1,\infty}(\Delta(\sss^{1/p}))$ isometrically, and  $d_{p,q}(\ww)=d_{q,q}(\sss^{q/p-1}\ww)$ isometrically    for $q<\infty$. 

Thus, dealing with bi-parametric Lorentz sequence spaces is somehow superfluous. Still, we use two parameters to emphasize that for fixed $p$ and $\ww$ all the spaces $d_{p,q}(\ww)$ belong to the ``same scale" of sequence Lorentz spaces and that, in some sense, the spaces are close to each other. As a matter of fact, 
\[
d_{p,q_0}(\ww)\subseteq d_{p,q_1}(\ww), \quad 0<q_0<q_1\le\infty
\] 
(see \cite{CS1993}), and the spaces involved in this embedding share the fundamental function.

Indeed, the definition of the spaces gives
\[
\varphi_{p,p,\ww}=\varphi_{p,\infty,\ww}= \sss^{1/p}.
\]
Since for $n\in\NN$,
\begin{equation*}
s_n^{q/p}-s_{n-1}^{q/p}=s_n^{q/p} \left( 1 -\left(\frac{s_{n-1}}{s_n}\right)^{q/p}\right) 
\approx s_n^{q/p-1}(s_n-s_{n-1})= s_n^{q/p-1} w_n,
\end{equation*} 
we have
\begin{equation}\label{eq:equalLorentzbis}
d_{p,q}(\ww)=d_{q,q}(\Delta(\sss^{q/p})), \quad 0<p,q<\infty.
\end{equation}
We deduce that condition \eqref{eq:trivialLorentz} is equivalent to $\ww\in\ell_1$, and that 
$
\varphi_{p,q,\ww}\approx \varphi_{q,q,\Delta(\sss^{p/q})}
$
for every $0<q<\infty$. Hence
\begin{equation}\label{eq:democraticLorentz}
\varphi_{p,q,\ww}(m) \approx s_m^{1/p}, \quad m\in\NN.
\end{equation}
for every $0<q\le \infty$.

Next we put together some properties of Lorentz sequence spaces that will be of interest to us.

\begin{proposition}[see \cite{CRS2007}*{Theorem 2.2.13}]\label{prop:LorentzQB} Suppose $0<p<\infty$ and $0<q\le\infty$. Then
$\Vert \cdot\Vert_{p,q,\ww}$ is a function quasi-norm (so that $d_{p,q}(\ww)$ is a quasi-Banach space) if and only if the primitive weight $\sss$ of $\ww$ is doubling.
\end{proposition}
If $d_{p,q}(\ww)$ is a quasi-Banach space, the unit vector system is a basis for the space 
$d_{p,q}^0(\ww)$ generated by $c_{00}$ in $d_{p,q}(\ww)$. This definition is convenient since in some cases $c_{00}$ is not dense in $d_{p,q}(\ww)$.

\begin{proposition}\label{eq:LorentzSeparable} Suppose that for some $0<p<\infty$, $0<q\le\infty$ and some weight $\ww$, the space $d_{p,q}(\ww)$ is quasi-Banach. Then $c_{00}$ is dense in $d_{p,q}(\ww)$ if and only if $q<\infty$.
\end{proposition}
\begin{proof} Given $f=(a_n)_{n=1}^\infty$, for  $m\in\NN$ we have
\[
\Vert \HH_m(f) \Vert_{p,q,\ww}= \left(\sum_{n=1}^\infty a_{n+m}^{\ast q} s_n^{q/p-1} w_n\right)^{1/q}
\]
with the usual modification if $q=\infty$.

If $q<\infty$ and $f\in d_{p,q}(\ww)$, by the Dominated Convergence Theorem $\lim_m \Vert \HH_m(f) \Vert_{p,q,\ww}=0$. If $q=\infty$ we pick $f=\sss^{-1/p}$, where $\sss=(s_m)_{m=1}^\infty$ is the primitive weight of $\ww$. We have
\[
\inf_m \Vert \HH_m(f) \Vert_{p,q,\ww}=\inf_m \sup_n \frac {s_n^{1/p}}{s_{n+m}^{1/p}} \ge \inf_m \frac {s_m^{1/p}}{s_{2m}^{1/p}}.
\]
Then $f\in d_{p,q,\ww}$ and, by Proposition~\ref{prop:LorentzQB}, $\inf_m \Vert \HH_m(f) \Vert_{p,q,\ww}>0$.
\end{proof}

Following \cite{DKKT2003}, we say that a weight $(s_m)_{m=1}^\infty$ has the \textit{upper regularity property} (URP for short) if there is an integer $b\ge 3$ such that
\begin{equation}\label{PW:def_URP}
s_{b m}\le \frac{b}{2} s_m,\quad m\in\NN.
\end{equation}

We will also need the so-called \textit{lower regularity property} (LRP for short). We say that $(s_m)_{m=1}^\infty$ has the LRP if there an integer $b\ge 2$ such  
\begin{equation}\label{PW:def_LRP}
2 s_m \le s_{bm}, \quad m\in\NN. 
\end{equation}
Note that $(s_m)_{m=1}^\infty$ has   the LRP if and only if $(m/s_m)_{n=1}^\infty$ has  the URP.

A weight $\vv=(v_n)_{n=1}^\infty$ is said to be \emph{regular} if it satisfies the Dini condition
\[
\sup_{n} \frac{1}{n v_n} \sum_{k=1}^n v_k <\infty.
\]

We say that a sequence $(s_n)_{n=1}^\infty$ of positive numbers is \textit{essentially increasing} (respectively \textit{essentially decreasing}) if there is a constant $C\ge 1$ with
\[
s_{k}\le C s_n \text{ (resp. }s_{k}\ge C s_n \text{) whenever }k\le n.
\]

\begin{lemma}[\cite{AA2015}*{Lemma 2.12}]\label{lem:AnsoWeights}Let $\sss=(s_m)_{m=1}^\infty$ be a non-{decreasing} weight such that $(s_m/m)_{m=1}^\infty$ is essentially decreasing. The following are equivalent:
\begin{itemize}
\item[(i)] $\sss$ has  the URP.
\item[(ii)]$1/\sss$ is a regular weight.
\item[(iii)] $(s_m/m^r)_{m=1}^\infty$ is essentially decreasing for some $0<r<1$.
\end{itemize}
\end{lemma}

\begin{proposition}[cf. \cite{CRS2007}*{$\S$2.2}]\label{prop:convexityLorentz} Let $1<q\le\infty$ and $0<r<1$. Suppose that  $\ww$ is a weight with primitive weight $\sss=(s_m)_{m=1}^\infty$. 
\begin{itemize}
\item[(i)] $d_{1,q}(\ww)$ is locally convex if and only if the primitive weight of $1/\sss$ is a regular weight.
\item[(ii)] If $\sss^{-r}$ is a regular weight then $d_{1,q}(\ww)$ is $r$-convex.
\item[(iii)] If $(m^{-1/r} s_m)_{m=1}^\infty$ is essentially decreasing, $d_{1,q}(\ww)$ is $s$-convex for every $0<s<r$.
\end{itemize}
\end{proposition}
\begin{proof}(i) follows from \cite{CRS2007}*{Theorem 2.5.10 and Theorem 2.5.11}. 

Assume that $\sss^{-r}$ is a regular weight. Then, by part (i) and \eqref{eq:equalLorentzbis}, the space $d_{1/r,rq,\ww}$ is locally convex. Since for  $f\in\FF^\NN$,
\[
\Vert f \Vert_{1,q,\ww}=\Vert |f |^r \Vert_{1/r,rq,\ww}^{1/r},
\] 
$d_{1,q}(\ww)$ verifies the lattice convexity estimate
\[
\left\Vert \left(\sum_{j} |f_j|^r \right)^{1/r} \right\Vert_{1,q,\ww} \le C \left( \sum_j \Vert f_j\Vert_{1,q,\ww}^r\right)^{1/r}, \quad f_j\in\FF^\NN,
\]
for some $C<\infty$. We infer that $d_{1,q}(\ww)$ is an $r$-convex quasi-Banach space, and so (ii) holds.

To show (iii), assume that   $(m^{-1/r} s_m)_{m=1}^\infty$ is essentially decreasing  and that $s<r$. By Lemma~\ref{lem:AnsoWeights}, $\sss^{- s}$ is a regular weight, and so (ii) yields that $d_{1,q}(\ww)$ is an $s$-convex quasi-Banach space.
\end{proof}

The discrete Hardy operator $A_d\colon\FF^\NN\to\FF^\NN$ is defined by 
\[
A_d\left((a_n)_{n=1}^\infty\right) =\left(\frac{1}{n} \sum_{k=1}^n a_k\right)_{n=1}^\infty.
\]
\begin{theorem}[see \cite{CRS2007}*{Theorems 1.3.7 and 1.3.8}]\label{thm:HardyLorentz} Suppose $1<q\le\infty$ and let $\ww$ be a non-increasing weight with primitive weight $\sss$. Then the discrete Hardy operator is bounded from $d_{1,q}(\ww)$ into $d_{1,\infty}(\ww)$ if and only if $\sss^{-1}$ is a regular weight.
\end{theorem}

The duals of Lorentz sequence spaces can be described in terms of Marcinkiewicz spaces. Given a weight $\ww=(w_n)_{n=1}^\infty$ the \emph{Marcinkiewicz sequence space} $m(\ww)$ is the Banach function space on $\NN$ associated to the function norm
\[
\Vert f \Vert_{m(\ww)}= \sup_{\substack{A\subseteq \NN \\ A \text{ finite}}} \frac{ \sum_{n\in A} |a_n|}{\sum_{n=1}^{|A|} w_n}, \quad f=(a_n)_{n=1}^\infty.
\]
We denote by $m_0(\ww)$ the separable part of $m(\ww)$.
\begin{proposition}[cf. \cite{CRS2007}*{$\S$2}]\label{prop:lorentzdual}
Let $0<p<\infty$ and $0<q<1$. Suppose that $\ww$  is a weight,  and that the space $d_{p,q}(\ww)$ is quasi-Banach. 
Then, under the natural dual mapping,
\begin{itemize}
\item[(i)] $d_{p,q}(\ww))^*=m(\Delta(\sss^{1/p}))$, and
\item[(ii)] $(d_{p,\infty}^0(\ww))^*=d_{1,1}(\sss^{-1/p})$,
\end{itemize}
where $\sss$ is the primitive weight of $\ww$.
\end{proposition}
\begin{proof}Just combine \cite{CRS2007}*{Theorem 2.4.14} with Proposition~\ref{eq:LorentzSeparable} and Lemma~\ref{lem:conjugate}.
\end{proof}

 The following result is crucial in our study of embeddings.
\begin{theorem}\label{prop:embedding1} Suppose $\BB$ is a semi-normalized $M$-bounded basis of a $p$-Banach space $\XX$. 
Let $\ww_l=\Delta(\varphi^\EE_l)$ and $\ww_u=\Delta(\varphi^\EE_u)$. Then:
\begin{itemize}
\item[(i)] $d_{1,p}(\ww_u) \stackrel{\BB}\hookrightarrow \XX$, and,
\item[(ii)] $\XX \stackrel{\BB}\hookrightarrow d_{1,\infty}(\ww_l)$ if the restricted truncation operator is uniformly bounded.
\end{itemize}
\end{theorem}

\begin{proof} (i) By Corollary~\ref{cor:convexity}~(ii), if $|A|\le m$ and $(b_{n})_{n\in A}$ are scalars with  $|b_n|\le 1$,
\begin{equation}\label{eq:embedding1}
\left\Vert \sum_{n\in A} b_n\, \xx_n\right\Vert \le A_p \varphi_u^\EE(m).
\end{equation}
Let $(a_n)_{n=1}^\infty\in c_{00}$ be such that $(|a_n|)_{n=1}^\infty$ is decreasing. Put $t=|a_1|$ and for each $k\in\NN$ consider the set
\[
J_k=\{n\in\NN \colon t 2^{-k}< |a_n| \le t 2^{-k+1}\}.
\]
Notice that $(J_k)_{k=1}^\infty$ is a partition of  $\{n\in\NN \colon a_n\not=0\}$. Set $n_k=|J_k|$ ($n_0=0$) and $m_k=\sum_{j=1}^k n_j$, so that $J_k=\{n\in\NN \colon m_{k-1}+1\le n \le m_k\}$. For $n\in\NN$, let $s_n=\varphi_u^\EE(n)$ and $w_n=s_n^p-s_{n-1}^p$.
Combining \eqref{eq:embedding1} with Abel's summation formula gives
\begin{align*}
\left\Vert \sum_{n=1}^\infty a_n \, \xx_n\right\Vert^p &=\left\Vert\sum_{k=1}^\infty \sum_{n\in J_k} a_n \, \xx_n\right\Vert^p\\
&\le \sum_{k=1}^\infty \left\Vert \sum_{n\in J_k} a_n \,\xx_n\right\Vert^p\\
&\le A_p^p \sum_{k=1}^\infty  (t 2^{-k+1} s_{m_k})^p \\
&= A_p^p (2t)^p \sum_{k=1}^\infty 2^{-kp} \sum_{j=1}^k \sum_{n\in J_j} w_n\\ 
&= \frac{A_p^p (2t)^p}{1-2^{-p}} \sum_{j=1}^\infty 2^{-jp} \sum_{n\in J_j} w_n\\ 
&\le 4^p A_p^{2p} \sum_{n=1}^\infty |a_n|^p w_n.
\end{align*}
Then for every $f=(a_n)_{n=1}^\infty\in c_{00}$,
\[
\left\Vert \sum_{n=1}^\infty a_n \, \xx_n\right\Vert\le 4 A_p^2\Vert f \Vert_{p,p,\ww},
\] 
where $\ww=(w_n)_{n=1}^\infty$. Since, by \eqref{eq:equalLorentzbis}, $d_{p,p}(\ww)=d_{1,p}(\ww_u)$ and $c_{00}$ is dense in $d_{1,p}(\ww_u)$, it follows that $d_{1,p}(\ww_u)\stackrel{\BB}\hookrightarrow \XX$.

(ii) Let $f\in\XX$ and denote by $(a_m^*)_{m=1}^\infty$ the non-increasing rearrangement of $\Fou(f)$. Given $m\in\NN$, pick a greedy set $A$ of $f$ with $|A|=m$. We have $\min_{n\in A} |\xx_n^*(f)|= a_m^*$ so that
\[
a_m^* \varphi^\EE_l(m) \le a_m^* \left \Vert \sum_{n\in A} \sgn(\xx_n^*(f)) \, \xx_n\right\Vert
=\Vert \UU(f,A)\Vert 
\le \Lambda_u \Vert f\Vert.
\]
Consequently, $\Vert \Fou(f)\Vert_{1,\infty,\ww_l}\le \Lambda_u \Vert f\Vert$.
\end{proof}

Combining Theorem~\ref{prop:embedding1} and equation \eqref{eq:democraticLorentz} with Lemma~\ref{lem:embedding3} yields the following result.
\begin{corollary}\label{cor:embeddingweights}
Suppose $\BB$ is a semi-normalized $M$-bounded basis of a quasi-Banach space $\XX$. Let $\ww$ be weight whose primitive weight is $\sss$. Then:

\begin{itemize}
\item[(i)] $d_{1,p}(\ww) \stackrel{\BB}\hookrightarrow \XX$ for some $p$ if and only if $\varphi^\EE_u[\BB,\XX]\lesssim \sss$.
\item[(ii)] If the restricted truncation operator is uniformly bounded, then $\XX \stackrel{\BB}\hookrightarrow d_{1,\infty}(\ww)$ if and only if
$\sss\lesssim \varphi^\EE_l[\BB,\XX]$.
\end{itemize}
\end{corollary}

\begin{theorem}\label{thm:embedding4} Let $\XX$ be a $p$-convex quasi-Banach space with an $M$-bounded semi-normalized basis $\BB$. 
\begin{itemize}
\item[(i)]Suppose that $\BB$ is super-democratic, that the restricted truncation operator is uniformly bounded, and that the primitive weight of $\ww$  is equivalent to $\varphi_u^\EE$. Then 
\begin{equation}\label{eq:enclosing}
d_{1,p}(\ww) \stackrel{\BB}\hookrightarrow \XX \stackrel{\BB}\hookrightarrow d_{1,\infty}(\ww),
\end{equation}
where $d_{1,p}(\ww)$ are $ d_{1,\infty}(\ww)$ are quasi-Banach function spaces on $\NN$ with equivalent fundamental functions. 

\item[(ii)]Conversely, if \eqref{eq:enclosing} holds for some $p$ and some weight $\ww$ with primitive weight  $\sss$, then $\varphi^\EE_l \approx \sss \approx \varphi^\EE_u$. 
\end{itemize}
\end{theorem}

\begin{proof} (i) By Theorem~\ref{prop:embedding1}, 
\[
d_{1,p}(\ww)\subseteq d_{1,p}(\ww_u) \stackrel{\BB}\hookrightarrow \XX \stackrel{\BB}\hookrightarrow d_{1,\infty}(\ww_l)\subseteq d_{1,\infty}(\ww).
\] 
Since the weight $\sss$ is doubling, Proposition~\ref{prop:LorentzQB} gives that the two Lorentz spaces involved are quasi-Banach. Finally, by \eqref{eq:democraticLorentz}, 
\[\varphi_{1,p,\ww}\approx\varphi_{1,\infty,\ww}\approx\sss.\] 

(ii) If \eqref{eq:enclosing} holds, Corollary~\ref{cor:embeddingweights} gives $\varphi^\EE_u \lesssim \sss \lesssim \varphi^\EE_l$.
\end{proof}

We emphasize that \eqref{eq:enclosing} is considered by some authors as a condition which ensures in a certain sense the optimality of the compression algorithms with respect to the basis (see \cite{Donoho1993}). We refer the reader to \cites{DKK2003, BDKPW2007, AADK2019} for the uses of this type of embeddings within the study of non-linear approximation in Banach spaces with respect to bases.

We are now in a position  to state the aforementioned converse of Proposition~\ref{prop:embedding97}. 
\begin{corollary}\label{cor:embedding2}Let $\BB$ be an $M$-bounded  super-democratic basis of a quasi-Banach space $\XX$.
Supposse that  the restricted truncation operator is uniformly bounded. Then $\XX$ can be sandwiched between symmetric spaces via $\BB$. Moreover, the enclosing symmetric spaces we obtain are quasi-Banach. 
\end{corollary}

\begin{proof}It follows from  Theorem~\ref{thm:embedding4}, Proposition~\ref{prop:LorentzQB} and \eqref{eq:doubling}.
\end{proof}

Notice that Proposition~\ref{prop:embedding97} together with Corollary~\ref{cor:embedding2} give a characterization of  those super-democratic bases for which the restricted truncation operator is uniformly bounded. Combining this characterization with Theorems~\ref{thm:ag3} and \ref{thm:chg} yields the following:

\begin{theorem}Let $\BB$ be a $M$-bounded semi-normalized basis of a quasi-Banach space $\XX$. Then:
\begin{itemize}
\item[(i)] $\BB$ is almost greedy if and only if it is quasi-greedy and $\XX$ can be sandwiched between symmetric spaces via $\BB$.
\item[(ii)] $\BB$ is greedy if and only if it is unconditional and $\XX$ can be sandwiched between symmetric spaces via $\BB$.
\end{itemize}
\end{theorem}

\section{Banach envelopes}\label{Sec13}
\noindent
The concept of \emph{Banach envelope} of a quasi-Banach space, or more generally \emph{$r$-Banach envelope} for $0<r\le 1$, lies in the following result in the spirit of  category theory. The objects are  the quasi-Banach spaces and the morphisms are the linear contractions.
\begin{theorem}[see \cite{AACD2018}*{$\S$2.2}]\label{thm:defenv}Let $\XX$ be a quasi-Banach space. For each $0<r\le 1$ there is an $r$-Banach space $\VV$ and a linear map $J\colon \XX \to \VV$ with $\Vert J\Vert \le 1$ satisfying the following property:

\begin{itemize}
\item[(P)] For every $r$-Banach space $\YY$ and every bounded linear map $T\colon\XX\to\YY$ there is a unique map $S\colon\VV\to \YY$ such that 
$S\circ J=T$. Moreover $\Vert S\Vert \le \Vert T\Vert$.
\end{itemize}
\end{theorem}

Since (P) is a universal property, given a quasi-Banach space $\XX$, the pair $(\VV,J)$ satisfying (P) is unique up to an isometry. So we can safely say that the pair $(\VV,J)$ in Theorem~\ref{thm:defenv} is the $r$-\emph{Banach envelope} of $\XX$, or that $\VV$ is the $r$-Banach envelope of $\XX$ under the mapping $J$.
We will put $\VV=\Env[r]{\XX}$ and $J=J_{\XX,r}$. 

We say that a quasi-Banach space $\VV'$ is isomorphic to the $r$-Banach envelope of $\XX$ under a mapping $J'$ if there is an ismorphism $T\colon\Env[r]{\XX}\to \VV'$ such that 
$J'=T\circ J_{\XX,r}$. Note that if $\XX$ is $r$-convex then $\XX$ is isomorphic to $\Env[r]{\XX}$ under the identity map.

If $r=1$ we simply put $\Env{\XX}=\Env[1]{\XX}$ and $J_{\XX}=J_{\XX,1}$, and we say that $(\Env{\XX},J_\XX)$ is the Banach envelope of $\XX$. 

The universal property of  $r$-Banach envelopes readily gives the following result.
\begin{lemma}\label{lem:EnvOperators}Let $\XX$ be a quasi-Banach space and $0<r\le 1$.
Given an $r$-Banach space $\YY$, the map $S\mapsto S\circ J_{\XX,r}$ defines an isometry from $\LL(\Env[r]{\XX}, \YY)$ onto $\LL(\XX, \YY)$. In particular, the map $f^*\mapsto f^*\circ J_{\XX,r}$ defines an isometry from 
$(\Env[r]{\XX})^*$ onto $\XX^*$.
\end{lemma}
We also infer from the universal property the following.
\begin{lemma}\label{lem:EnvDensity}Let $\XX$ be a quasi-Banach space and $0<r\le 1$. Then $J_{\XX,r}(\XX)$ is a dense subspace of $\Env[r]{\XX}$.
\end{lemma}

\begin{remark}The proof of Theorem~\ref{thm:defenv} uses the Minkowski functional of the smallest $r$-convex set containing the unit ball of $\XX$. In the case when $r=1$, a less constructive and more functional approach is possible: $J_\XX$ is the bidual map and $\Env{\XX}$ is the closed subspace of $\XX^{**}$ generated by the range of the bidual map. Thus,
\[
\Vert J_\XX(f)\Vert=\sup_{f^*\in B_{\XX^*}} |f^*(f)|, \quad f\in\XX.
\]
\end{remark}
\begin{remark}Given a quasi-Banach space $\XX$, the family of spaces 
$\Env[r]{\XX}$ for $0<r\le 1$
 constitutes a ``scale'' of quasi-Banach spaces.
In fact,  the universal property yields the existence of bounded linear maps
\[
J_{r,s}\colon \Env[r]{\XX} \to \Env[s]{\XX}, \quad 0<r\le s \le 1
\]
such that $J_{s,t} \circ J_{r,s}=J_{r,t}$ for every $0<r\le s \le t\le 1$. Morever, $\Env[s]{\XX}$ is the $s$-Banach envelope of 
$\Env[r]{\XX}$ under  $J_{r,s}$.
\end{remark}

Let us record a result about $r$-Banach envelopes that will be useful below. Given $0\le p\le\infty$ we define $\env{p}{r}$ by 
\[
\env{p}{r}
=\begin{cases}p&\text{ if }p\in[r,\infty]\cup\{0\},\\r&\text{ if } p\in(0,r).
\end{cases}
\]
\begin{proposition}\label{prop:envelopesum}
Let $(\XX_i)_{i\in I}$ be a family of quasi-Banach spaces and $0\le p\le\infty$. Then for any $0<r\le 1$ the $r$-Banach envelope of $(\oplus_{i\in I} \XX_i)_p$ is
$(\oplus_{i\in I} \Env[r]{(\XX_i)})_{\env{p}{r}}$ under the map 
\[
(f_i)_{i\in I}\mapsto (J_{\XX_i, r}(f_i))_{i\in I}.
\]
\end{proposition}
The proof of Proposition~\ref{prop:envelopesum} is rather straightforward. For further reference,  in the next Corollary we write down some consequences that spring from it.

\begin{corollary}\label{cor:envelopes}Let $0<r\le 1$ and $p$, $q\in[0,\infty]$.
\begin{itemize}
\item[(i)] If $\XX$ and $\YY$ are quasi-Banach spaces, the $r$-Banach envelope of $\XX\oplus\YY$ is, under the natural mapping, $\Env[r]{\XX}\oplus \Env[r]{\YY}$.
\item[(ii)] With the usual modification if $p=0$, the $r$-Banach envelope of $\ell_p$ is $\ell_{\env{p}{r}}$ (under the inclusion map).
\item[(iii)] With the usual modifications if $p$ or $q$ are $0$, the $r$-Banach envelope of $\ell_q(\ell_p)$ is $\ell_{\env{q}{r}}(\ell_{\env{p}{r}})$ (under the inclusion map).
\item[(iv)] With the usual modifications if $p$ or $q$ are $0$, the $r$-Banach envelope of $\ell_p\oplus \ell_q$ is $\ell_{\env{p}{r}} \oplus \ell_{\env{q}{r} }$ (under the inclusion map).
\item[(v)]  The $r$-Banach envelope of $(\oplus_{n=1}^\infty \ell_p^n)_q$ is  $(\oplus_{n=1}^\infty \ell_{\env{p}{r}}^n)_{\env{q}{r} }$  (under the inclusion map).
\end{itemize}
\end{corollary}

The next result discusses how certain properties of bases transfer to envelopes. 

\begin{definition} Given a basis $\BB=(\xx_n)_{n=1}^\infty$  of a quasi-Banach space $\XX$,  the \emph{Banach envelope of the basis $\BB$} is the sequence $\Env[r]{\BB}=(J_{\XX,r}(\xx_n))_{n=1}^\infty$ in the  Banach envelope $\Env[r]{\XX}$ of $\XX$.
\end{definition}

\begin{proposition}\label{prop:basesenv} Let $\BB=(\xx_n)_{n=1}^\infty$ be a basis of a quasi-Banach space $\XX$ with coordinate functionals $(\xx_n^*)_{n=1}^\infty$, and let $0<r\le 1$. Then:\begin{itemize}
\item[(i)] $\Env[r]{\BB}$ is a basis of $\Env[r]{\XX}$. Moreover, if $(\yy_n^*)_{n=1}^\infty$ are the coordinate functionals of $\Env[r]{\BB}$ we have $\yy_n^*\circ J_{\XX,r}=\xx_n^*$ for  all $n\in\NN$.

\item[(ii)] If $\BB$ is $M$-bounded so is $\Env[r]{\BB}$.

\item[(iii)] If $\BB$ is a Schauder basis so is $\Env[r]{\BB}$.

\item[(iv)] If $\BB$ is an unconditional basis so is $\Env[r]{\BB}$.

\item[(v)] If $\BB$ is $M$-bounded and semi-normalized so is $\Env[r]{\BB}$.

\end{itemize}
\end{proposition}
\begin{proof} Since $\langle \xx_n \colon n\in \NN\rangle$ is dense in $\XX$,  by Lemma~\ref{lem:EnvDensity}, $\langle J_{\XX,r}(\xx_n) \colon n\in \NN\rangle$ is dense in $\Env[r]{\XX}$. We use Theorem~\ref{thm:defenv} to pick   functionals $\yy_n^*\in \Env[r]{\XX}$ satisfying $\yy_n^*\circ J_{\XX,r}=\xx_n^*$. For $n$, $k\in\NN$ we then have 
\[
\yy_n^*(J_{\XX,r}(\xx_k))=\xx_n^*(\xx_k)=\delta_{n,k}
\] 
so that (i) holds. 

For every $A\subseteq\NN$ finite  we have
\[ 
J_{\XX,r}\circ S_A[\BB,\XX]=S_A[\Env[r]{\XX}, \Env[r]{\BB}] \circ J_{\XX,r}
\]
therefore, by Lemma~\ref{lem:EnvOperators},
\[
 \Vert S_A[\Env[r]{\XX}, \Env[r]{\BB}]\Vert \le \Vert S_A[\BB,\XX]\Vert.
\] 
From here (ii), (iii) and (iv) hold.  

Since $\Vert \yy_n^*\Vert=\Vert \xx_n^*\Vert$ for all $n\in\NN$, we have
\[
\sup_n \{ \Vert J_{\XX,r}(\xx_n)\Vert, \Vert \yy_n^*\Vert \} \le \sup_n \{ \Vert \xx_n \Vert, \Vert \xx_n^*\Vert \}.
\]
Lemma~\ref{lem:bases1}, yields (v).  
\end{proof}  

\begin{corollary}\label{cor:equivenv} Let $\BB=(\xx_n)_{n=1}^\infty$ be a basis of a quasi-Banach space $\XX$ and $0<r\le 1$. Then the dual basis of $\Env[r]{\BB}$ is isometrically equivalent to the dual basis of $\BB$.
\end{corollary}

\begin{proof}Just combine Proposition~\ref{prop:basesenv}~(i) with Lemma~\ref{lem:EnvOperators}.
\end{proof}  

In the case when $J_{\XX,r}$ is one-to-one we can assume that $\XX\subseteq \Env[r]{\XX}$ and that $J_{\XX,r}$ is the inclusion map. However, in general the map $J_{\XX,r}$ need not be one-to-one. 
 For instance, we have $J_{L_p,r}=0$ for $0<p<r\le 1$ (see, e.g., \cite{AACD2018}*{Theorem 4.13 and Proposition 4.20}). The existence of a basis for $\XX$ is a guarantee that this will not occur. 
\begin{lemma}\label{lem:1-1envelope}Let $\XX$ be a quasi-Banach space  equipped with a total basis $\BB=(\xx_n)_{n=1}^\infty$ . Then for $0<r\le 1$, the map $J_{\XX,r}$ is ono-to-one 
\end{lemma}
\begin{proof}Assume that $J_{\XX,r}(f)=0$ for some $f\in\XX$. Then, if $(\yy_n^*)_{n=1}^\infty$ are the coordinate functionals of $\Env[r]{\BB}$, $\yy_n^*( J_{\XX,r}(f))=0$ for every $n\in\NN$. Therefore, by Proposition~\ref{prop:basesenv}~(i), $\xx_n^*(f)=0$ for every $n\in\NN$. Since $\BB$ is total, we have $f=0$.
\end{proof} 

The proof of Proposition~\ref{prop:basesenv} shows that transferring properties of bases  that are defined in terms of bounded linear operators such as unconditionality, from $\XX$ to $\Env[r]{\XX}$ is an easy task. However, trying to transfer ``nonlinear'' properties such as democracy or quasi-greediness is a more subtle issue.  Since the universal property that permits these transferrings is  highly linear, one may argue that it is hopeless to try to obtain positive answers for those questions. For instance, in general $\Env[r]{\BB}$ does not inherit democracy from $\BB$ (see $\S$\ref{sect:democracyenvelopes} below). Wether or not $\Env[r]{\BB}$ inherits quasi-greediness from $\BB$ seems to be more challenging (see the Problems section). As far as quasi-greediness is concerned, one can also argue in the opposite direction: since quasi-greediness is a substitute of unconditionality, and unconditionality transfers to envelopes, quasi-greediness is expected to behave accordingly. Below we give some partial results on this issue. These results along with the examples in $\S$\ref{sect:democracyenvelopes} exhibit that the obstructions to transfer properties to envelopes appear when the basis is close to the  unit vector basis of $\ell_1$, in the sense that its fundamental function is close to the sequence $(m)_{m=1}^\infty$.

\begin{proposition}\label{prop:envlr}Let $\XX$ be a quasi-Banach space endowed with a semi-normalized $M$-bounded basis $\BB$ for which the restricted truncation operator is uniformly bounded. Suppose that for some $0<r\le 1$,
\[
\sum_{m=1}^\infty \frac{1}{(\varphi_l^\EE[\BB,\XX](m))^r}<\infty.
\]
Then $\Env[r]{\BB}$ is equivalent to the unit vector system of $\ell_r$, and $\Env[r]{\XX}$ is isomorphic to $\ell_r$.
\end{proposition}

\begin{proof}Let $\ww_l=\Delta(\varphi^\EE_l)$. By Theorem~\ref{prop:embedding1}, $\XX \stackrel{\BB}\hookrightarrow d_{1,\infty}(\ww_l)$. Our assumption gives $d_{1,\infty}(\ww_l)\subseteq \ell_r$ continuously. Therefore the universal property of the $r$-Banach envelopes gives 
\[ 
\Env[r]{\XX} \stackrel{\Env[r]{\BB}}\hookrightarrow \ell_r. 
\] 
Since $\Env[r]{\BB}$ is a bounded sequence and $\Env[r]{\XX}$ is a $r$-Banach space,
\[ 
\ell_r \stackrel{\Env[r]{\BB}}\hookrightarrow \Env[r]{\XX}. 
\]
Combining gives that the unit vector system of $\ell_r$ is equivalent to $\Env[r]{\BB}$. Therefore the coordinate transform is an isomorphism from $\Env[r]{\XX}$ onto $\ell_r$.
\end{proof}

\begin{lemma}\label{BDtoEnv}Let $\BB=(\xx_n)_{n=1}^\infty$ be a bidemocratic basis of a quasi-Banach space $\XX$. Then $\Env[r]{\BB}$ is a bidemocratic basis for all $0<r\le 1$. Moreover 
\begin{equation}\label{eq:fundamentalenvelope} 
\varphi_u^\EE[\BB,\XX](m)\approx \varphi_u^\EE[\Env[r]{\BB} , \Env[r]{\XX}](m) \approx \frac{m}{ \varphi_u^\EE[\BB^*,\XX^*](m)}, \quad m\in\NN.
\end{equation}
\end{lemma}
\begin{proof}By definition $\varphi_u^\EE[\Env[r]{\BB} , \Env[r]{\XX}]\le \varphi_u^\EE[\BB,\XX]$. Since $\BB^*$ is (naturally isometric to) the dual basis of $\Env[r]{\BB}$ we have $\varphi_u^\EE[\Env[r]{\BB}^* , \Env[r]{\XX}^*]= \varphi_u^\EE[\BB^*,\XX^*]$. Thus $\Env[r]{\BB}$ is bidemocratic with $\Delta_{sb}[\Env[r]{\BB} , \Env[r]{\XX}]\le \Delta_{sb}[\BB,\XX]$. Finally, \eqref{eq:fundamentalenvelope} follows from \eqref{eq:BDFundamental}  and \eqref{analog}.
\end{proof}

\begin{theorem}\label{cor:BDtoReflexivity} Let $\BB$ be a bidemocratic quasi-greedy basis of a quasi-Banach space $\XX$. Then $h_{\Env{\BB},\Env{\XX}}$ is an isomorphic embedding. Thus, the bases $\BB^{**}$ and $\Env{\BB}$ are equivalent. 
\end{theorem}
\begin{proof}
The identification of $\Env{\XX}^*$ with $\XX^*$ provided by  Lemma~\ref{lem:EnvOperators}, together with the identification of $\Env{\BB}^*$ with $\BB^*$ provided by   Corollary~\ref{cor:equivenv}, yield
$h_{\Env{\BB},\Env{\XX}} \circ J_\XX = h_{\BB,\XX}$.

Suppose that $\BB^*$ is equivalent to the unit vector system of $c_0$. Since $\Env{\XX}$ is locally convex and $\Env{\BB}$ is a bounded sequence we have  
\[
\ell_1\stackrel{\Env{\BB}}\hookrightarrow \Env{\XX},
\]
i.e., the linear operator $\Env{\II}$ defined  in \eqref{eq:InvFourier} with respect to $\Env{\BB}$ is bounded from $\ell_1$ into $\Env{\XX}$.
Since, by duality, $\BB^{**}$ is equivalent to the unit vector system of $\ell_1$, the coefficient transform $\Fou^{**}$ with respect to the bidual basis $\BB^{**}$ is an isomorphism from $[\BB^*]^*$ into $\ell_1$.  Using \eqref{eq:bidualbasis} we deduce that for $n\in\NN$,
\[
\ee_n \stackrel{\Env{\II}}\mapsto J_\XX(\xx_n)  \stackrel{h_{\Env{\BB},\Env{\XX}}}\mapsto h_{\BB,\XX}(\xx_n)= \xx_n^{**}
\stackrel{\Fou^{**}}\mapsto \ee_n.
\]
Therefore $\Fou^{**}\circ h_{\Env{\BB},\Env{\XX}} \circ \Env{\II}$ is the identity map. It follows that  $\Env{\II}$ is  an isomorphism from $\ell_1$ onto $\Env{\XX}$ and that $ h_{\Env{\BB},\Env{\XX}}$ is an isomorphism from $\Env{\XX}$ onto $[\BB^*]^*$.

In the case when $\BB^*$ is non-equivalent to the unit vector system of $c_0$, our proof relies on Lemma~\ref{lem:basesreflexivity} and Corollary~\ref{cor:dicotomy}. Let $A\subseteq\NN$ be finite and let $f^*\in\XX^*=\Env{\XX}^*$. Set $A_0=A\cap\supp(f^*)$ and $A_1=A\setminus\supp(f^*)$. Since $\Fou^*(f^*)\in c_0$, there is a finite greedy set $B_0$ of $f^*$ such that $A_0\subseteq B_0$. Then, if $B=B_0\cup A_1$, we have $A\subseteq B$ and, by Theorem~\ref{thm:BDtoQG},
\[
\Vert S_B^*(f^*)\Vert=\Vert S_{B_0}^*(f^*) +S_{A_1}(f^*)\Vert = \Vert S_{B_0}^*(f^*)\Vert \le C \Vert f^* \Vert.\qedhere
\]
\end{proof}

Our next result is an easy consequence of Theorem~\ref{cor:BDtoReflexivity}. 
\begin{theorem}\label{thm:BDtoQG+}Let $\BB=(\xx_n)_{n=1}^\infty$ be a bidemocratic quasi-greedy basis of a locally convex quasi-Banach space $\XX$. Then $\BB$ and $\BB^{**}$ are equivalent bases.
\end{theorem} 

\begin{proof}Just notice that, since $\XX$ is locally convex, $\BB$ is equivalent to $\Env{\BB}$.
\end{proof}  

\begin{theorem}\label{thm:EnvAG}Let $\BB$ be a bidemocratic quasi-greedy basis of a quasi-Banach space $\XX$. Then $\Env{\BB}$ is an almost greedy basis.
\end{theorem}

\begin{proof} Lemma~\ref{lem:dualbidem} and Corolary~\ref{cor:BDtoQG} yield that $\BB^*$ is a bidemocratic  quasi-greedy basic sequence. Then, applying Corolary~\ref{cor:BDtoQG} to $\BB^*$ we obtain that $\BB^{**}$ is an almost greedy basis.
An appeal to Theorem~\ref{cor:BDtoReflexivity} finishes the proof.
\end{proof}

\begin{proposition}\label{prop13a}Let $\XX$ be a quasi-Banach space equipped with a semi-normalized $M$-bounded basis $\BB$ for which the restricted truncation operator is uniformly bounded.  Suppose that $(\varphi_l^\EE[\BB,\XX])^{-r}$ is a regular weight for some $0<r\le 1$. Then
\begin{itemize}
\item[(i)] $\varphi_l^\EE[\Env[r]{\BB},\Env[r]{\XX}]\approx \varphi_l^\EE[\BB,\XX]$.

\item[(ii)] If $\BB$ is democratic then $\Env[r]{\BB}$ is a democratic basis for which the restricted truncation operator is uniformly bounded.

\item[(iii)] If $\BB$ is democratic and $r=1$ then $\BB$ is a bidemocratic basis.
\end{itemize}
\end{proposition}
\begin{proof}Let $\sss=\varphi_l^\EE[\BB,\XX]$. By Proposition~\ref{prop:convexityLorentz}~(ii), $d_{1,\infty}(\Delta(\sss))$ is an $r$-convex quasi-Banach space. Therefore, combining Theorem~\ref{prop:embedding1}~(ii) with the universal property of $r$-Banach envelopes yields 
\begin{equation}\label{eq:sandwichwithurp} 
\Env[r]{\XX} \stackrel{\Env[r]{\BB}}\hookrightarrow d_{1,\infty}(\Delta(\sss)). 
\end{equation}
Using Corollary~\ref{cor:embeddingweights}~(ii) we obtain (i). 

Suppose that $\BB$ is democratic. By Corollary~\ref{cor:embeddingweights}~(i), $d_{1,p} (\Delta(\sss)) \stackrel{\BB}\hookrightarrow\XX$ for some $p\le 1$ so that
\begin{equation*}
d_{1,p} (\Delta(\sss)) \stackrel{\Env[r]{\BB}}\hookrightarrow \Env[r]{\XX}. 
\end{equation*} 
Using Corollary~\ref{cor:embedding2} finishes the proof of (ii).   

Assume that $r=1$ and that $\BB$ is democratic. Set $\ttt=(m/s_m)_{m=1}^\infty$. Dualizing in \eqref{eq:sandwichwithurp} and taking into consideration Proposition~\ref{prop:lorentzdual}~(ii) and the regularity of $\sss^{-1}$, we obtain
\begin{equation*}
d_{1,1} (\Delta(\ttt)) \subseteq d_{1,1} (1/\sss)\stackrel{\BB^*}\hookrightarrow \XX^*. 
\end{equation*}
Corollary~\ref{cor:embeddingweights}~(i) puts an end to the proof.
\end{proof}

\begin{corollary}Let $\BB$ be an almost greedy basis of a quasi-Banach space $\XX$. Suppose that $1/\varphi_u^\EE[\BB,\XX]$ is a regular weight. Then $\BB^*$ and $\Env{\BB}$ are almost greedy bases.
\end{corollary}

\begin{proof}Just combine Proposition~\ref{prop13a}~(iii) with Corollary~\ref{cor:BDtoQG} and Theorem~\ref{thm:EnvAG}.
\end{proof}

\section{Examples}\label{Sec11}
\subsection{Symmetric and subsymmetric bases}
Let us denote by $\OO$ the set of all increasing functions from $\NN$ to $\NN$. A basis $\BB=(\xx_n)_{n=1}^\infty$ of a quasi-Banach space $\XX$ is said to be \emph{subsymmetric} if it is unconditional and equivalent to all its subsequences. If for some constant $C\ge 1$ the basis satisfies
\[
\frac{1}{C} \Vert f \Vert \le \left\Vert \sum_{n=1}^\infty \varepsilon_n \, a_n \, \xx_{\phi(n)} \right\Vert \le C \Vert f \Vert
\]
for all $f=\sum_{n=1}^\infty a_n \, \xx_n\in \XX$, all $(\varepsilon_n)_{n=1}^\infty\in \EE_\NN$ and all $\phi\in\OO$, then $\BB$ is said to be $C$-subsymmetric. Mimicking the proof from the locally convex case we obtain the following result, which uses some linear operators related to subsymmetric bases. 

Given a basis $\BB$ of a quasi-Banach space $\XX$, an increasing function $\phi\colon A \subseteq \NN \to \NN$ and $\varepsilon=(\varepsilon_n)_{n\in A} \in\EE_A$ we consider the linear map
\[
U_{\phi,\varepsilon} \colon \langle \xx_n \colon n\in\NN \rangle \to \XX, \quad 
\sum_{n=1}^\infty a_n \,\xx_n \mapsto \sum_{n=1}^\infty a_n \, \varepsilon_n \, \xx_{\phi(n)}.
\] 

\begin{theorem}[cf.\ \cite{Ansorena}*{Theorem 3.5, Theorem 3.7 and Corollary 3.9}]\label{thm:charSS}Let $\BB$ be a basis of a quasi-Banach space $\XX$.
\begin{itemize}
\item[(i)] $\BB$ is $C$-subsymmetric if and only if $U_{\phi,\varepsilon}$ is well-defined on $\XX$ and $\Vert U_{\phi,\varepsilon} \Vert \le C$ for all increasing maps $\phi\colon A \subseteq \NN \to \NN$ and all $\varepsilon=(\varepsilon_n)_{n\in A} \in\EE_A$.

\item[(ii)] $\BB$ is subsymmetric if and only if it is $C$-subsymmetric for some $1\le C<\infty$.
\item[(iii)] If $\XX$ is $p$-convex and $\BB$ is subsymmetric there is an equivalent subsymmetric $p$-norm for $\XX$ with respect to which $\BB$ is $1$-subsymmetric.
\end{itemize}
\end{theorem}

\begin{corollary}\label{cor:subsymenvelope}  
Let $\BB$ be a $C$-subsymmetric basis of a quasi-Banach space $\XX$. Then
$\Env[r]{\BB}$ is a $C$-subsymmetric basis of $\Env[r]{\XX}$ for $0<r\le 1$ and $\BB^*$ is a $C$-subsymmetric basis of $\XX^*$.
\end{corollary}
\begin{proof}The result about $\Env[r]{\XX}$ follows from Theorem~\ref{thm:charSS}~(i) and the universal property of $r$-Banach envelopes. The result about $\XX^*$ follows from Theorem~\ref{thm:charSS}~(i) and duality.
\end{proof}

Clearly, a $1$-subsymmetric basis is $1$-unconditional and $1$-democratic. Thus, by Theorem~\ref{thm:charSS}, subsymmetric bases are greedy. Quantitatively, applyling Theorem~\ref{thm:chg2} yields that every $1$-subsymmetric basis of a $p$-Banach space is $2^{1/p}$-greedy. The following example exhibits that this estimate is optimal.  
\begin{example} 
Set 
\[ 
\WW:=\left\{(w_n)_{n=1}^\infty\in c_0\setminus\ell_1:1=w_1\geq w_2>\cdots w_n \ge w_{n+1} \ge\cdots>0\right\}. 
\] 
Given $0< p<\infty$ and $\ww=(w_n)_{n=1}^\infty\in\WW$, the \emph{Garling sequence space} $g(\ww,p)$ is the function quasi-Banach space on $\NN$ associated to the function quasi-norm
\[ 
\Vert f \Vert_{g(\ww,p)} = \sup_{\phi\in\OO } \left( \sum_{n=1}^\infty |a_{\phi(n)}|^p w_n \right)^{1/p}, \quad f=(a_n)_{n=1}^\infty\in\FF^\NN. 
\] 
It is straightforward to check that $g(\ww,p)$ is a $\min\{p,1\}$-Banach space. 

Garling sequence spaces were studied in depth for $p\ge 1$ in \cite{AAW1}. We shall extend to the case $p<1$ a couple of results that are of interest for us.
\begin{theorem}
The unit vector system $\BB_e$ is a $1$-subsymmetric basis of $g(\ww,p)$ for every $0< p<\infty$ and $\ww=(w_n)_{n=1}^\infty\in\WW$. 
\end{theorem}
\begin{proof}
It is clear that $\BB_e$ is a $1$-subsymmetric basic sequence, so we need only prove that its closed linear span is the entire space $g(\ww,p)$. Let $f=(a_n)_{n=1}^\infty\in g(\ww,p)$. Then $|f|^p\in g(\ww,1)$. Since the result holds in the case when $p=1$ (see \cite{AAW1}*{Theorem 3.1}) we have
\[
\lim_m \Vert f -S_m(f)\Vert_{g(\ww,p)}=\lim_m \Big\Vert |f|^p -S_m(|f|^p)\Big\Vert^{1/p}_{g(\ww,1)}=0.\qedhere
\]
\end{proof}
Given a tuple $f=(a_j)_{j=1}^m$ in a quasi-Banach sequence space $(\XX, \Vert \cdot\Vert)$ we put $\Vert f \Vert:=\Vert \sum_{j=1}^m a_n\, \ee_n\Vert$. If 
$f$ and $g$ are a pair of tuples in $\XX$ we denote its concatenation by $f \smallfrown g$.
\begin{lemma}[cf. \cite{AADK2019+}*{Lemma 2.3}]\label{Lemma:1}Let $0< p<\infty$ and $\ww\in\WW$.
Given $0<\varepsilon<1$ and tuples $f$ and $g$ with $\Vert f\Vert_{g(\ww,p)} \le 1$, there is a tuple $h$ such that $\Vert h\smallfrown f\Vert_{g(\ww,p)} \le 1$ and $\Vert g\smallfrown h\Vert \ge (\Vert g \Vert^p_{g(\ww,p)} +1-\varepsilon)^{1/p}$. Moreover, $h$ can be chosen to be a positive constant-coefficient $k$-tuple with $k$ as large as wished.
\end{lemma}
\begin{proof} The proof for $p\ge 1$ from \cite{AADK2019+} can be reproduced for $p<1$ with no issues.
\end{proof}
Now, with the help of Lemma~\ref{Lemma:1}, we proceed to estimate the greedy constant of Garling sequence spaces.
Given $\epsilon>0$ there are $0<a<1$ and $n\in\NN$ such that, if $A=\{2,\dots, n+1\}$, 
\[
\Vert \ee_{n+2}+a \Ind_A\Vert=1\quad \text{and}\quad \Vert \ee_{1} + a \Ind_A\Vert\ge (2-\epsilon)^{1/p}.
\]
Consequently, 
\[
C_g[\BB_e,g(\ww,p)] \ge \Gamma[\BB_e,g(\ww,p)]\ge 2^{1/p}.
\]
\end{example}
The concept of a subsymmetric basis is closely related to that of a symmetric basis. A basis is said to be \emph{symmetric} if it is
equivalent to all its permutations. A basis is said to be $C$-symmetric if 
\[ 
\left\Vert \sum_{n=1}^\infty \varepsilon_n \, a_n \, \xx_{\sigma(n)} \right\Vert \le C \Vert f \Vert
\]
for all $f=\sum_{n=1}^\infty a_n \, \xx_n\in \XX$, all $(\varepsilon_n)_{n=1}^\infty\in \EE_\NN$, and all permutations $\sigma$ of $\NN$.

Theorem~\ref{thm:charSS} has its symmetric counterpart, which can also be proved as in the locally convex setting.
\begin{theorem}[cf. \cites{Singer1, Singer2}]\label{thm:charS}Let $\BB$ be a basis of a of a quasi-Banach space $\XX$. Then:
\begin{itemize}
\item[(i)] $\BB$ is symmetric if and only if it is $C$-symmetric for some $1\le C<\infty$.
\item[(ii)] If $\BB$ is $C$-symmetric then $\BB$ is $C$-subsymmetric.
\item[(iii)] If $\XX$ is $p$-convex and $\BB$ is symmetric, there is an equivalent symmetric $p$-norm for $\XX$ with respect to which $\BB$ is $1$-symmetric.
\end{itemize}
\end{theorem}
 We record a result showing that the
 greedy algorithm provides better approximations for symmetric bases than for subsymmetric ones.
\begin{proposition}[cf. \cite{AW2006}*{Theorem 2.5}]Every $1$-symmetric basis in a quasi-Banach space is $1$-greedy.
\end{proposition}
\begin{proof}
The proof of this result for Banach spaces from \cite{AW2006} can be adapted to quasi-Banach spaces in a straightforward way.
\end{proof}

\subsection{Direct sums of bases}
Let $\XX_1$, $\XX_2$ be quasi-Banach spaces. 
Given bases $\BB_i=(\xx_{n,i})_{n=1}^\infty$ of $\XX_i$, $i=1,2$, $\BB_1\oplus\BB_2$, denotes the basis of $\XX_1\oplus\XX_2$ given by
\[
\BB_1\oplus\BB_2=((\xx_{1,1},0),(0,\xx_{1,2}), \dots, (\xx_{n,1},0),(0,\xx_{n,2}),\dots).
\]
Direct sums of bases in quasi-Banach spaces inherit from their components all the unconditionality-type conditions we have defined, namely, SUCC, LUCC, boundedness of the restricted truncation operator, quasi-greediness, QGLC, and unconditionality. The relation between the democracy of a direct sum of bases and that of its components is also well-known (see e.g. \cite{GHO2013}*{Proposition 6.1}), and the lack of local convexity does not alter the state of affairs. Thus, we can easily characterize when a direct sum of bases possesses a property which has a democratic component and an unconditionality-type component.
\begin{proposition}\label{prop:sumdemocracy} Let $\BB_1$ and $\BB_2$ be two democratic bases (respectively super-democratic bases, SLC bases, almost greedy bases, greedy bases, and bases sandwiched between symmetric spaces) of two quasi-Banach spaces $\XX_1$ and $\XX_{2}$. Then $\BB_1\oplus\BB_2$ is democratic (resp., super-democratic, SLC, almost-greedy, greedy and sandwiched between symmetric spaces) if and only if $\varphi_u[\BB_1,\XX_1]\approx \varphi_u[\BB_2,\XX_2]$.
\end{proposition}

Proposition~\ref{prop:sumdemocracy} tells us how  to easily  build examples of bases in direct  sums in such a way that we break  the  democracy while maintaining the starting 
unconditionality-type condition. For instance, the unit vector system of $\ell_p\oplus\ell_q$, $0<p\not=q\le\infty$, is an unconditional basis which is not greedy despite the fact that the unit verctor system in each component is greedy. In fact,  as we next show, the space $\ell_p\oplus\ell_q$ has no greedy basis.

\begin{theorem}[cf.\ \cite{Ortynski1981}*{Corollary 2.8}]\label{Nogreedylp+lq} Let $0<p_1<\dots <p_n\le \infty$, with $n\ge 2$. Then the space
$\oplus_{j=1}^n \ell_{p_j}$ has no greedy basis (we replace $\ell_\infty$ with $c_0$ if $p_n=\infty$).
\end{theorem}

Before seeing the proof Theorem~\ref{Nogreedylp+lq} we give two auxiliary results.

\begin{lemma}\label{lem:EWGideonIdea} Let $\XX$ be a quasi-Banach space. Suppose that every semi-normalized unconditional basis $\BB$ of the Banach envelope $\Env{\XX}$ of $\XX$ verifies one of the two following conditions:
\begin{itemize}
\item [(1)] Either $\BB$ has subbases $\BB_1$ and $\BB_2$ with
\[
\sup_m \frac{\varphi_u^{\EE}[\BB_2,\Env{\XX}] (m) }{\varphi_l^{\EE}[\BB_1,\Env{\XX}](m)}=\infty,
\]
and $\BB_1$ generates a space with non-trivial type, or
\item[(2)] $\BB$ has a subbasis $\BB_1$ equivalent to the unit vector system of $c_0$, and another subbasis which is not equivalent to the unit vector system of $c_0$.
\end{itemize}
Then $\XX$ has no greedy basis.
\end{lemma}

\begin{proof} Assume, by contradiction, that $\BB$ is a greedy basis of $\XX$. By Theorem~\ref{thm:chg}, $\BB$ is unconditional in $\XX$ and so, by Proposition~\ref{prop:basesenv}, is semi-normalized and unconditional as a basis of $\Env{\XX}$. Let $\YY_1$ (respectively, $\XX_1$) denote  the space generated by $\BB_1$ in $\Env{\XX}$ (respectively, $\XX$) either under condition (1) or  (2).  Corollary~\ref{cor:envelopes}~(i) yields that 
$\YY_1$ is isomorphic to the Banach envelope of $\XX_1$ under the map  $J_\XX|_{\XX_1}$.
If condition (1) holds, \cite{Kalton1978}*{Theorem 2.8} yields that $\XX_1\simeq \YY_1$. Therefore, 
\[ 
\varphi_u^\EE[\BB_2, \Env{\XX}]
\le \varphi_u^\EE[\BB, \Env{\XX}] 
\le \varphi_u^\EE[\BB,\XX] 
\approx \varphi_l^\EE[\BB,\XX] 
\le \varphi_l^\EE[\BB_1,\XX] 
\approx \varphi_l^\EE[\BB_1,\Env{\XX}].
\]

If condition (2) holds, by \cite{KaltonCanadian}*{Theorem 6.2}, there is a subsequence $\BB_3$ of $\BB_1$ which, when regarded in $\XX$, is equivalent to the unit vector system of $c_0$. Consequently, 
\[\varphi_u^\EE[\BB,\XX](m)\approx \varphi_u^\EE[\BB_3,\XX](m)\approx 1,\quad\text{for}\;m\in\NN.
\] We infer that $\BB$, when regarded in $\XX$, is equivalent to the unit vector system of $c_0$. Therefore $\Env{\XX}=\XX$ and so every subbasis of $\BB$ when regarded in $\Env{\XX}$ is equivalent to the unit vector system of $c_0$.

That is, in both cases we reach an absurdity and so  $\XX$ has no greedy basis.
\end{proof}

\begin{corollary}\label{cor:EWGideonIdea} Let $\XX$ be a quasi-Banach space. Suppose that every semi-normalized unconditional basis of the Banach envelope $\Env{\XX}$ of $\XX$ has subbases $\BB_1$ and $\BB_2$ 
which, for some $1\le s_2<s_1\le \infty$, generate spaces $\XX_1$ and $\XX_2$ respectively isomorphic to $\ell_{s_1}$ and $\ell_{s_2}$($c_0$ if $s_1=\infty$). Then $\XX$ has no greedy basis.
\end{corollary}
\begin{proof}Given $1\le s \le\infty$, every semi-normalized unconditional basic sequence of $\ell_s$ has a subbasis equivalent to the unit vector system of $\ell_s$. Consequently, if $\BB_1$ and $\BB_2$ are as in the hypothesis,
\[
\frac{\varphi_u^{\EE}[\BB_2,\XX_2] (m) }{\varphi_l^{\EE}[\BB_1,\XX_1](m)}\gtrsim m^{1/s_2-1/s_1}, \quad m\in\NN.
\]
Moreover, if $s_1<\infty$, the space $\ell_{s_1}$ has non-trivial type. So the result follows from Lemma~\ref{lem:EWGideonIdea}.
\end{proof}

\begin{proof}[Proof of Theorem~\ref{Nogreedylp+lq}] 
Set $\XX=\oplus_{j=1}^n \ell_{p_j}$ and  put $q=p_n$. 
Suppose first that  $q\le 1$ and let $r=p_{n-1}$. Let us assume that $\XX$ has a greedy basis that we call $\BB$.
 Then, the $r$-Banach envelope $\Env[r]{\XX}$ of $\XX$ is $\ell_r\oplus \ell_{q}$ under the inclusion map, and $\Env[r]{\BB}$ is in particular a semi-normalized unconditional basis of $\Env[r]{\XX}$. The uniqueness of unconditional basis of $\ell_r\oplus \ell_{q}$ (see \cites{KLW1990, AKL2004}) yields 
\[
\varphi_l^\EE[\BB,\XX](m)\gtrsim \varphi_u^\EE[\BB,\XX](m)\gtrsim \varphi_u^\EE[\Env[r]{\BB}, \Env[r]{\XX}](m)\approx m^{1/r}, \quad m\in\NN.
\]
Pick $r<s<q$. Since
\[
\sum_{m=1}^\infty \frac{1}{(\varphi_l^\EE[\BB,\XX](m))^s}<\infty,
\]
 Proposition~\ref{prop:envlr} yields that $\Env[s]{\XX}\simeq \ell_s$, which is absurd because $\Env[s]{\XX}=\ell_s\oplus \ell_q$.

Suppose now that $q>1$ and put $r=\max\{1,p_{n-1}\}$. In this case the Banach envelope $\Env{\XX}$ of $\XX$ is (under the inclusion map) $\ell_r\oplus \ell_{q}$. Then, by \cite{EdWo1976}*{Theorem 4.11}, every unconditional basis of $\Env{\XX}$ splits into two subbases $\BB_1$ and $\BB_2$ which generate subspaces isomorphic to $\ell_r$ and $\ell_{q}$, respectively. Since $1\le r<q\le\infty$, an appeal to Corollary~\ref{cor:EWGideonIdea} puts and end to the proof.
\end{proof} 

Infinite direct sums of bases will also be of interest for us. Given a sequence $(\BB_k)_{k=1}^\infty$ of (finite or infinite) bases for Banach spaces $(\XX_k)_{k=1}^\infty$ we define its infinite direct sum in the following way: if $\BB_k=(\xx_{k,n})_{n\in J_k}$ we set $J=\{(k,n) \colon n\in J_k, k\in\NN\}$ and define $\oplus_{k=1}^\infty \BB_k=(\yy_j)_{j\in J}$ in $\Pi_{k=1}^\infty \XX_k$ by
\[
\yy_{k,n}=(f_k)_{k=1}^\infty, \quad f_j=\begin{cases} \xx_n & \text{ if }j=k \\ 0 & \text{ otherwise.} \end{cases}
\]

The following elementary lemma is the corresonding result to Proposition~\ref{prop:sumdemocracy} for infinite sums of bases.
\begin{lemma}\label{lem:infinitesum} Suppose $p\in(0,\infty]$. For $k\in\NN$ let $\BB_k$ be a basis of a quasi-Banach space $\XX_k$. Assume that there are constants
$C_1$ and $C_2$ such that for all $A\subseteq \NN$ finite and all $k\in\NN$,
\[
C_1 |A|^{1/p} \le \Vert \Ind_A[\BB_k,\XX_k]\Vert \le C_2 |A|^{1/p}.
\]
 Then $\BB=\oplus_{k=1}^\infty \BB_k$ is a democratic basis (basic sequence if $p=\infty$) of $\XX=(\oplus_{k=1}^\infty \XX_k)_p$ with 
\[
C_1 m^{1/p}\le \varphi_l[\BB,\XX](m) \le \varphi_u[\BB,\XX](m)\le C_2 m^{1/p}, \quad m\in\NN.
\]
\end{lemma}

\subsection{Greedy bases in Triebel-Lizorkin and Besov spaces}Given a dimension $d\in\NN$ we denote $\Theta_d=\{0,1\}^d \setminus\{ 0\}$ and we consider the set of indices 
\[
\Lambda_d= \ZZ \times \ZZ^d \times \Theta_d.
\]
The homogeneous Besov (respectively Triebel-Lizorkin) sequence space $\ring{\bv}_{p,q}^{s,d}$ (resp. $\ring{\tl}_{p,q}^{s,d}$) of indeces and $p$ and $q\in(0,\infty]$ and smoothness $s\in\RR$ consists of all scalar sequences $f=(a_\lambda)_{\lambda\in\Lambda}$ for which
\begin{align*}
\Vert f \Vert_{\bv_{p,q}^{s}}&= \left( \sum_{j=-\infty}^\infty 2^{jq(s+d(1/2-1/p))} \sum_{\delta\in\Theta_d} \left(\sum_{n\in\ZZ^d} |a_{j,n,\delta}|^p\right)^{q/p}\right)^{1/q}\\
\text{(resp. } \Vert f \Vert_{\tl_{p,q}^{s}}&= \left\Vert \left(\sum_{j=-\infty}^\infty \sum_{\delta\in\Theta_d} \sum_{n\in\ZZ^d} 2^{jq(s+d/2)} |a_{j,n,\delta}|^q \chi_{Q(j,n)}\right)^{1/q}\right\Vert_p\text{),}
\end{align*}
were $Q(j,n)$ denotes the cube of length $2^{-j}$ whose lower vertex is $2^{-j}n$. If we restrict ourselves to non-negative ``levels'' $j$ and we add $\ell_p$ as a component we obtain the inhomegeneous Besov and Triebel-Lizorkin sequence spaces. To be precise, set
\[
\Lambda_d^+=\{(j,n,\delta)\in\Lambda_d \colon j\ge 0\},
\]
and define
\begin{align*}
\bv_{p,q}^{s,d} & =\ell_p(\ZZ^d) \oplus \{ f=(a_\lambda)_{\lambda\in\Lambda_d^+} \colon \Vert f \Vert_{\bv_{p,q}^{s}}<\infty \},\\
\tl_{p,q}^{s,d} & =\ell_p(\ZZ^d) \oplus \{ f=(a_\lambda)_{\lambda\in\Lambda_d^+} \colon \Vert f \Vert_{\tl_{p,q}^{s}}<\infty \}.
\end{align*}
It is known that the wavelet transforms associated to certain wavelet bases normalized in the $L_{2}$-norm are isomorphisms from $F_{p,q}^s(\RR^d)$ (respectively $\ring{F}_{p,q}^s(\RR^d)$, $B_{p,q}^s(\RR^d)$ and $\ring{B}_{p,q}^s(\RR^d)$) onto $\tl_{p,q}^s(\RR^d)$ (resp., $\ring{\tl}_{p,q}^{s,d}$, $\bv_{p,q}^{s,d}$ and $\ring{\bv}_{p,q}^{s.d}$). See \cite{FrJaWe1991}*{Theorem 7.20} for the homegeneous case and \cite{TriebelIII}*{Theorem 3.5} for the inhomogenous case, the spaces $F_{\infty,q}(\RR^d)$ must be excluded. Thus, Triebel-Lizorkin and Besov spaces are isomorphic to the corresponding sequence spaces, and the aforementioned wavelet bases (regarded as distributions on Triebel-Lizorkin or Besov spaces) are equivalent to the unit vector systems of the corresponding sequence spaces. 

A similar technique to the one used by Temlyakov in \cite{Temlyakov1998} to prove that the Haar system is a greedy basis for $L_p$ when $1<p<\infty$ allows us to prove that Triebel-Lizorkin spaces have a greedy basis.

\begin{proposition}[cf. \cite{IzukiSawano2009}*{Theorem 16}]Let $d\in\RR^d$, $0<p<\infty$, $0<q\le \infty$ and $s\in\RR$. The normalized unit vector system of the spaces $\tl_{p,q}^{s,d}$ and $\ring{\tl}_{p,q}^{s,d}$ is a greedy basis (basic sequence if $q=\infty$) with fundamental function equivalent to $(m^{1/p})_{m=1}^\infty$.
\end{proposition}

\begin{proof}It suffices to prove that the unit vector system of the homogeneous space $\ring{\tl}_{p,q}^{s,d}$ is democratic in the case when $s+d(1/2)= d/p$, so that it is normalized.

For every finite set $A\subseteq\Lambda$ and every $x\in\RR^d$, the nonzero  terms  of the $\ell_q$-norm 
\[
F(x):=\left( \sum_{(j,n,\delta)\in A} 2^{jd/p} \chi_{Q(j,n)}(x)\right)^{1/q}
\]
belong to the geometric sequence $(2^{jd/p})_{j=-\infty}^\infty$, and a given term of this geometric sequence appears at most $2^d-1$ times in the expression defining $F(x)$. Hence, for $x\in\RR^d$,
\[
F(x)\approx\left(\sum_{(j,n,\delta)\in A} 2^{jd} \chi_{Q(j,n)}(x)\right)^{1/p}.
\]
Raising to the $p$th-power and integrating on $\RR^d$ finishes the proof.
\end{proof}

The behavior of Besov spaces is quite different. Indeed, every sequence Besov space of indeces $p$ and $q$ is naturally isomorphic to the mixed-norm space $\ell_q(\ell_p)$ and so if $p\not=q$ its unit vector system cannot be greedy because it is not democratic. Here we take one step forward and extend the main result from \cite{Gideon2014} to the whole range of indices $p$, $q\in (0, \infty]$.
\begin{proposition}\label{NogreedyZpq} Let $0<p\not= q\le \infty$. Then $\ell_q(\ell_p)$ has no greedy basis (we replace $\ell_\infty$ with $c_0$ if some of the indices is $\infty$) .
\end{proposition}
\begin{proof} The case when $s:=\max\{p,q\} \le 1$ follows as a consequence of the uniqueness of unconditional basis of the space $\XX:=\ell_p(\ell_q)$ proved in \cites{KLW1990, AKL2004}.

Assume that $s>1$ and set $r=\max\{1,\min\{p,q\}\}$. The Banach envelope $\Env{\XX}$ of $\XX$ is, with the usual modifications if $s=\infty$, either $\ell_s(\ell_r)$ or $\ell_r(\ell_s)$ under the inclusion map. Therefore, by \cite{Gideon2014}*{Theorem 1}, every unconditional basis of $\Env{\XX}$ has two subbases $\BB_1$ and $\BB_2$ which are equivalent to the unit vector system of $\ell_r$ and $\ell_s$, respectively. Since $1\le r <s\le\infty$, appealing to Corollary~\ref{cor:EWGideonIdea} the proof is over.
\end{proof}
\subsection{Almost greedy bases in Besov spaces}
Now we introduce a different  family  of Besov sequence spaces. Set
\[
B_{p,q}=(\oplus_{n=1}^\infty \ell_p^n)_q, \quad 0<p\le\infty, \, 0\le q <\infty.
\]
These spaces are isomorphic to Besov spaces over  $[0,1]^d$ (see \cites{DVP1988,AA2017+}), and are mutually non-isomorphic, with the only exception of the case $B_{q,q}\simeq B_{2,q}$ for $1<q<\infty$ (see \cite{AA2017+}).  In the locally convex setting, the existence of greedy bases in these spaces was studied in \cites{DOSZ2011}. Since greedy bases are in particular almost greedy, in this section we go further and provide a couple of results on the existence of almost greedy bases in nonlocally convex spaces $B_{p,q}$.

\begin{proposition}[cf. \cite{AADK2019}*{Proposition 4.21}]\label{prop:AADK} Let $(\XX_n)_{n=1}^\infty$ be a sequence of finite-dimensional quasi-Banach spaces, and let $q\in[0,\infty)$. If  $\BB$ is a super-democratic basic sequence in $\XX=(\oplus_{n=1}^\infty \XX_n)_q$, then:
\begin{itemize}
\item[(i)] If $q>0$, $\varphi_u^\EE[\BB,\XX]\approx m^{1/q}$ for $m\in\NN$.
\item[(ii)] If $q=0$, $\BB$ is equivalent to the unit vector system of $c_0$.
\end{itemize}
\end{proposition}

\begin{proof}The proof  for locally convex spaces from \cite{AADK2019} works for $0<q<1$ and $\XX_n$ quasi-Banach.
\end{proof}

\begin{lemma}\label{lem:AGFiniteSums}Let $0<q<1$ and $(\XX_n)_{n=1}^\infty$ be a sequence of finite dimensional Banach  quasi-Banach spaces. If the space $\XX:=(\oplus_{n=1}^\infty\XX_n)_q$ has an almost greedy basis, then for  $q<r\le 1$ we have 
\[(\oplus_{n=1}^\infty\Env[r]{(\XX_n)})_r\simeq \ell_r. \]
\end{lemma} 
\begin{proof} If
$\BB$ is an almost greedy basis of $\XX$, by Proposition~\ref{prop:AADK}~(i), $\varphi_u^\EE[\BB,\XX]\approx m^{1/q}$ for $m\in\NN$ and so, by Proposition~\ref{prop:envlr}, the $r$-Banach envelope of $\XX$ is isomorphic to $\ell_r$. 
The proof is over by Proposition~\ref{prop:envelopesum}.
\end{proof}

\begin{proposition}\label{prop:NQGBesov} The space $B_{p,q}$ has no almost greedy basis for $0<q<1$ and $q<p\le\infty$.
\end{proposition} 
\begin{proof}
Assume by contradiction that $\BB$ is an almost greedy basis of $B_{p,q}$. Pick $q<r<\min\{p,1\}$. Then, by Lemma~\ref{lem:AGFiniteSums} and Corollary~\ref{cor:envelopes}, $B_{p,r}\simeq \ell_r$. By  \cite{AA2017+}*{Theorem 1.2}, this an absurdity.
\end{proof}

\begin{proposition}\label{prop:NQGBesov+} The space $B_{p,0}$, $0<p<\infty$, has no almost greedy basis.
\end{proposition}

\begin{proof}Just combine  Proposition~\ref{prop:AADK}~(ii) with Theorem 1.2 from \cite{AA2017+}.
\end{proof}

\subsection{Democratic bases that fail to be SUCC}
\begin{example}
Given $0<p<\infty$ we consider the James-type quasi-norm
\[
\Vert f \Vert_{(p)}=\sup_{\phi\in\OO} \left(\sum_{k=1}^\infty |a_{\phi(k)}-a_{\phi(k-1)}|^p\right)^{1/p}, \quad f=(a_n)_{n=1}^\infty\in\FF^\NN.
\]
Here we are using the convention $\phi(0)=0$ and $a_0=0$. If $1<p<\infty$ the space 
\[\XX_{(p)}=\{f\in\FF^\NN \colon \Vert f\Vert<\infty\}\] is the James quasi-reflexive space $\JJ^{(p)}$. 

For $0<p\le 1$, the space $\XX_{(p)}$ behaves quite differently. In fact, since 
\[|a_{n}-a_{k}|\le \left(\sum_{j=k+1}^n |a_j-a_{j-1}|^p\right)^{1/p}\] we have
\[
\Vert f \Vert_{(p)}=\left(\sum_{n=1}^\infty |a_n-a_{n-1}|^p\right)^{1/p}, \quad f=(a_n)_{n=1}^\infty\in\FF^\NN, \, 0\le p\le 1.
\]
Hence $\XX_{(1)}$ is the space consisting of all sequences of bounded variation, usually denoted by $v_1$. By analogy, we will denote $\XX_{(p)}$ by $v_p$ for $p\le 1$. Since the mapping $P\colon v_p\to \ell_p$ defined by $(a_n)_{n=1}^\infty\mapsto (a_n-a_{n-1})_{n=1}^\infty$ is an isometry, the space $v_p$ is nothing but $\ell_p$ in a rotated position. Since $|a_n|\le \Vert f \Vert_{(p)}$ for every $f=(a_n)_{n=1}^\infty\in\FF^\NN$, we readily infer that the unit vector system is a Schauder basis of $v_p$ with basis constant $2^{1/p}$. 

Consider the ``alternating basis'' $\BB=(\xx_n)_{n=1}^\infty$ of $v_p$ given by $\xx_n=(-1)^{n-1}\, \ee_n$ for every $n\in\NN$. Notice that for $f=\sum_{n=1}^\infty a_n\, \xx_n\in v_p$,
\[
\Vert f \Vert_{(p)} =\left(\sum_{n=1}^\infty |a_n+a_{n-1}|^p\right)^{1/p}, \quad \, 0<p\le 1. 
\]
Then, if $\Fou$ denotes the coefficient transform with respect to $\BB$ defined as in \eqref{eq:Fourier}, in the case when $\Fou(f)\ge 0$ we have
\[
\Vert f \Vert_{(p)} \le 2^{1/p} \Vert \Fou(f)\Vert_p,\quad f\in v_p,
\]
and
\[2 \Vert \Fou(f) \Vert_p \le \Vert f \Vert_{(p)},\quad f\in v_p.\]
We deduce that $\BB$ is a $2^{1/p}$-democratic basis of $v_p$ with
\[
\varphi_u^\EE[\BB, v_p](m) \approx m^{1/p}, \quad m\in\NN.
\] 
However, since $\Vert \sum_{n=1}^m (-1)^{n-1} \ee_n\Vert_{(p)} =2$ for every $m\in\NN$, 
\[
\varphi_l^\EE[\BB, v_p](m) \approx 1, \quad m\in\NN.
\]
Therefore, for any $0<p\le 1$, $\BB$ is not a superdemocratic basis of $v_p$.
\end{example}

Another example of a democratic basis that is not superdemocratic is the summing basis of $c_0$. Next, in order to clear up the fact that the geometry of the space has no effect on the existence of such bases, we construct an $M$-bounded total basis in a Hilbert space.

Our construction relies on the following elementary lemma.

\begin{lemma} Let $\XX$ be the space $\RR^2$ endowed with the standard Euclidean norm.  For each $n\in \mathbb N$ there exist vectors $\ah_1,\ah_2\in \XX$ such that
\begin{equation}
\|\ah_1\|=\|\ah_2\|=\sqrt n\; \text{ with }\; \|\ah_1-\ah_2\|=1 \label{eq:orthopairnorm}
\end{equation}
and 
\begin{equation}
\|\alpha_1 \ah_1+\alpha_2\ah_2\|\geq \sqrt n \sqrt{\alpha_1^2+\alpha_2^2}\label{eq:HilbertBound}
\end{equation}
whenever $\alpha_1\bar\alpha_2\geq 0$.  Moreover the vectors $\ah_1^*$, $\ah_2^*$ given by $\langle \ah_j^*, \ah_s\rangle=\delta_{j,s}$ satisfy
\begin{equation}
\|\ah_1^*\|=\|\ah_2^*\|=\sqrt{\frac{n}{n-\frac14}}.
\end{equation}
\end{lemma}
\begin{proof}
In the canonical system in $\RR^2$ we take $\ah_1=(\sqrt{n-\frac14},\frac12)$ and $\ah_2=(\sqrt{n-\frac14},-\frac12)$. The rest is a straightforward calculation.
\end{proof}

\begin{example} Consider the Euclidean space $\FF^{n}$ with the unit vector basis $(\ee_j)_{j=1}^{n}$. Put $\ah=\sum_{j=1}^{n} \ee_j$. Let $\Hb_n$ be the $(2n)$-dimensional Euclidean space $\XX_1\oplus \XX\oplus \XX_2$ where $\XX_1$ and $\XX_2$ are $(n-1)$-dimensional euclidean spaces. Let $\Hb_{n}^1=\XX_1\oplus [\ah_1]$ and $\Hb_{n}^2= [\ah_2] \oplus \XX_2$. For $s\in\{ 1,2\}$ let  $T_s\colon \FF^{n}\to \Hb_{n}^s$  be an isometry with $T_s(\ah)=\ah_s$. Since $\Hb_n=\Hb_n^1\oplus\Hb_n^2$ algebraically, the sequence
\[
\BB_n:=(T_s(\ee_j) \colon s=1,2, \, j=1,\dots,n).
\]
is a basis of $\Hb_n$. Let $\ee_j^s=T_s(\ee_j)$ for $(s,j)\in \Theta_n:=\{1,2\}\times\{1,\dots, n\}$. Note that, by construction,
\begin{equation}\label{eq:projectionsHilbert}
\sum_{j=1}^n \ee_j^s =\ah_s, \quad s=1,2.
\end{equation}

For $s\in\{1,2\}$ let $\pi_s$ be the canonical projection of $\Hb_n$ onto $\XX_s$, and let $\pi$ be the canonical projection of $\Hb_n$ onto $\XX$.
Since the orthogonal projection of 
$\FF^n$ onto $[\ah]=\FF\ah$ is given by 
\[
\sum_{j=1}^n a_j\, \ee_j\mapsto \frac{1}{n}\left(\sum_{j=1}^n a_j\right) \ah,
\]
and  $\ee_j^s\in \Hb_n^s\subseteq\XX_s\oplus\XX$ it follows that for  $j=1,\dots, n$ and $s\in\{1,2\}$,
\begin{align}
\label{eq:Wo2} 
\pi_t(\ee_j^s)&=0, \quad t\in\{1,2\}, \, t\not=s\\
\label{eq:Wo3}
\pi(\ee_j^s)&= \frac{\ah_s}{n}.
\end{align}

Let us analyse the coordinate functionals $\BB_n^*=(\psi_j^s)_{(s,j)\in\Theta_n}$ of $\BB_n$. Note that for every $\psi\in\Hb_n^*$,
\begin{align*}
\|\psi\|^2 
&=\|\psi|\XX_1\|^2+\|\psi|\XX\|^2+\|\psi|\XX_2\|^2 \\
&\le \|\psi|\Hb_1\|^2+\|\psi|\XX\|^2+\|\psi|\Hb_2\|^2\\
&=\Vert \psi(\ah_1) \ah^*_1+\psi(\ah_2) \ah^*_2\Vert^2 +\sum_{(s,j)\Theta_n} |\psi(\ee_j^s)|^2.
\end{align*}
By \eqref{eq:projectionsHilbert}, $\psi_j^s(\ah_t)=\delta_{s,t}$ for $s$, $t\in\{1,2\}$ and $j=1$,\dots, $n$. Using \eqref{eq:orthopairnorm} we obtain
\begin{equation}\label{eq:MbddHilbert}
1\le \|\psi_j^s\|^2\le  \frac{n}{n-\frac14}+1\le 3.  
\end{equation}

Let also analyse the basis constant $K_n:=K[\BB_n,\XX_n]$ with respect to an arbitrary ordering of $\BB_n$. For $A\subseteq\Theta_n$ and $s\in\{1,2\}$, consider
\[
A_s=\{(t,j)\in A \colon t=s\}. 
\]
Given $0\le k \le n$ there is $0\le m \le 2n$ for which
the partial-sum projection $S_m$ coincides with the coordinate projection on a set   $A$ such that the cardinality of    $A_1$ is $k$. If $R=\pi_1\circ S_A|_\XX$ we have
\[
\Vert R \Vert = \Vert \pi_1\circ S_m|_\XX\Vert \le \Vert S_m\Vert \le K_n.
\]
By \eqref{eq:Wo2}, 
\[
R(\ah_2)=0\;\text{ and }\; R(\ah_1)=\pi_1(S_{A_1}(\ah_1)),\]
  and also  
  \[
  \pi_2(S_{A_1}(\ah_1))=0.\]
   From    \ref{eq:Wo3} we get 
   \[
   \pi(S_{A_1}(\ah_1))=\frac{|A_1|}{n} \ah_1,\]
    and from \ref{eq:projectionsHilbert} we obtain
    \[\|S_{A_1}(\ah_1)\|=\sqrt{|A_1|}.\] 
This yields
\begin{align*}
\|R\|&\geq  \frac {\|R(\ah_1-\ah_2)\|}{\|\ah_1-\ah_2\|}\\
&=\|\pi_1(S_{A_1}(\ah_1))\|\\
&=\sqrt{\|S_{A_1}(\ah_1)\|^2-\|\pi(S_{A_1}(\ah_1))\|^2}\\
&= \sqrt{|A_1|-\frac{|A_1|^2}{n^2} n}.
\end{align*}
Choosing  $k=\lfloor n/2 \rfloor$ gives
 \begin{equation}K_n\geq \frac{\sqrt 2}{3}\sqrt n.
\label{eq:basisconstantbound}
\end{equation}

Finally, we estimate the democracy functions of $\BB_n$. Let $A\subseteq\Theta_n$. By \eqref{eq:Wo2}, \eqref{eq:Wo3} and \eqref{eq:HilbertBound} we have
\begin{align*}
\Vert \Ind_A\Vert^2
&=\Vert \pi_{1} (\Ind_A) \Vert^2+\Vert \pi (\Ind_A) \Vert^2 +\Vert \pi_{2} (\Ind_A) \Vert^2\\
&=\Vert \pi_{1} (\Ind_{A_1}) \Vert^2 +\left\Vert \frac{|A_1|}{n}\ah_1 + \frac{|A_2|}{n}\ah_2\right\Vert^2 +\Vert \pi_{2} (\Ind_{A_2}) \Vert^2\\
&\ge\Vert \pi_{1} (\Ind_{A_1}) \Vert^2+ \frac{1}{n}\left(|A_1|^2+ |A_2|^2\right)+ \Vert \pi_{2} (\Ind_{A_2}) \Vert^2 \\
&=\Vert \pi_{1} (\Ind_{A_1}) \Vert^2 +\left\Vert \pi (\Ind_{A_1}) \right\Vert^2 + \left\Vert \pi (\Ind_{A_2})\right\Vert^2+ \Vert \pi_{2} (\Ind_{A_2}) \Vert^2\\
&=\left\Vert \Ind_{A_1} \right\Vert^2 + \left\Vert \Ind_{A_2} \right\Vert^2\\
&=|A_1| + |A_2|\\
&=|A|.
\end{align*}

Conversely, by the triangle law,
\begin{align*}
\Vert \Ind_A \Vert
&= \left\| \Ind_{A_1} + \Ind_{A_2} \right\| \\
&\leq \left\| \Ind_{A_1} \right\|+\left\| \Ind_{A_2} \right\|\\
&=\sqrt{|A_1|}+\sqrt{|A_2|}\\
&=\sqrt 2\sqrt{|A_1|+|A_2|}.
\end{align*}

Now, let $\varepsilon=(\varepsilon_{s,j})_{(j,s)\in\Theta_n}\in\EE_\Theta$ be  the sequence of signs given by $\varepsilon_{s,j}=(-1)^j$. By \eqref{eq:projectionsHilbert} and \eqref{eq:orthopairnorm} we have
\[
\Vert \Ind_{\varepsilon,\Theta_n}\Vert =\left\Vert-\sum_{j=1}^{n}\ee^1_j+\sum_{j=1}^{n} \ee^2_j\right\Vert=\|\ah_2-\ah_1\|=1.
\]
By \eqref{eq:MbddHilbert}, the direct sum $\BB$ of the bases $(\BB_n)_{n\in \NN}$, is an $M$-bounded basis of the the Hilbert space $\Hb:=(\oplus_{n=2}^\infty \Hb_n)_2$. From \eqref{eq:basisconstantbound} we deduce that there is no ordering of $\BB$ with respect to which $\BB$ becomes a Schauder basis. Putting together all the above estimates and taking into account Lemma~\ref{lem:infinitesum} we obtain
\[
m^{1/2}\le \varphi_l[\BB,\Hb](m)\le \varphi_u[\BB,\Hb](m)\le \sqrt{2} m^{1/2},
\]
and
\[
\varphi_l^\EE[\BB, \Hb](m)\le 1
\]
for every $m\in\NN$. Thus $\BB$ is democratic but not superdemocratic.
\end{example}

\subsection{Superdemocratic bases that fail to be SLC and LUCC}\label{SUCCnotLUCC}
Our source of inspiration in this Section is Example 4.8 from \cite{BDKOW}, which we generalize in the following Example. 
\begin{example}
Let $\XX$ be the direct sum $\ell_p\oplus\ell_q$ for $0<p<q\le\infty$, with the understanding that $\XX=\ell_p\oplus c_0$ when $q=\infty$. The sequence $\BB=(\xx_n)_{n=1}^\infty$ given by
\[
\xx_{2k-1}=(\ee_k,\ee_k), \quad \xx_{2k}=\left( \frac{1}{2} \ee_k,\ee_k\right), \quad k\in\NN.
\]
is clearly a normalized Schauder basis of $\XX$.

For $A\subseteq\NN$ let 
\[
B=\{k\in\NN \colon \{2k-1,2k\}\subseteq A\},
\] 
\[
B_o=\{k\in\NN \colon 2k-1\in A, 2k\notin A\}
\]
and 
\[
B_e=\{k\in\NN \colon 2k-1\notin A, 2k\in A\}.
\]
If $\pi_1$ and $\pi_2$ denote respectively the projections of $\XX$ onto $\ell_p$ and $\ell_q$, we have 
\[
N_1:=\left\Vert \pi_1\left( \sum_{n\in A} \Ind_{\varepsilon,A} \right)\right\Vert_p^p
=\sum_{k\in B} \left|\varepsilon_{2k-1}+ \frac{1}{2} \varepsilon_k\right|^p +|B_o|+
 \left(\frac{1}{2}\right)^p |B_e|
\]
and
\[
N_2:=\left\Vert \pi_2\left( \sum_{n\in A} \Ind_{\varepsilon,A} \right)\right\Vert_q^q
=\sum_{k\in B} \left|\varepsilon_{2k-1}+ \varepsilon_k\right|^q +|B_o|+ |B_e|.
\]
Hence
\[
\left(\frac{1}{2}\right)^p |B|+|B_o|+  \left(\frac{1}{2}\right)^p |B_e| \le N_1 \le \left(\frac{3}{2}\right)^p |B|+|B_o|+  \left(\frac{1}{2}\right)^p |B_e|
\]
and
\[
N_2\le 2^q|B|+|B_o|+|B_e|.
\]
Since $|A|=2|B|+|B_o|+|B_e|$, it follows that 
\[
\frac{1}{2^{p+1}} |A| \le N_1\le \max\left\{1, \frac{3^p}{2^{p+1}}\right\} |A|,
\]
and
\[
N_2\le \max\{2^{q-1},1\}|A|.
\]
Hence,
\[
\frac{1}{2^{1+1/p}} |A|^{1/p} \le \left\Vert \sum_{n\in A} \Ind_{\varepsilon,A} \right\Vert \le \max\left\{ 1, 2^{1-1/q}, \frac{3}{2^{1+1/p}}\right\} |A|^{1/p}.
\]
and so $\BB$ is superdemocratic (thus SUCC).

Set
\begin{align*}
f_m&=-\sum_{k=1}^m \xx_{2k-1}+2\sum_{k=1}^m \xx_{2k}=\left(0,\sum_{k=1}^m \ee_k\right),\\
g_m&=-\sum_{k=1}^m \xx_{2k-1}+\sum_{k=1}^m \xx_{2k}.\\
h_m&=-\sum_{k=1}^m \xx_{2k-1}.
\end{align*}
Since $\Vert f_m\Vert=m^{1/q}$ and $\Vert g_m\Vert \approx \Vert h_m\Vert\approx m^{1/p}$ for $m\in\NN$, we have
\[
\lim_m \frac{ \Vert g_m \Vert}{ \Vert f_m \Vert}= \lim_m \frac{ \Vert h_m \Vert}{ \Vert f_m \Vert}=\infty.
\]
So, $\BB$ is neither LUCC nor QGLC.
\end{example}

\subsection{Democracy does not transfer to the Banach envelope}\label{sect:democracyenvelopes} 
The aim of this section is to build an example of a democratic basis in a quasi-Banach space $\XX$ which is not democratic as a basis of the Banach envelope $\Env{\XX}$ of $\XX$. Our construction relies on the following lemma.

\begin{lemma}\label{lem:subsyml1} Let $\XX\subseteq \FF^\NN$ be a quasi-Banach space for which the unit vector system $\BB_e$ is a subsymmetric basis.
\begin{itemize}
\item[(i)]Suppose that $\XX$ is locally convex and $\varphi_u[\BB_e,\XX](m) \approx m$ for $m\in\NN$.
Then $\XX=\ell_1$ (up to an equivalent norm).
\item[(ii)]Suppose that $\varphi_u[\BB_e,\Env{\XX}](m) \approx m$ for $m\in\NN$. Then $\XX\subseteq\ell_1$.
\end{itemize}
\end{lemma}

\begin{proof} (i) By \cite{LinTza1977}*{Proposition 3.a.6}, $\BB_e$ is a subsymmetric basis of $\XX^*$ with bounded fundamental function. Hence $\BB_e$ is equivalent to the unit vector system of $c_0$ as a basis of $\XX^*$ and so $\BB_e$ is equivalent to the unit vector system of $\ell_1$ as a basis of $\XX$.

(ii) By Corollary~\ref{cor:subsymenvelope}, $\BB_e$ is a subsymmetric basis of $\Env{\XX}$.
Therefore, using part~(i), we obtain
$
\XX \subseteq \Env{\XX}=\ell_1.
$
\end{proof}

\begin{example}\label{ex:envnotdemocratic} Let $\XX\subseteq\FF^\NN$ be a quasi-Banach space such that $\ell_1\not\subset \XX$ and the unit vector system $\BB_e$ is a subsymmetric basis of $\XX$ with fundamental function equivalent to $(m)_{m=1}^\infty$. 
Then, by Proposition~\ref{prop:sumdemocracy}, $\BB_e\oplus\BB_e$ is a greedy basis of $\ell_1\oplus \XX$ with fundamental function equivalent to $(m)_{m=1}^\infty$. By Corollary~\ref{cor:envelopes},  the Banach envelope of $\BB_e\oplus\BB_e$, as a basis of $\ell_1\oplus \XX$, is the very $\BB_e\oplus\BB_e$ regarded as a basis of $\ell_1\oplus \Env{\XX}$. By Lemma~\ref{lem:subsyml1}, $\varphi_u[\BB_e,\Env{\XX}]\not\approx m$ for $m\in\NN$.  Therefore, by Proposition~\ref{prop:sumdemocracy}, the Banach envelope of the canonical basis of $\ell_1\oplus \XX$ is not greedy since it is not democratic.
\end{example}

\begin{example} Taking into consieration Proposition~\ref{prop:convexityLorentz}~(i), as
 a particular case of Example~\ref{ex:envnotdemocratic} we can pick $\XX=\ell_{1,q}$, $1<q<\infty$, and $\XX=\ell_{1,\infty}^0$. Note that by Proposition~\ref{prop:lorentzdual}, if $\uu_0=(1/n)_{n=1}^\infty$,
\[
(\ell_{1,\infty}^0)^{**}=m(\uu_0)
\]
under the inclusion map. Therefore
\[
\Env{(\ell_{1,\infty}^0)}=m_0(\uu_0)
\]
under the inclusion map. Therefore the fundamental function of $\BB_e$ as a basis of $\Env{(\ell_{1,\infty}^0)}$ is equivalent to $(m/\log(1+m))_{m=1}^\infty$.
\end{example}

\begin{example}In order to obtain quasi-Banach spaces $\XX$ such that $\varphi_u[\BB_e,\Env{\XX}] \not\approx \varphi_u[\BB_e,\XX] \approx m$ for $m\in\NN$ we can also apply Lemma~\ref{lem:subsyml1} to some Garling sequence spaces. For $0<\alpha\le 1$, let $\uu_\alpha$ be the potential weight defined in \eqref{eq:potential}. Note that $\uu_\alpha\in \WW$ and that for every $0<p<\infty$ the fundamental function of the unit vector system of $g(\uu_\alpha,p)$  is
\begin{equation*}
\varphi_u^\EE[\BB_e, g(\uu_\alpha,p)](m)
=\left(\sum_{n=1}^m n^{\alpha-1}\right)^{1/p}, \quad m\in\NN.
\end{equation*} 
Consequently,
\begin{equation}\label{eq:EstimateGarlingfundamental}
\varphi_u^\EE[\BB_e, g(\uu_\alpha,p)](m)\approx m^{\alpha/p}, \quad m\in\NN.
\end{equation}
\begin{proposition}\label{prop:GarlingNotEmbedinl1}Suppose $0<p<1$. Then $g(\uu_{p},p)\not\subset \ell_1$.
\end{proposition}
\begin{proof} By Lemma~\ref{Lemma:1}, for each $N\in\NN$ there exist positive constant-coefficient tuples $(f_j)_{j=1}^N$ such that $\lambda_j:=\Vert f_j\Vert_{g(\uu_{p},p)}\ge 1/2$ and
\[
\left \Vert f_N \smallfrown \dots \smallfrown f_j \smallfrown\dots \smallfrown f_1\right\Vert_{g(\uu_{p},p)}=1.
\]
For each $m\in\NN$ let $\hhh_m$ denote the positive constant coefficient $m$-tuple of sum one, that is,
\[
\hhh_m=\left(\underbrace{\frac{1}{m},\dots,\frac{1}{m},\dots,\frac{1}{m}}_{m \text{ times}}\right).
\]
If $m_j$ denotes the length of $f_j$ and we let $\varphi=\varphi_u^\EE[\BB_e, g(\uu_p,p)]$ we have
\[
\hhh_{m_j}=\frac{\varphi(m_j)}{\lambda_j m_j} f_j, \quad j=1,\dots,N.
\]
  By \eqref{eq:EstimateGarlingfundamental} there is a constant $C$ depending only on $p$ such that ${\phi(m_j)}/{(\lambda_j m_j)}\le C$  for all $j=1$, \dots, $N$.  Thus, 
if we put
\[
h= \hhh_{m_N} \smallfrown \dots \smallfrown \hhh_{m_j} \smallfrown\dots \smallfrown \hhh_{m_1},
\]
we have $\Vert h \Vert_{g(\uu_{p},p)}\le C$. Since $\Vert h \Vert_1=N$ we are done.
\end{proof}
\begin{remark}Proposition~\ref{prop:GarlingNotEmbedinl1} brings forward an important structural difference between $g(\uu_{p},p)$ and its symmetric counterpart, the space $\ell_{1,p}$. Notice that, since $\Vert f\Vert_{g(\uu_{p},p)}=\Vert f \Vert_{\ell_p(\uu_p)}=\Vert f \Vert_{\ell_{1,p}}$ for $f$ positive and non-increasing, we also have $\ell_1\not\subset g(\uu_{p},p)$ for $0<p<1$.
\end{remark}
\end{example}

\begin{example}The behavior of nonlocally convex Garling sequence spaces  $g(\uu_\alpha,p)$ for $\alpha\not=p$ is quite different. If $0<p<\alpha\le 1$, combining Proposition~\ref{prop:envlr} with  equation \eqref{eq:EstimateGarlingfundamental} yields that the Banach envelope of $g(\uu_\alpha,p)$ is $\ell_1$. If $0<\alpha<p<1$, by  equation \eqref{eq:EstimateGarlingfundamental}, $\varphi_u^\EE[\BB_e,g(\uu_\alpha,p)]$ is a regular weight. Proposition~\ref{prop13a} now yields that
\[
\varphi_u^\EE[\BB_e, g(\uu_\alpha,p)] \approx \varphi_u^\EE[\BB_e, \Env{(g(\uu_\alpha,p))}]
\] 
and that the unit vector system of $g(\uu_\alpha,p)$ is bidemocratic. 
\end{example}

\subsection{LPU and democratic bases that fail to be quasi-greedy}

Our examples in this section are modeled after a method for constructing quasi-Banach spaces that goes back to Konyagin and Temlyakov's article \cite{KoTe1999}.

Let $\ww=(w_n)_{n=1}^\infty$ be a non-increasing weight. Given $ f=(a_n)_{n=1}^\infty\in\FF^\NN$ we put
\[
\Vert f \Vert_\ww=\sup_m \left|\sum_{n=1}^{m} a_n\, w_n \right|, 
\]
and then define the Banach space $s_\ww$ as
\[s_\ww=\{ f\in\FF^\NN \colon \Vert f \Vert_\ww<\infty\}.\]

Of course, the mapping $(a_n)_{n=1}^\infty\mapsto (\sum_{n=1}^{m} a_n\, w_n)_{m=1}^\infty$ restricts to an isometry from $s_\ww$ onto $\ell_\infty$. 

Let $\XX$ be a $p$-Banach space with a greedy basis $\BB$ whose fundamental function is equivalent to the primitive weight $\sss$ of $\ww$. Then, by Theorem~\ref{thm:embedding4}, $\XX$ is sandwiched between $d_{1,p}(\ww)$ and $d_{1,\infty}(\ww)$ via $\BB$. Assume without loss of generality that $\BB$ is the unit vector system and that $\XX\subseteq\FF^\NN$. Then $d_{1,p}(\ww)\subseteq \XX\subseteq d_{1,\infty}(\ww)$ continuously. Merging those two ingredients we consider the space
\[
\KT[\XX,\ww]=\XX\cap s_\ww=\{ f \in \XX \colon \Vert f\Vert_{\KT[\XX,\ww]}<\infty\},
\] 
where for $f\in\XX$,
\[
\Vert f\Vert_{\KT[\XX,\ww]}=\max\left\lbrace \| f \|_{\XX} , \Vert f \Vert_\ww \right\rbrace.
\]
The convexity of the space $\KT[\XX,\ww]$ is at least the same as the convexity of $\XX$; in particular $\KT[\XX,\ww]$ is a $p$-Banach space. 

If $s_m=\varphi_u[\BB_e,\XX](m)$ for every $m\in\NN$ we simply write $\KT[\XX,\ww]=\KT[\XX]$. The example in \cite{KoTe1999}*{$\S$3.3} is the case $\KT[\ell_2]$, while $\KT[\ell_p]$, $1<p<\infty, $ was later considered in \cite{GHO2013}. The case in which $\XX$ is a Lorentz space was studied in \cite{BBGHO2018}. We will refer to this method for building quasi-Banach spaces as the KT-method.

Let us gather together some properties of the spaces $s_\ww$ that will be required in our study of the spaces $\KT[\XX,\ww]$. It is easy to check that the unit vector system is a monotone basic sequence of $s_\ww$. It is satisfied that
\begin{equation}\label{eq:kt0}
\Vert f \Vert_\ww
\le \Vert f\Vert_{1,1,\ww}\quad \text{for all } f\in \FF^\NN.
\end{equation}

If $(a_n)_{n=1}^\infty\in\FF^\NN$ is non-negative and non-increasing, the quasi-norm $\Vert ((-1)^{n-1}a_n)_{n=1}^\infty\Vert_\ww$ is the supremum of the partial sums of the  ``alternating''  series $\sum_{n=1}^\infty (-1)^{n-1} a_n$. Then, 
\begin{equation}\label{eq:kt2}
\Vert ((-1)^{n-1}a_n)_{n=1}^\infty\Vert_\ww=a_1 \, w_1, \quad a_n\searrow 0.
\end{equation}

Let us return to the $\KT[\XX,\ww]$ spaces. From   \eqref{eq:kt0}  we deduce that we always have 
\begin{equation}\label{eq:KTembedding}
d_{1,p}(\ww)\subseteq\XX\subseteq \KT[\XX,\ww] \subseteq d_{1,\infty}(\ww).
\end{equation}
We also infer from   \eqref{eq:kt0}  that 
$d_{1,1}(\ww)\subseteq \KT[\XX,\ww]$ {if and only if} $d_{1,1}(\ww)\subseteq \XX.$

The definition of $\KT[\XX,\ww]$ yields that the unit vector system $\BB_e$ is a basic sequence of $\KT[\XX,\ww]$ with basis constant no larger than that of $\BB_e$ seen as a basis of $\XX$. The inclusions in \eqref{eq:KTembedding} combined with Corollary~\ref{cor:embedding2} show that the unit vector system of $\KT[\XX,\ww]$ is democratic and that the truncation operator with respect to it is bounded. Therefore the unit vector system of $\KT[\XX,\ww]$ is SLC and LUCC.

If $\XX\subseteq d_{1,1}(\ww)$ then $\KT[\XX,\ww]=\XX$, and so the unit vector system is an unconditional basis of $\XX$. The converse also holds under a mild condition on $\XX$ and $\ww$.

\begin{lemma}\label{lem:KTConditional}Let let $\ww$ be a non-increasing weight such that $1/\ww$ is doubling and let $\XX\subseteq\FF^\NN$ be a quasi-Banach space for which the unit vector system $\BB_e$ is a basis  with fundamental function equivalent to the primitive weight of $\ww$.
\begin{itemize}
\item[(i)] If $\BB_e$, regarded as basis of $\XX$, is a greedy basis equivalent to its square and $\BB_e$, regarded a basis of $\KT[\XX,\ww]$, is unconditional, then $ \XX\subseteq \ell_1(\ww)$.

\item[(ii)]Assume that $\BB_e$, regarded as basis of $\XX$, is symmetric. Then $\BB_e$, regarded as a basis of $\KT[\XX,\ww]$, is unconditional if and only if $ \XX\subseteq d_{1,1}(\ww)$.
\end{itemize}
\end{lemma}
\begin{proof}
(i) Let $C_1<\infty$ be such that 
\[
w_{n}\le C_1 w_{2n-1},\quad n\in\NN.
\]
 Let $C_2<\infty$ be such that 
\[
\Vert f\oplus g \Vert \le C_2\max\{ \Vert f \Vert , \Vert g \Vert \},\quad f,g\in\XX.
\]
Here, if $f=(a_{n,1})_{n=1}^\infty$ and $g=(a_{n,0})_{n=1}^\infty$, 
\[
f\oplus g =( a_{\lceil n \rceil, \lceil n \rceil-n})_{n=1}^\infty.
\]
Let $C_3<\infty$ be such that 
\[
|a_n|\le C_3 \Vert f\Vert_\XX,\quad f=(a_n)_{n=1}^\infty\in\XX.
\]
Suppose $\BB_e$ is $C$-lattice unconditional when regarded as a basis of $\KT[\XX,\ww]$.
Then for every $f=(a_n)_{n=1}^\infty\in\XX$ we have
\begin{align*}
\Vert f \Vert_{\ell_1(\ww)} 
&=\Vert \, |f| \, \Vert_\ww\\
&\le C_1 \Vert \, |f|\oplus 0\, \Vert_\ww \\
&\le C_1 \Vert \, |f|\oplus 0\, \Vert_{\KT[\XX,\ww]} \\
&\le C C_1 \Vert f\oplus (-f) \Vert_{\KT[\XX,\ww]} \\
&\le C C_1 \max\left\{ C_2 \Vert f \Vert_\XX, \sup_n |a_n| w_{2n-1}+\sum_{k=1}^{n-1} |a_k| (w_{2k-1}-w_{2k}) \right\}\\
&\le C C_1 \max\left\{ C_2, C_3w_1+C_3 \sum_{k=1}^{\infty} (w_{2k-1}-w_{2k}) \right\} \Vert f \Vert_\XX.
\end{align*}
Since 
\[
\sum_{k=1}^{\infty} (w_{2k-1}-w_{2k})\le \sum_{k=1}^{\infty} (w_{k}-w_{k+1})=w_1-\lim_n w_n
\]
we are done.

(ii) is immediate from (i).
\end{proof}

Lemma~\ref{lem:KTConditional} puts the focus of attention on  symmetric quasi-Banach spaces $\XX$ such that $\XX\not\subset d_{1,1}(\ww)$ or, as a particular case, in  symmetric quasi-Banach spaces such that $d_{1,1}(\ww) \subsetneq\XX$. Of course, if $\XX$ is locally convex then $d_{1,1}(\ww)\subseteq \XX$, but there are non-locally convex quasi-Banach spaces for which the embedding $d_{1,1}(\ww)\subseteq \XX$ still holds. For instance, certain sequence Lorentz spaces $d_{1,q}(\ww)$ for $1<q\le \infty$ are non-locally convex quasi-Banach spaces. We will characterize when the spaces $d_{1,q}(\ww)$ and $\KT(d_{1,q}(\ww),\ww]$ are locally convex in Proposition~\ref{prop:LorentzBanach}. To that end we need to see an auxiliary result. 

\begin{lemma}\label{lem:KTBanach} Let $\XX\subseteq\FF^\NN$ be a quasi-Banach space whose unit vector system $\BB_e$ is  a greedy basis with fundamental function  equivalent to the primitive weight of a  non-increasing weight $\ww=(w_n)_{n=1}^\infty$. Suppose that $d_{1,q}(\ww)\subseteq \XX$ for some some $1<q\le \infty$.  If the space $\KT[\XX,\ww]$ is  locally convex  then $\sss$ has the URP.
\end{lemma}

\begin{proof}Our assumptions yield a constant $C<\infty$ such that for every $m\in\NN$ and $(f_j)_{j=1}^m$ in $\KT[d_{1,q}(\ww),\ww]$,
\[
\left\Vert \sum_{j=1}^m f_j \right\Vert_{1,\infty,\ww}\le C \sum_{j=1}^m \max\left\{ \Vert f_j \Vert_{1,q,\ww}, \Vert f_j \Vert_\ww\right\}.
\] 
Let $g=(a_n)_{n=1}^\infty$ be a non-decreasing sequence of non-negative numbers.
For $m\in\NN$, let $(g_j)_{k=0}^{m-1}$ consist of all cyclic rearrangements of 
\[
g_0=\sum_{n=1}^{m-1} (-1)^{n-1} a_n\, \ee_n.
\]
For every $j=0$, \dots, $m-1$ we have the estimates
\[
\Vert g_j\Vert_{1,q,\ww}\le \Vert g\Vert_{1,q,\ww},
\]
and, by \eqref{eq:kt2},
\[
\Vert g_j\Vert_\ww\le \max_{1\le k \le m} a_{j+1} w_1 +a_1 w_{m-j+1}\le 2 a_1 \, w_1 \le 2 \Vert g\Vert_{1,q,\ww}.
\]

We also have
\[
\left\Vert    \sum_{k=0}^{m-1}(-1)^k g_k \right\Vert_{1,\infty,\ww} = \left\Vert \sum_{k=1}^m (-1)^{k-1} \ee_k \right\Vert_{1,\infty,\ww} \sum_{n=1}^m a_n= s_m \sum_{n=1}^m a_n.
\]
Hence, 
\[
\Vert A_d(g)\Vert_{1,\infty,\ww} \le 3 C \Vert g \Vert_{1,q,\ww}.
\]
By Theorem~\ref{thm:HardyLorentz}, the sequence $(1/s_n)_{n=1}^\infty$ is a regular weight. Then, by Lemma~\ref{lem:AnsoWeights}, the weight $\sss$ has the URP.
\end{proof}

\begin{proposition}\label{prop:LorentzBanach} Let $\ww=(w_n)_{n=1}^\infty$ be a non-increasing weight with primitive weight $\sss=(s_n)_{n=1}^\infty$, and let $1<q\le\infty$. The following are equivalent.
\begin{itemize}
\item[(i)] $d_{1,q}(\ww)$ is locally convex.
\item[(ii)] $\KT[d_{1,q}(\ww),\ww]$ is locally convex.
\item[(ii)] $\sss$ has the URP.
\end{itemize}
\end{proposition}

\begin{proof} (i) $\Rightarrow$ (ii) is obvious. (ii) $\Rightarrow$ (iii) is a straightforward consequence of Lemma~\ref{lem:KTBanach}. (iii) $\Rightarrow$ (i) follows from combining 
Lemma~\ref{lem:AnsoWeights} with Proposition~\ref{prop:convexityLorentz}~(i).
\end{proof}

\begin{example}\label{ex:KTnotQG}
Since for $1<q\le\infty$ the fundamental function of $\ell_{1,q}$ is $(m)_{m=1}^\infty$, we can safely define $\KT[\ell_{1,q},\uu]$, 
where $\uu=(u_n)_{n=1}^\infty$ is the weight defined by $u_n=1$ for every $n\in\NN$. Since $\KT[\ell_{1,q},\uu]$ inherits its convexity from $\ell_{1,q}$, by Proposition~\ref{prop:convexityLorentz}~(iii) it is $r$-convex for every $r<1$ but, by Proposition~\ref{prop:LorentzBanach}, it is not locally convex. By Lemma~\ref{lem:KTConditional} the unit vector system $\BB_e$ of $\KT[\ell_{1,q},\uu]$ is a conditional basis. By construction we have 
\[ 
\ell_1\subseteq \KT[\ell_{1,q},\uu] \subseteq \ell_{1,\infty}. 
\] 
Thus the unit vector system of  $\KT[\ell_{1,q},\uu]$ is a  conditional democratic basis for which the restricted truncation operator is uniformly bounded.  

Delving deeper into the construction of the space, the authors of \cite{BBGHO2018} proved that $\BB_e$ is not a quasi-greedy basis for $\KT[\ell_{1,q},\uu]$.  This result can also be derived from our next general theorem. We need to introduce some terminology.
\end{example}

Given an increasing sequence $\eta=(m_k)_{k=1}^\infty$ of positive integers we put 
\[
I_k(\eta)=\left\{n \in\NN \colon \sum_{j=1}^{k-1} m_j<n \le \sum_{j=1}^{k} m_j\right\}\] and define the map
\[
T_\eta\colon \FF^\NN \to \FF^\NN, \quad (a_k)_{k=1}^\infty \mapsto \sum_{k=1}^\infty \frac{a_k}{m_k} \Ind_{I_k(\eta)}.
\]

\begin{theorem}\label{prop:KTNotQG}Let $\XX\subseteq\FF^\NN$ a quasi-Banach space for which the unit vector system is a subsymmetric basis with fundamental function of the same order as $(m)_{m=1}^\infty$. Assume that $\XX\not\subset \ell_1$ and that $T_\eta$ is bounded for some increasing sequence $\eta$ of positive integers. Then the unit vector system of $\KT[\XX,\uu]$ is not a quasi-greedy basis.
\end{theorem}

\begin{proof}Without loss of generality we assume that $(\XX,\Vert \cdot\Vert_\XX)$ is $p$-Banach and that $\BB_e$ is a $1$-subsymmetric basis of $\XX$. Set $\Vert P_\eta\Vert=C$. Let $R<\infty$ and pick $f=(a_n)_{n=1}^\infty\in \XX$ with finite support such that $a_n\ge 0$ for every $n\in\NN$ and $\Vert f\Vert_1> (1+C^p)^{1/p} R \Vert f\Vert_{\XX}$. Let $\tau=(m_k)_{k=1}^\infty$ be a subsequence of $\eta$ such that $a_k/m_k\le \min_{n\in\supp(f)} a_n$ for every $k\in\NN$. Put 
\[
g=(\underbrace{ \frac{a_1}{m_1}, \dots, \frac{a_1}{m_1}}_{m_1 \text{ times}},-a_1,\dots, 
\underbrace{\frac{a_k}{m_k},\dots, \frac{a_k}{m_k}}_{m_k \text{ times}}, -a_k, \dots)\]
and 
\[h=(\underbrace{0 \dots, 0}_{m_1 \text{ times}},-a_1, \dots, \underbrace{0 \dots, 0}_{m_k \text{ times}}, -a_k,\dots,).
\]
We have 
\[
\Vert g \Vert_\uu = \max_n a_n \le \Vert f \Vert_\XX
\]
and 
\[
\Vert g\Vert^p_\XX \le \Vert f \Vert_\XX^p+\Vert P_\tau(f) \Vert^p_\XX 
\le \Vert f \Vert_\XX^p+ C^p\Vert f \Vert^p_\XX =(1+C^p) \Vert f \Vert_\XX^p.
\]
Consequently 
\[
\Vert g \Vert_{\KT[\XX, \uu]}\le (1+C^p)^{1/p} \Vert f \Vert_\XX.
\]
We also have 
\[
\Vert h \Vert_{\KT[\XX, \uu]} \ge \Vert h \Vert_{\uu} =\Vert f \Vert_1 > (1+C^p)^{1/p} R.
\]
Thus, 
\[
\Vert h \Vert_{\KT[\XX, \uu]}> R \Vert g \Vert_{\KT[\XX, \uu]}.
\]
Since $h$ is a greedy sum of $f$, we are done.
\end{proof}

We claim that if $\BB_e$ is a subsymmetric basis of $\XX\subseteq\FF^\NN$ and $\XX$ is locally convex (i.e., a Banach space) then $T_\eta$ is a bounded operator from $\XX$ to $\XX$ for every $\eta$. Indeed, this can be readily deduced from the boundedness of the averaging projections (see \cite{LinTza1977}*{Proposition 3.a.4}). However, in light of Lemma~\ref{lem:subsyml1}, Theorem~\ref{prop:KTNotQG} can only be applied to non-locally convex spaces, and this result does not carry over to quasi-Banach spaces. We note that, despite this fact, there are non-locally convex spaces with a symmetric basis such that the operators $T_\eta$ are bounded, as we next show. 

\begin{proposition}\label{prop:Tnubounded}Let $0<q\le \infty$ and $\eta=(m_k)_{k=1}^\infty$ be an incrasing sequence of positive integers.
\begin{itemize}
\item[(i)] $T_\eta$ is bounded from $\ell_{1,q}$ into $\ell_{1,q}$ if and only if $q\ge 1$.
\item[(ii)] $T_\eta$ is unbounded from $g(\uu_p,p)$ into $g(\uu_p,p)$ for every $0<p<1$.
\end{itemize}
\end{proposition}

\begin{proof} (i) If $M_k= \sum_{j=1}^{k} m_j$ we have $M_k\le k m_k$. Let $f=(a_k)_{k=1}^\infty\in c_0$ and let $(b_k)_{k=1}^\infty$ be its non-increasing rearrangement.

 If $q\ge 1$ we have
\begin{align*}
\Vert T_\eta(f)\Vert_{\ell_{1,q}}^q
&\le \sum_{k=1}^\infty m_k \frac{b_k^q}{m_k^q} M_k^{q-1}\\
&=\sum_{k=1}^\infty b_k^q \left(\frac{M_k}{m_k}\right)^{q-1}\\
&\le \sum_{k=1}^\infty b_k^q k^{q-1}\\
&=\Vert f\Vert_{\ell_{1,q}}^q.
\end{align*}
Thus,   $\Vert T_\eta\Vert \le 1$.

Assume now  that $0<q<1$.  The same argument yields 
\[
\Vert T_\eta(f)\Vert_{\ell_{1,q}}^q\ge \sum_{k=1}^\infty b_k^q \left(\frac{M_k}{m_k}\right)^{q-1}.
\]
If $T_{\eta}$ were bounded,  then, for some constant $C<\infty$,
\[
C\sum_{k=1}^\infty b_k^q k^{q-1} 
= C\Vert f\Vert_{\ell_{1,q}}^q
\ge\Vert T_\eta(f)\Vert_{\ell_{1, q}}^q
\ge\sum_{k=1}^\infty b_k^q \left(\frac{M_{k}}{m_k}\right)^{q-1}.
\]
Since 
\[
\lim_k \frac{ \sum_{k=1}^n ({M_k}/{m_k})^{q-1}}{\sum_{k=1}^n k^{q-1} }=\lim_k \left(\frac{M_k}{k\, m_k}\right)^{q-1}=\infty,
\]
we reach a contradiction.

(ii) Assume, by contradiction, that, when regarded as an automorphism of $g(\uu_p,p)$, $T_\eta$ is bounded by some $C<\infty$. 
If $w_n=n^{p-1}$, we have
\begin{equation}\label{eq:potentialsum}
\sum_{n=j}^{m-1} w_n\ge \int_j^m x^{p-1}\, dx = \frac{1}{p} (m^p-j^p), \quad j,m\in\NN,\, j\le m.
\end{equation}
Set $\eta=(m_k)_{k=1}^\infty$. We recursively construct an increasing sequence $\psi\colon\NN\to\NN$ satisfying
\[
-1+m_{\psi(k+1)}\ge s_k:=\sum_{j=1}^k m_{\psi(j)}, \quad k\in\NN.
\]
Then we define $\phi\in\OO$ by $\phi(\NN)=\cup_{k=1}^\infty I_{\psi(k)}(\eta)$.

Given $f=(a_j)_{j=1}^\infty\in\FF^\NN$ we define $f_0=(a_{n,0})_{n=1}^\infty\in\FF^\NN$ by $a_{\psi(n),0}=a_n$ for every $n\in\NN$ and $a_{n,0}=0$ if $n\notin\psi(\NN)$. 
Let $T_\eta(f)=(b_j)_{j=1}^\infty$. Taking into account \eqref{eq:potentialsum} and that the mapping $x\mapsto (1+x)^p-x^p$ is non-increasing we have
\begin{align*}
C\Vert f \Vert_{g(\uu_p,p)}
&=C\Vert f_0 \Vert_{g(\uu_p,p)}\\
&\ge \Vert T_\eta(f_0)\Vert_{g(\uu_p,p)}\\
&\ge\sum_{n=1}^\infty |b_{\phi(n)}|^p w_n\\
&=\sum_{j=1}^\infty \frac{ |a_{j}|^p}{m_{\psi(j)}^p}\sum_{n=1+s_{j-1}}^{s_j} w_n\\
&\ge \frac{1}{p}\sum_{j=1}^\infty \frac{ |a_{j}|^p}{m_{\psi(j)}^p} ((1+s_j)^p-(1+s_{j-1})^p)\\
&= \frac{1}{p}\sum_{j=1}^\infty \frac{ |a_{j}|^p}{m_{\psi(j)}^p} ((1+s_{j-1}+m_{\psi(j)})^p -(1+s_{j-1})^p)\\
&\ge \frac{2^p-1}{p} \sum_{j=1}^\infty |a_j|^p.
\end{align*}
We have obtained $\ell_1\subseteq g(\uu_p,p)$. But, in light of Proposition~\ref{prop:GarlingNotEmbedinl1},  this is an absurdity.
\end{proof}

The examples of non-quasi greedy bases obtained in \cite{BBGHO2018}*{Lemma 8.13} can be alternatively constructed combining Theorem~\ref{prop:KTNotQG} with Proposition~\ref{prop:Tnubounded}. Next we exhibit an example, constructed by the KT-method from a space for which the operator $T_\eta$ is unbounded, of a non-quasi greedy basis that can be squeezed between $\ell_1$ and $\ell_{1,\infty}$.

\begin{proposition}Let $0<p<1$. Then the unit vector system of $\KT[g(\uu_p,p),\uu]$ is not quasi-greedy.
\end{proposition}
\begin{proof}We proceed as in the proof of Proposition~\ref{prop:GarlingNotEmbedinl1}. Given $N\in\NN$, there is an $N$-tuple $(m_j)_{j=1}^N$ of natural numbers such that, in the terminology of that proof,
\[
\Vert \hhh_{m_N} \smallfrown \dots \smallfrown \hhh_{m_j} \smallfrown\dots \smallfrown \hhh_{m_1}\Vert_{g(\uu_p,p)}\le C.
\]
Then, there is another $N$-tuple $(r_j)_{j=1}^N$ of natural numbers such that $\min_{1\le j \le N} r_j >\max_{1\le j \le N} m_j$ and 
\[
\Vert \hhh_{r_N} \smallfrown \dots \smallfrown \hhh_{r_j} \smallfrown\dots \smallfrown \hhh_{r_1}\Vert_{g(\uu_p,p)}\le C.
\]
Denote by $0_m$ the null vector of $\FF^m$ and define
\begin{align*}
f&= \hhh_{m_N} \smallfrown -\hhh_{r_N} \smallfrown \dots \smallfrown \hhh_{m_j} \smallfrown -\hhh_{r_j} \smallfrown\dots \smallfrown \hhh_{m_1} \smallfrown -\hhh_{r_1},\\
g&= \hhh_{m_N} \smallfrown 0_{r_N} \smallfrown \dots \smallfrown \hhh_{m_j} \smallfrown 0_{r_j} \smallfrown\dots \smallfrown \hhh_{m_1} \smallfrown 0_{r_1}.
\end{align*}
We have $\Vert f \Vert_{g(\uu_p,p)}\le 2^{1/p} C$, $\Vert f \Vert_\uu=1$ and $\Vert g \Vert_\uu=N$. Consequently
$\Vert f \Vert_{\KT[g(\uu_p,p),\uu]}\le \max\{ 2^{1/p} C,1\}$ and $\Vert f \Vert_{\KT[g(\uu_p,p),\uu]}\ge N$. Since $g$ is a greedy sum of $f$ we are done.
\end{proof} 

As the attentive reader may have noticed, all the examples of non-quasi greedy bases we have built in this section are bases of non-locally convex quasi-Banach spaces.
For an example of a Banach space with a non-quasi-greedy basis which is both a SLC and   LUCC we refer to \cite{BBG2017}. It is clear that the non-quasi-greedy basis $\BB$ for the Banach space $\XX$ constructed in the proof of \cite{BBG2017}*{Proposition 5.6} can be squeezed between two spaces as follows
\[
\ell_{1} \stackrel{\BB}\hookrightarrow \XX \stackrel{\BB}\hookrightarrow \ell_{1,\infty}.
\]

\subsection{Conditional almost greedy bases.}
The topic of finding conditional quasi-greedy bases was initiated by Konyagin and Telmyakov \cite{KoTe1999} and developed subsequently in several papers \cites{AAW2019, AADK2019, DKK2003, DKW, GHO2013, Gogyan2010, Wo2000}. It turns out that a wide class of Banach spaces posses a conditional almost greedy basis (see \cite{AADK2019}). So, in light of Proposition~\ref{prop:sumdemocracy}, giving examples of non-locally convex quasi-Banach spaces with an almost greedy basis is an easy task. For instance, we have the following result. 
\begin{proposition}\label{prop:existence} Let $\XX$ be a quasi-Banach space $\XX$ with a basis $\BB_0$. Let $\YY$ be either $\ell_1$ or a Banach space equipped with a subsymmetric basis whose fundamental function $\sss$ has both the LRP and the URP.
\begin{itemize}
\item[(i)] If $\BB_0$ is unconditional, then $\XX\oplus \YY$ has a conditional quasi-greedy basis.
\item[(ii)] If $\BB_0$ is greedy and  $\varphi_u^\EE[\BB_0,\XX]\approx \sss$ 
 then $\XX\oplus \YY$ has a conditional almost greedy basis whose  fundamental function is equivalent to $\sss$.
\end{itemize}
\end{proposition}

Before we see  the proof of Proposition~\ref{prop:existence} let us fix some notation and give an auxiliary lemma. Recall that a function $\delta\colon(0,\infty)\to (0,\infty)$ is said to be \emph{doubling} if there is a positive constant $C<\infty$ such that $\delta(2x)\le C\delta(x)$ for all $x\in[0,\infty)$. 
\begin{lemma}\label{lem:doubling}Let $(L_m)_{m=1}^\infty$ be an unbounded non-decreasing sequence of positive scalars. Then there is a non-decreasing unbounded doubling function $\delta\colon[0,\infty)\to(0,\infty)$ such that $\delta(m)\le L_m$ for every $m\in\NN$.
\end{lemma}
\begin{proof} Let $(d_n)_{n=0}^\infty$ be the sequence defined recursively by the formula
\[
d_0= L_1, \quad d_n=\min\{ 2 d_{n-1}, L_{2^n} \}, \quad n\in\NN. 
\]
Since  $(L_m)_{m=1}^\infty$ is non-decreasing, so is $(d_n)_{n=0}^\infty$. It is clear from the definition that $d_n\le 2 d_{n-1}$ for every $n\in\NN$. If $d_n\le C$ for some $C<\infty$ and every $n\ge 0$ there is $m_0$ such that $L_m>C$ for every $m\ge m_0$. Then, if $n\ge \log_2(m_0)$, $d_n = 2 d_{n-1}$ which implies that $(d_n)_{n=0}^\infty$ is  unbounded. This contradiction proves that $(d_n)_{n=0}^\infty$ is unbounded. 

Since $d_n \le L_{2^n}$ for every $n\in\NN$, the function
\[
\delta(x)=d_0\chi_{[0,2)}+\sum_{n=1}^\infty d_n\chi_{[2^n,2^{n+1})}
\]
satisfies the desired properties.
\end{proof}

Given a basis $\BB$ of a quasi-Banach space $\XX$ we consider the sequence defined for $m\in \NN$ by
\begin{equation*}
L_m[\BB,\XX]=\sup\left\{ \frac{ \Vert S_A[\BB,\XX](f)\Vert}{ \Vert f \Vert } \colon \max(\supp(f))\le m, \, A\subseteq\NN\right\}.
\end{equation*}
Note that the basis $\BB$ is unconditional if and only if $(L_m[\BB,\XX])_{m=1}^\infty$ is bounded.
\begin{proof}[Proof of Proposition~\ref{prop:existence}] We show (ii) and leave (i) for the reader, as it is similar and easier. By a classical theorem of Pe\l czy\'{n}ski and Singer \cite{PeuSinger1965}, we can guarantee the existence of a conditional basis, say $\BB_2$, in $\YY$. Then, by Lemma~\ref{lem:doubling} there is an unbounded non-decreasing doubling function $\delta$ such that $L_m[\BB_2,\YY]\gtrsim \delta(m)$ for $m\in\NN$. Therefore, by \cite{AADK2019}*{Remark 4.2}, $\YY$ has an almost greedy basis $\BB_3$ with fundamental function equivalent to $\sss$, such that $L_m[\BB_3,\YY]\gtrsim \delta(\log(m))$ for $m\in\NN$. Hence by Proposition~\ref{prop:sumdemocracy}, the basis $\BB_0\oplus\BB_3$ of $\XX\oplus\YY$ satisfies the desired properties.
\end{proof}

Proposition~\ref{prop:existence}~(ii) can be applied to show that the separable part of a Triebel-Likorkin space of indices $1\le p<\infty$ and $0<q\le \infty$ has a conditional almost greedy basis. In turn, Proposition~\ref{prop:existence}~(i) can be used to prove that the separable part of a Triebel-Lizorkin or Besov space of indices   $0< p\le \infty$ and  $1\le q<\infty$ has a conditional quasi-greedy basis. On the negative side, one finds interesting and important quasi-Banach spaces such as for instance $\ell_p$ for $p<1$, which are out of the scope of Proposition~\ref{prop:existence} since they do not have a locally convex complemented subspace. In fact, as Proposition~\ref{prop:NQGBesov} shows, the existence of a nonlocally convex complemented subspace with a symmetric basis is  not a guarantee  of the existence of an almost greedy basis.

The last part of this section will be dedicated to generalizing an example of a conditional almost greedy basis from \cite{BBGHO2018}. 

\begin{theorem}\label{thm:KTGreedy}
Let $\XX\subseteq\FF^\NN$ be a quasi-Banach space. Suppose that the unit vector system $\BB_e$ of $\XX$ is a greedy basis with fundamental function equivalent to the primitive weight $\sss=(s_n)_{n=1}^\infty$ of a non-increasing weight $\ww$, and that $\sss$ has  both the LRP and the URP. Then the unit vector system of $\KT[\XX,\ww]$ is a quasi-greedy basis.
\end{theorem}

Before proving Theorem~\ref{thm:KTGreedy}, we establish an auxiliary lemma.

\begin{lemma}\label{lem:urpestimate} Let $\sss=(s_n)_{n=1}^\infty$ be a non-decreasing  weight such that $(s_n/n)_{n=1}^\infty$ is  essentially decreasing. Assume that $\sss$ has  both the LRP and the URP. Then there is $r>1$ such that
\[
C[\sss,r]:=\sup_n s_n^{r-1} \sum_{j=1}^{\infty} \frac{1}{s_j^{r}} \frac{s_{n+j-1}}{n+j-1} <\infty.
\]
\end{lemma}
\begin{proof}By Lemma~\ref{lem:AnsoWeights} there is $1<r<\infty$ such that $(s_n^{r}/n)_{n=1}^\infty$ is essentially decreasing. Then, also by Lemma~\ref{lem:AnsoWeights}, the sequence $(s_n^{-r})_{n=1}^\infty$, is a regular weight. Set
\[
D=\sup_{n}   \frac{s_n^r}{n} \sum_{j=1}^n \frac{1}{s_j^{r}}<\infty.
\] 
Using Lemma~\ref{lem:AnsoWeights} for a third time gives $0<s<1$ such that  $(n^{-s} s_n)_{n=1}^\infty$ is essentially increasing. Set
\[
C=\sup_{n\ge j} \frac{s_n}{s_j} \frac{j}{n} <\infty \text{ and }
E=\sup_{n\le j} \frac{s_n}{s_j} \frac{j^s}{n^s} <\infty.
\]
On the one hand we have, 
\[
\sup_n s_n^{r-1} \sum_{j=1}^{n} \frac{1}{s_j^{r}} \frac{s_{n+j-1}}{n+j-1} 
\le C \sup_n s_n^{r-1} \frac{s_{n}}{n} \sum_{j=1}^n \frac{1}{s_j^{r}} 
\le  C D,
\]
and on the other hand, 
\begin{align*}
\sup_n s_n^{r-1} \sum_{j=n+1}^{\infty} \frac{1}{s_j^{r}} \frac{s_{n+j-1}}{n+j-1}
&=\sup_n  \sum_{j=n+1}^{\infty} \frac{s_n^{r-1}}{s_j^{r-1}} \frac{s_{n+j-1}}{s_j} \frac{1}{n+j-1} \\ 
&\le C E^{r-1} \sup_n  \sum_{j=n+1}^{\infty} \frac{n^{s(r-1)}}{j^{s(r-1)}} \frac{{n+j-1}}{j} \frac{1}{n+j-1} \\
&=C E^{r-1} \sup_n  \sum_{j=n+1}^{\infty} \frac{n^{s(r-1)}}{j^{1+s(r-1)}} \\
&\le \frac{C E^{r-1}}{s(r-1)}.\qedhere
\end{align*}
\end{proof}

\begin{proof}[Proof of Theorem~\ref{thm:KTGreedy}] Since $\XX$ is contained in $d_{1,\infty}(\ww)$ and the inclusion map is continuous, it suffices to prove that there is a constant $C<\infty$ such that 
\begin{equation}\label{eq:goalineq}
\Vert S_A(f)\Vert_\ww\le C \max\{ \Vert f \Vert_{1,\infty,\ww},\Vert f \Vert_\ww\}
\end{equation}
for all $f=(a_n)_{n=1}^\infty\in\FF^\NN$ and all greedy sets $A$ of $f$. 

For $m\in\NN$ let $A_m=\{n\in A \colon n\le m\}$. If $\alpha=\min\{ |a_n| \colon n\in A\}$ let $M$ be the largest 
integer such that $\alpha S_M\le \Vert f\Vert_{1,\infty,\ww}$ (by convention we take $M=0$ if such an integer does not exist). We have
\begin{align*}
\left|\sum_{\substack{n\in A_m\\ n\le M} } a_n w_n\right| 
&\le \left|\sum_{n=1}^{\min\{m,M\}} a_n w_n\right|+\left|\sum_{\substack{{ n\notin A}\\ n\le \min\{m,M\}} } a_n w_n\right|\\
&\le \Vert f \Vert_\ww + \alpha s_{\min\{m,M\}}\\
&\le \Vert f \Vert_\ww+ \Vert f\Vert_{1,\infty,\ww}.
\end{align*}
Use Lemma~\ref{lem:urpestimate} to pick $r>1$ such that $C[\sss,r]<\infty$. Choose a bijection $\beta\colon\{1,\dots, N\}\to \{ n\in A \colon M < n \le m\}$ ($N=0$ if the involved set is empty). Applying the rearrangement inequality gives
\begin{align*}
\left|\sum_{\substack{n\in A_m\\ n> M} } a_n w_n\right|
&\le \sum_{j=1}^N |a_{\beta(j)}| w_{\beta(j)} \\
&\le \alpha^{1-r}\sum_{j=1}^N |a_{\beta(j)}|^r w_{\beta(j)}\\
&\le \alpha^{1-r} \sum_{n=1}^{\infty} (a_j^*)^r w_{j+M} \\
&\le \alpha^{1-r}\Vert f\Vert_{1,\infty,\ww}^r \sum_{j=1}^{\infty} \frac{1}{s_j^r} \frac{s_{j+M}}{j+M} \\
&\le C[\sss,r] \alpha^{1-r}S_{1+M}^{1-r}\Vert f\Vert_{1,\infty,\ww}^r \\
&\le C[\sss,r] \Vert f\Vert_{1,\infty,\ww}.
\end{align*}
Summing up, \eqref{eq:goalineq} holds with $C=2(1+ C[\sss,r])$.
\end{proof}

\begin{remark} By Lemma~\ref{lem:KTConditional} in the case when $d_{1,1}(\ww)\subsetneq \XX$, the weight $\ww$ is doubling, and the unit vector system of $\XX$ is a symmetric basis,  the quasi-greedy bases originating from Theorem~\ref{thm:KTGreedy} are conditional. 
\end{remark}

\begin{remark} The most natural application of Theorem~\ref{thm:KTGreedy} is obtained by putting $\XX=d_{1,q}(\ww)$ with $q>1$ and $\ww$ regular. In this case, by Proposition~\ref{prop:LorentzBanach}, the quasi-Banach space $\KT[d_{1,q}(\ww),\ww]$ is locally convex. In particular, if for $1<p<\infty$ we consider $\ww=\uu_{1/p}$ (with the terminology of \eqref{eq:potential}) we obtain that the unit vector system of the locally convex space $\KT[\ell_{p,q},\uu_{1/p}]= \KT[d_{1,q}(\uu_{1/p}),\uu_{1/p}]$ is quasi-greedy. This result was previously proved in \cite{BBGHO2018}, thus Theorem~\ref{thm:KTGreedy} provides an extension.
\end{remark}

\section{Renorming quasi-Banach spaces with greedy-like bases}\label{sec:renorming}
\noindent
The topic of renorming Banach spaces with greedy (or almost greedy, or quasi-greedy) bases has its origins in \cite{AW2006}, where the authors characterized $1$-greedy bases and posed the problem, still unsolved as of today, of finding a renorming of $L_p$, $1<p<\infty$, with respect to which the Haar system is $1$-greedy. Subsequently, $1$-almost greedy  bases and $1$-quasi-greedy bases were also characterized in \cites{AA2017,AA2016}. The first examples of non-symmetric $1$-greedy bases of an infinite-dimensional Banach space appeared if \cite{DOSZ2011} (see \cite{AAW2018} for another relevant contribution to this subject). In those papers convexity is both a requirement and a key tool. Note that every renorming $\Vert \cdot\Vert_0$ of a Banach space $(\XX,\Vert \cdot \Vert)$ has the form
\begin{equation}\label{eq:newnorm}
\Vert f \Vert_0=\max\{ a \Vert f \Vert, \Vert T (f)\Vert_\YY\}, \quad f\in\XX,
\end{equation}
for some $0<a<\infty$ and some bounded linear operator from $\XX$ into a Banach space $(\YY,\Vert \cdot\Vert_\YY)$. Indeed, it is clear that \eqref{eq:newnorm} gives a new norm and, conversely, given a new norm $\| \cdot \|_0$, if we put $\YY=(\XX, \|\cdot\|_0)$, choose $T$ to be the identity operator and pick $a>0$ small enough, \eqref{eq:newnorm} holds. The situation in quite different when dealing with quasi-norms as the following easy result evinces.
\begin{lemma}\label{lem:renorming1} Let $(\XX,\Vert \cdot \Vert)$ be a quasi-Banach space. Assume that $\Vert \cdot\Vert_0\colon \XX\to[0,\infty)$ is such that
\begin{itemize}
\item[(i)] $\Vert t f \Vert_0=|t| \Vert f \Vert_0$ for every $t\in\FF$ and $f\in\XX$, and
\item[(ii)] $\Vert f \Vert_0 \approx \Vert f \Vert$ for $f\in\XX$.
\end{itemize}
Then $\Vert \cdot\Vert_0$ is a renorming of $\Vert \cdot \Vert$.
\end{lemma}
Lemma~\ref{lem:renorming1} allows us to build renormings of quasi-Banach spaces based on non-linear operators. So it is not surprising that we are able to find renormings of quasi-Banach spaces for which the properties associated to the greedy algorithm hold isometrically. For instance, the following result is essentially based on the fact that the families of maps $(\GG_{m})_{m=0}^\infty$, $(\HH_{m})_{m=0}^\infty$ and $(\TT_{m})_{m=0}^\infty$ are semigroups of (non-linear) operators on $\XX$.

\begin{theorem}\label{thm:renorming2}Let $\BB$ be a quasi-greedy basis of a quasi-Banach space $\XX$. Then there is a renorming of $\XX$ with respect to which $C_{qg}=\Lambda_t=\Lambda_u=1$.
\end{theorem}

\begin{proof}By Lemma~\ref{lem:qg9} we need only estimate $C_{qg}$ and $\Lambda_t$. For $f\in\XX$ put
\[
\RRR_0(f)=\{(A_1,A_2) \colon A_i \text{ greedy set of $f$, } A_1\subseteq A_2\subseteq\supp(f)\},
\]
and 
\[
\Vert f\Vert_0=\sup \{ \Vert S_{A_2\setminus A_1}(f) \Vert \colon (A_1,A_2)\in\RRR_0(f) \}.
\]
Since $\BB$ is quasi-greedy, $\Vert \cdot\Vert_0$ is a renorming of $\Vert \cdot\Vert$. 

If $(A_1,A_2)\in\RRR_0(f)$ and $(B_1,B_2)\in \RRR_0( S_{A_1\setminus A_2}(f))$ then we have that $(A_1\cup B_1, A_1\cup B_2)\in\RRR_0(f)$, $S_{B_2\setminus B_1}( S_{A_2\setminus A_1}(f))= S_{B_2\setminus B_1}(f)$, 
and $ A_1\cup B_2\setminus (A_1\cup B_1)=B_2\setminus B_1$. We infer that for every $f\in\XX$ and $(A_1,A_2)\in\RRR_0(f)$,
\begin{equation} \label{PW11.2.a}
\Vert S_{A_2\setminus A_1}(f)\Vert_0 \le \Vert f\Vert_0.
\end{equation}

 For $f\in\XX$ put
\[
\Vert f\Vert_1=\sup \{ \Vert \TT(f,A) \Vert_0 \colon A \text{ strictly greedy set of } f \}.
\]
By Proposition~\ref{prop:qg6}, $\Vert \cdot \Vert_1$ is a renorming of $\Vert \cdot \Vert$. Let us check that it is the one we are after.

If $A$ a strictly greedy set of $f$ and $B$ is a non-empty strictly greedy set of $\TT(f,A)$ then $A\subseteq B$ and $B$ is a strictly greedy set of $f$. Moreover, $\TT(\TT(f,A),B)=\TT(f,B)$. It follows that $\Vert \TT(f,A)\Vert_1\le \Vert f\Vert_1$ for every $f\in\XX$ and every $A$ strictly greedy set of $f$. Now observe that for every greedy set $A$ of $f\in\XX$ there is a strictly greedy set $A_0$ with $A\subseteq A_0$ and $\TT(A,f)=\TT(A_0,f)$, so we get $\Lambda_t\leq 1$ with respect to $\Vert\cdot\Vert_1$.

Let $(A_1,A_2)\in\RRR_0(f)$ and $B$ be a non-empty strictly greedy set of $S_{A_2\setminus A_1}(f)$. Then $B\cup A_1$ is a strictly greedy set of $f$, and $B\cup A_1\subseteq A_2$. Consequently $(A_1,A_2)\in\RRR_0( \TT(f,A_1\cup B))$ and 
\begin{equation}\label{PW11.2.b}
\TT(S_{A_2\setminus A_1}(f), B)=S_{A_2\setminus A_1} (\TT(f,A_1\cup B)).
\end{equation}
From (\ref{PW11.2.b}) we get
\[
\|S_{A_2\setminus A_1}f\|_1=\sup_B\|S_{A_2\setminus A_1}(\TT(f,A_1\cup B))\|_0,
\]
where  $B$ runs over all  strictly  greedy sets of $ S_{A_2\setminus A_1}(f)$. Since the pair $(A_1,A_2)$ belongs to $\RRR_0(\TT(f,A_1\cup B))$, from (\ref{PW11.2.a}) we get 
\[
\|
S_{A_2\setminus A_1}f\|_1\leq \sup_B \|\TT(f,A_1\cup B)\|_0
\]
 where $B$ runs again over all  strictly  greedy sets of $ S_{A_2\setminus A_1}(f)$.
Now we observe that if $B$ is a strictly greeedy set of $S_{A_2\setminus A_1}(f)$ then $A_1\cup B$ is a strictly greedy set of $f$. Thus we get  $\|S_{A_2\setminus A_1}f\|_1\leq \|f\|_1$.
\end{proof}

Inspired by Corollary~\ref{cor:ag4}, we introduce the following function on a quasi-Banach space $\XX$ associated to a basis $\BB$ of $\XX$. For $f\in\XX$ we put
\begin{equation}\label{eq:newquasinorm}
\Vert f\Vert_a=\inf\{ \Vert f-S_A(f)+z \Vert \colon (A,z)\in\DDD(f)\},
\end{equation}
where $\DDD(f)$ is the set consisting of all pairs $(A,z)\in\PP(\NN)\times \XX$ such that $|A|<\infty$, $(\supp(f)\setminus A)\cap \supp (z)=\emptyset$, 
$A\subseteq\supp(f),\, |A|\le|\supp(z)|$, and $\max_{n\in\NN} |\xx_n^*(f)|\le \min_{n\in\supp (z)} |\xx_n^*(z)|$.

\begin{theorem}\label{thm:renorming4} Let $\BB$ be an almost greedy basis of a quasi-Banach space $(\XX,\|.\|)$. The function $\Vert \cdot \Vert_a$ defined in \eqref{eq:newquasinorm} gives a renorming of $\XX$ with respect to which $C_{ag}=\Lambda_t=1$. Moreover, if we additionally assume that $C_{qg}=1$ for $\|\cdot \|$, then $C_{qg}=\Lambda_u=1$ for $\Vert \cdot \Vert_a$.
\end{theorem}

\begin{proof}By Corollary \ref{cor:ag4}, $\|f\|_a\approx \|f\|$ for $f\in \XX$, so by Lemma \ref{lem:renorming1}, $\|\cdot \|_a$ is a renorming of $\|\cdot\|$.  Let  $(A,z)\in\DDD(f)$ and set $g=f-S_A(f)+z$. Let $(B,y)\in\DDD(g)$ and denote
$B_1=B\cap(\supp(f)\setminus A)$ and $B_2=B\cap\supp(z)$, so that 
\[
g-S_B(g)=f -S_{A\cup B_1}(f)+z-S_{B_2}(z).
\]
We have  $\supp(z-S_{B_2}(z))\cap \supp(y)=\emptyset$ and
\begin{align*}
|A\cup B_1|&=|A|+|B_1|\\
&=|A|+|B|-|B_2| \\
&\le |\supp (z)|+|\supp (y)|-|B_2|\\
&=|\supp(z-S_{B_2}(z))| +|\supp(y)| \\
&= |\supp(z-S_{B_2}(z)+y)|.
\end{align*}
We infer that $(A\cup B_1, z-S_{B_2}(z) +y)\in \DDD(f)$. Therefore,
\[
\Vert f\Vert_a \le \Vert f -S_{A\cup B_1}(f)+z-S_{B_2}(z)+y\Vert
=\Vert g-S_B(g)+y \Vert.
\]
Taking the infimum over $(B,y)$ we get $\Vert f\Vert_a\le \Vert g\Vert_a$. 

Clearly, $S_A(f)=S_{A\cap \supp(f)}(f)$ so the inequality 
\begin{equation}\label{PW11.2.c}
\Vert f\Vert_a \le \Vert f-S_A(f)+z\Vert_a
\end{equation}
holds for  every  $A \subseteq \NN$  such that $(A\cap \supp f,z)\in \DDD(f)$. By Lemma~\ref{lem:ag2} and Lemma~\ref{lem:ag5}, $C_{ag}=\Lambda_t=1$ with respect to the equivalent quasi-norm $\Vert \cdot\Vert_a$. 

From our additional assumption that $C_{qg}=1$ for $\|\cdot \|$ it follows that $\Vert S_B(f)\Vert \le \Vert f\Vert$ for every $f\in\XX$ and every greedy set $B$ of $f$. Let $(A,z)\in\DDD(f)$. Then $(A\cap B,z)\in \DDD(S_{B}(f))$. Moreover, $D=\supp(z)\cup(B\cap A^c)$ is a greedy set of $g=f-S_A(f)+z$ and so
\[
\Vert S_{B}(f)\Vert_a 
\le \Vert S_{A^c\cap B}(f)+z\Vert 
= \Vert S_{D}(g)\Vert \le \Vert g \Vert=\Vert f-S_A(f)+z\Vert.
\]
Taking the infimum over $(A,z)$ we get 
\begin{equation}\label{PW11.3.d}
\|S_B(f)\|_a\leq \|f\|_a.
\end{equation}
Substituting $f=h-S_B(h)$, $A=\emptyset$ and $z=S_B(h)$ in (\ref{PW11.2.c}) we get 
\begin{equation}\label{PW11.2.d} 
\Vert h -S_B(h)\Vert_a \le\Vert h \Vert_a
\end{equation}
for every $h\in\XX$ and every greedy set $B$ of $h$. From (\ref{PW11.3.d}) and (\ref{PW11.2.d}) it follows that $C_{qg}=1$ with respect to $\|\cdot\|_a$. Now thanks to Lemma~\ref{lem:qg9} we get $\Lambda_u=1$ for the quasi-norm $\|\cdot \|_a$.
\end{proof}

\begin{corollary}\label{cor:renorming4} Let $\BB$ be an almost greedy basis of a quasi-Banach space $\XX$. Then there is a renorming of $\XX$ with respect to which $C_{qg}=C_{ag}=\Lambda_u=\Lambda_t=1$.
\end{corollary}
\begin{proof}Since an almost greedy basis is quasi-greedy, Theorem~\ref{thm:renorming2} yields that  $\XX$ has a renorming with respect to which $C_{qg}[\BB,\XX]=1$. Now Theorem~\ref{thm:renorming4} finishes the proof.
\end{proof}

\begin{theorem}\label{PW:th_greedy_renorming}Let $\BB$ be a greedy basis of a quasi-Banach space $\XX$. There is a renorming of $\XX$ with respect to which $C_{g}=K_u=1$.
\end{theorem}

\begin{proof} Without loss of generality we may assume that $\XX$ is equipped with a quasi-norm $\Vert \cdot\Vert $ with respect to which $K_{u}=1$. Let $\Vert \cdot \Vert_a$ be as is \eqref{eq:newquasinorm}. Since every greedy basis is almost greedy, by Theorem~\ref{thm:renorming4}, $\Vert \cdot \Vert_a$ is a renorming of $\Vert \cdot\Vert $ with respect to which $C_{ag}=1$. By Theorem~\ref{thm:chg} it suffices to prove that $K_{u}=1$ with respect to $\Vert \cdot \Vert_a$. Let $f\in\XX$, $\gamma=(\gamma_n)_{n=1}^\infty\in\FF^\NN$ with $\Vert \gamma\Vert_\infty\le 1$, and $(A,z)\in\DDD(f)$. Put $B=\{n\in A \colon \gamma_n\not=0\}$ and consider the sequence $\mu=(\mu_n)_{n=1}^\infty$ defined by
\[
\mu_n=\begin{cases} \gamma_n & \text{ if } n\notin\supp(z), \\ 1 & \text{ if } n\in\supp(z). \end{cases}
\]
Since $(B,z)\in\DDD(M_\gamma(f))$ we have
\[
\Vert M_\gamma(f)\Vert_a
\le \Vert S_{B^c}( M_\gamma(f))+z\Vert 
=\Vert M_{\mu} (S_{A^c}(f)+z)\Vert
\le \Vert S_{A^c}(f)+z\Vert.
\]
Minimizing over $(A,z)\in\DDD(f)$ finishes the proof.
\end{proof}

\section{Open Problems}\label{Sec:Problems}

It is clear that our work leaves many questions unanswered. This is a sign that the subject of greedy approximation using bases is still very much alive as intriguing  new problems arise from the main theory. Below we include a non-exhaustive list of questions that spring very naturally from our exposition and that we think are the natural road to take to make headway from here.
\begin{problem}
Unconditionality for constant coefficients, or SUCC for short, is a very natural assumption about the basis.
It seems to be unknown whether  LUCC implies SUCC for every basis.
\end{problem}

\begin{problem} The renorming results of Section \ref{sec:renorming} leave a lot of open questions. Basically we would like to produce better renormings. As an example let us start with the following still open problem from \cite{AW2006}: 
\begin{quote}
Does there exists a renorming of the Banach space $L_p$, $1<p<\infty$, with respect to which the Haar system is $1$-greedy? 
\end{quote}
Theorem \ref{PW:th_greedy_renorming} says yes, but it gives a (most likely discontinuous) {\em quasi-norm} when it is clear that the original problem asks for {\em norm}. Thus the following weaker questions are also open:
\begin{itemize}
\item[(i)] Does there exists an equivalent, continuous quasi-norm on $L_p$, $1<p<\infty$, with respect to which the Haar system is $1$-greedy?
\item[(ii)]  Does there exists an equivalent $s$-norm (for some $0<s<1$) on $L_p$, $1<p<\infty$, with respect to which the Haar system is $1$-greedy?
\end{itemize}
We would also like to pose the following question:
\begin{itemize}
\item[(iii)]Characterize (or describe an interesting) class of quasi-Banach spaces which have a $1$-greedy (or quasi-greedy) basis with respect to a $p$-norm (or continuous quasi-norm).
\end{itemize}

\end{problem}
\begin{problem}[Weakly quasi-greedy basis] Let us introduce the following definition: Let $\XX$ be a quasi-Banach space and let $\tau$ be a topology  on $\XX$ weaker than the norm topology.  A basis $\BB=(\xx_n)_{n=1}^\infty$ in $\XX$ is said to be \emph{$\tau$-quasi-greedy} if for every $f\in \XX$ the greedy series \eqref{PWgreedyseries} converges  to $f$ in the $\tau$-topology. Important examples include the following.
\begin{itemize}
\item $\XX$ is a space of functions on a space $K$ and $\tau$ is a pointwise convergence or convergence almost everywhere or convergence in measure. Many classical cases were considered already (see  \cite{korner}  for a general introduction).
\item Let $\tau$  be a weak topology on $\XX$.
Let us make some remarks.
\begin{enumerate}
\item[(a)] Since the basis is total if such a series converges, it converges to $f$.
\item[(b)] If $\XX$ is a Banach space then weak concergence implies that partial sums are bounded so Theorem \ref{PW12.thm1} implies that the basis is quasi-greedy. 
\item[(c)] For a general quasi-greedy Banach space,  weak convergence does not imply boundedness; it implies boundedness in the Banach envelope.
\end{enumerate}
\end{itemize}
\end{problem}

\begin{problem}The definition of bidemocracy was tailored for Banach spaces. Every (sub)symmetric basis in a Banach space is bidemocratic while the most natural bases like the unit vector bases in $\ell_p$ for $0<p<1$ are not. It is natural to think that the existence of a bidemocratic basis in a space $\XX$ implies some convexity that brings
 $\XX$ to being close to a Banach space. The following questions are of interest in this context:
\begin{enumerate}
\item[(i)] Does there exists a bidemocratic basis in the space $\ell_p\oplus\ell_2$ with $0<p<1$?
\item[(ii] Suppose $\XX$ is infinite-dimensional and has a a bidemocratic basis. Given $0<p<1$, does there exists an infinite-dimensional subspace $\XX_p\subset \XX$ which has an equivalent $p$-norm?
\end{enumerate}
\end{problem}

\begin{problem} Our discussion in Section~\ref{Sec13} on Banach envelopes  leaves the following open questions: 
\begin{enumerate}
\item[(i)] If $\BB$  is a quasi-greedy basis in a quasi-Banach space $\XX$,  is its Banach envelope $\Env{\BB}$  quasi-greedy in the Banach envelope $\Env{\XX}$ of $\XX$? More generally, is $\Env[r]{\BB}$
 quasi-greedy in the $r$-Banach envelope $\Env[r]{\XX}$ for every (some) $0<r\le 1$?
\end{enumerate}
The reader may argue that this question is too hard to tackle. Indeed, perhaps it would be more sensible to get started by focusing on special types of bases. In particular, the question remains open for bidemocratic quasi-greedy bases:

\begin{enumerate}
\item[(ii)] Let $0<r<1$.  If a basis is bidemocratic and quasi-greedy in $\XX$, is it quasi-greedy in the $r$-Banach envelope $\Env[r]{\XX}$?
\end{enumerate}
\end{problem}

\begin{problem}The greedy algorithm with respect to bases is essentially independent of the particular ordering we chose for the basis. That is, an $M$-bounded semi-normalized basis  $\BB=(\xx_n)_{n=1}^\infty$ enjoys the same greedy-like properties as any of its reorderings  $(\xx_{\phi(n)})_{n=1}^\infty$.  Thus, when dealing with greedy-like properties, imposing conditions on the basis such as being  Schauder, which depend on a particular ordering,   is somewhat unnatural.   In fact, investigating the greedy algorithm without  assuming that  our basis is Schauder, not only does enable us to obtain more general results, but also, and  above all, gives us the opportunity to differentiate  those results  that can  be obtained only for Schauder bases from those other results that can be proved circumventing this condition. However, it must be conceded that the most important examples of bases  respect to which the greedy algorithm is considered are Schauder bases,  to the extent that the following question seems to be unsolved.

\begin{itemize}
\item Is there a quasi-greedy basis which is not a Schauder basis for any ordering?
\end{itemize}
\end{problem}

\begin{problem}Now we focus on locally convex quasi-Banach spaces. Every Schauder basis $\BB$  for a Banach space is equivalent to its bidual basis $\BB^{**}$ (see Theorem~\ref{thm:Schaudereflexivity}), but this result does not  carry over to total $M$-bounded bases (see Proposition~\ref{prop:NonNorming}). This forced us to conduct our study of duality properties of greedy-like bases without this powerful and widely used tool. In this context, it is natural to wonder whether or not a  given total $M$-bounded basis is equivalent to its bidual basis. As for quasi-greedy bases, which are total  by Corollary~\ref{cor:QGtotal},  Theorem~\ref{thm:BDtoQG+} provides a partial answer to this reflexivity question. To the best of our knowledge, the general problem remains open:

\begin{itemize}
\item Let $\BB$ be a quasi-greedy basis for a Banach space. Is  the mapping  $h_{\BB,\XX}$ defined in \eqref{eq:bibiorthogonal} an isomorphic embedding?
Is $\BB^{**}$ equivalent to $\BB$?
\end{itemize}
\end{problem}

\begin{problem}\label{prob:8} It is well known that  $\ell_{1}$, $\ell_{2}$, and $c_{0}$ are the only Banach spaces with a unique unconditional basis (\cites{LP1968, LZ1969, KT1934}). It happens that in all three spaces, quasi-greedy bases are democratic (thus almost greedy) \cites{DST2012, Wo2000, DKK2003}. The spaces $\ell_{p}$ for $p<1$ also have a unique unconditional basis \cite{Kalton1977}, and   their Banach envelope is $\ell_{1}$. Hence, by analogy, it is very natural to ask whether the quasi-greedy bases in  $\ell_{p}$ for $p<1$ are democratic.  Note that, as of today, there is no known example of a conditional quasi-greedy basis for $\ell_p$ when $0<p<1$.\end{problem}

\begin{problem}The question on the existence of an almost greedy basis in locally convex mix-norm sequence spaces  $\ell_p\oplus \ell_q$, and matrix spaces $B_{p,q}$ and $\ell_q(\ell_p)$ was completely settled in \cite{DKK2003}. As for non-locally convex spaces, Propositions~\ref{prop:NQGBesov} and \ref{prop:NQGBesov+} seem to be the sole advances as of today on this question. So, we wonder if the following spaces have an almost greedy basis.
\begin{itemize}
\item $\ell_p(\ell_q)$ and  $\ell_q(\ell_p)$ if $0<p<1$ and $0<q \le \infty$ (with the usual modification if $q=\infty$). Note that, by Theorems~\ref{Nogreedylp+lq} and \ref{NogreedyZpq}, if such a basis exists, it is conditional.

\item $B_{p,q}$  if $0<p<1$ and $p<q<\infty$.
\end{itemize}
\end{problem}

\newpage

\section*{Annex}

\subsection*{Summary of the most commonly employed constants}

\begin{center}
\begin{table}[ht]
\centering
\resizebox{\textwidth}{!} {%
\begin{tabular}{c c c}
\hline
& & \\
\textbf{Symbol}& \textbf{Name of constant} & \textbf{Equation}\\
& & \\
\hline
& & \\
$A_{p}$ & Geometric constant & \eqref{eq:fieldconstant1} \\ 
& & \\
$B_{p}$ & Geometric constant & \eqref{eq:fieldconstant2} \\ 
& & \\
$\eta_{p}(\cdot)$ & Geometric function & \eqref{eq:function} \\ 
& & \\
$C_{g}$ & Greedy constant & \eqref{eq:greedy} \\ 
& & \\
$C_{ag}$ & Almost-greedy constant & \eqref{eq:ag} \\ 
& & \\
$C_{qg}$ & Quasi-greedy constant & \eqref{eq:qg} \\ 
& & \\
$C_{ql}$ & Quasi-greedy for largest coefficients constant & \eqref{eq:qg} \\ 
& & \\
$\Delta $ & Democracy constant & \eqref{Equation:Democracy} \\ 
& & \\
$\Delta_{d}$ & Disjoint-democracy constant & \eqref{Equation:Democracy} \\ 
& & \\
$\Delta_{s}$ & Superdemocracy constant & \eqref{Equation:Democracy} \\ 
& & \\
$\Delta_{sd}$ & Disjoint-superdemocracy constant & \eqref{Equation:Democracy} \\  
& & \\
$\Gamma$ & Symmetry for largest coefficients constant & \eqref{Equation:Democracy} \\     
& & \\
$\Delta_{b}$ & bidemocracy constant & \eqref{eq:bidem} \\ 
& & \\
$\Delta_{sb}$ & Bi-superdemocracy constant & \eqref{eq:bisuperdem} \\  
& & \\
$K_{u}$ & Lattice unconditional constant & \eqref{eq:lu} \\
& & \\
$K_{su}$ & Suppression unconditional constant & \eqref{eq:su} \\ 
& & \\
$K_{sc}$ & Suppression unconditional for constant coefficients constant & \eqref{eq:succ} \\ 
& & \\
$K_{lc}$ & Lower unconditional for constant coefficients constant & \eqref{eq:lucc} \\ 
& & \\
$K_{pu}$ & Partially lattice unconditional constant & \eqref{eq:lpu} \\ 
& & \\
$\Lambda_t$ & Truncation operator constant & \eqref{eq:truncation} \\ 
& & \\
$\Lambda_u$ & Restricted truncation operator constant & \eqref{eq:unnamed} \\ 
& &\\
\hline
\end{tabular}
}
\end{table}
\end{center}

\subsection*{Acronym List}

\begin{center}
\begin{table}[ht]
\centering
\resizebox{\textwidth}{!} {%
\begin{tabular}{c c c}
\hline
& & \\
\textbf{Acronym}& \textbf{Meaning} & \textbf{Place}\\
& & \\
\hline
& & \\
SUCC & suppression unconditional for constant coefficients &  page \pageref{eq:succ}\\
& & \\
LUCC& lower unconditional for constant coefficients & page \pageref{eq:lucc}\\
& & \\
LPU &lattice partially unconditional & page \pageref{eq:ucc7}\\
& & \\
QGLC & quasi-greedy for largest coefficients & page \pageref{PW:def_qglc}\\
& & \\
SLC & symmetric for largest coefficients & Section \ref{Sec6}\\
& & \\
URP& upper regularity property& \eqref{PW:def_URP}, page \pageref{PW:def_URP}\\
& & \\
LRP& lower regularity property& \eqref{PW:def_LRP}, page \pageref{PW:def_LRP}\\
& & \\
\hline
\end{tabular}
}
\end{table}
\end{center}


\end{document}